\documentclass[a4paper,11pt]{amsart}

\usepackage[matrix,arrow,tips,curve]{xy}
\usepackage[english]{babel}
\usepackage{amsmath}
\usepackage{amssymb}

\usepackage{mathrsfs}

\usepackage{enumerate}

\newcounter{long}

\usepackage{tikz}
\usetikzlibrary{shapes,arrows,positioning,snakes}

\newtheorem{thm}{Theorem}[section]
\newtheorem{lemma}[thm]{Lemma}
\newtheorem{corollary}[thm]{Corollary}
\newtheorem{proposition}[thm]{Proposition}

\newtheorem*{thm*}{Theorem}

\theoremstyle{definition}
\newtheorem{definition}[thm]{Definition}

\newtheorem{remark}[thm]{Remark}

\newtheorem{notation}[thm]{Notation}

\makeatletter
\let\c@equation\c@thm

\makeatother

\makeatletter
\renewcommand{\@secnumfont}{\bfseries}
\makeatother

\makeatletter
\let\c@subsection\c@thm
\def\thesubsection{\thethm}
\makeatother

\makeatletter
\let\c@figure\c@thm

\makeatother

\newcommand{\ph}{\varphi}
\newcommand{\w}{\widetilde}
\newcommand{\ma}{\mathcal}
\newcommand{\la}{\longrightarrow}
\newcommand{\ol}{\mathcal{O}}

\newcommand{\pr}{\mathbb{P}}
\newcommand{\Q}{\mathbb{Q}}
\newcommand{\C}{\mathbb{C}}
\newcommand{\R}{\mathbb{R}}
\newcommand{\Z}{\mathbb{Z}}
\newcommand{\N}{\mathcal{N}_1}
\newcommand{\Nu}{\mathcal{N}^1}

\newcommand{\Sing}{\operatorname{Sing}}
\newcommand{\dP}{\operatorname{dP}}
\newcommand{\Pic}{\operatorname{Pic}}

\newcommand{\im}{\operatorname{Im}}
\newcommand{\NE}{\operatorname{NE}}
\newcommand{\Exc}{\operatorname{Exc}}

\newcommand{\Ext}{\operatorname{Ext}}

\newcommand{\codim}{\operatorname{codim}}
\newcommand{\Eff}{\operatorname{Eff}}
\newcommand{\Nef}{\operatorname{Nef}}
\newcommand{\Mov}{\operatorname{Mov}}
\newcommand{\MCD}{\operatorname{MCD}}
\newcommand{\ST}{\operatorname{ST}}

\newcommand{\Aut}{\operatorname{Aut}}
\newcommand{\Bl}{\operatorname{Bl}}
\newcommand{\Bs}{\operatorname{Bs}}
\newcommand{\rk}{\operatorname{rk}}
\newcommand{\Id}{\operatorname{Id}}

\calclayout 

\newcounter{a}
\subjclass[2010]{14J60,14J35,14J45,14E30}

 \title[$\Bl_8\pr^4$, its Fano model, and vector bundles on a del Pezzo surface]{The blow-up of $\pr^4$ at $8$ points and its Fano model, via vector bundles on a del Pezzo surface}

\author[C.\ Casagrande]{Cinzia Casagrande}
\address{\sc Cinzia Casagrande\\
Universit\`a di Torino\\
Dipartimento di Matematica\\
via Carlo Alberto 10, 10123 Torino\\ Italy}
\email{cinzia.casagrande@unito.it}

\author[G.\ Codogni]{Giulio Codogni}
\address{\sc Giulio Codogni\\ EPFL\\ SB MATHGEOM CAG\\
MA B3 444 (Batiment MA)\\ Station 8\\
CH-1015 Lausanne\\ Switzerland}
\email{giulio.codogni@epfl.ch}

\author[A.\ Fanelli]{Andrea Fanelli}
\address{\sc Andrea Fanelli\\ Heinrich-Heine-Universit\"{a}t\\
Mathematisches Institut\\
D-40204 D\"{u}sseldorf, Germany }
\email{fanelli@hhu.de}
\begin{document}
\maketitle
{\footnotesize\tableofcontents}
\section{Introduction}
\noindent Let $S=\Bl_{q_1,\dotsc,q_8}\pr^2$ and $X=\Bl_{p_1,\dotsc,p_8}\pr^4$ be the blow-ups respectively of $\pr^2$ and $\pr^4$ at $8$ general points.
There is  a classical connection between these two varieties due to projective association, or Gale duality, which gives a bijection between sets of $8$ general points in $\pr^2$ and 
in $\pr^4$, up to projective equivalence (see \ref{introassociation}).
 In this framework, 
a beautiful relation among $S$ and $X$ has been established by Mukai, using moduli of sheaves on $S$, as follows.
\begin{thm}[\cite{mukaiADE}, \S 2]\label{Xintro}
If   $\{q_1,\dotsc,q_8\}\subset\pr^2$ 
and $\{p_1,\dotsc,p_8\}\subset\pr^4$ are associated sets of points, 
then
$X$ is isomorphic to the moduli space  of rank 2 torsion free sheaves $F$ on $S$, with $c_1(F)=-K_S$ and $c_2(F)=2$, semistable 
 with respect to $-K_S+2h$, where $h\in\Pic(S)$ is the pull-back of $\ol_{\pr^2}(1)$ under 
the blow-up map $S\to\pr^2$. 
\end{thm}
This result is the starting point of this paper, which has three main subjects: 
\begin{enumerate}[A.]
\item the  moduli spaces $M_{S,L}$ of rank 2 torsion free sheaves $F$ on a (smooth) degree one del Pezzo surface $S$, with $c_1(F)=-K_S$ and $c_2(F)=2$, semistable 
(in the sense of Gieseker-Maruyama) with respect to $L\in\Pic(S)$ ample; 
\item the smooth Fano $4$-fold $Y:=M_{S,-K_S}$;
\item the geometry of $X=\Bl_{p_1,\dotsc,p_8}\pr^4$.
\end{enumerate}
Mukai's proof of Th.~\ref{Xintro} is based on the study of the birational geometry of $M_{S,L}$ in terms of the variation of the stability condition given by $L$. In this paper we resume and expand Mukai's study of these moduli spaces, 
proving that the birational geometry of  $M_{S,L}$ is completely governed by the variation of stability conditions.
Then we apply this to study the Fano $4$-fold $Y$ and the blow-up $X$ of $\pr^4$ at $8$ general points.
Let us describe in more detail these three points, and present our main results.
\renewcommand{\thesubsection}{\Alph{a}}
\setcounter{subsection}{0}
\subsection{Moduli of vector bundles on a degree $1$ del Pezzo surface}
To describe the moduli spaces $M_{S,L}$,
we introduce two convex rational polyhedral cones 
$$\Pi\subset\ma{E}\subset \Nef(S)\subset H^2(S,\R).$$
See \S \ref{delpezzo}, in particular \ref{secE} and \ref{secPi}, for the explicit definitions; the cone $\Pi$ has been introduced in \cite[p.~8]{mukaiADE}.

Let us first state some general properties of  $M_{S,L}$. In the following proposition some statements are standard; the new part is the characterization of the polarizations $L$ for which the moduli space is non-empty. See Cor.~\ref{smoothness}, \ref{nonempty2},  \ref{connected}, and \ref{locallyfree}, and Rem.~\ref{stable}.
\begin{proposition}\label{propertiesintro}
Let $L\in\Pic(S)$ ample.
The moduli space $M_{S,L}$ is non-empty  if and only if $L\in\ma{E}$, and in this case $M_{S,L}$ is a smooth, projective, rational $4$-fold. Every sheaf parametrized by $M_{S,L}$ is locally free and stable.
\end{proposition}
Let us consider now the birational geometry of the moduli spaces  $M_{S,L}$. 
The relation between variation of polarization via wall-crossings, and birational geometry of moduli spaces of sheaves on surfaces, is 
classical and has been
intensively studied, see  for instance \cite{FriedmanQin,ellingsrudgottsche,matsukiwentworth}, and for the case of del Pezzo surfaces, \cite{costamiroroig,gomez} and the more recent \cite{BMW}. 
In our setting, this relation can be made 
 completely explicit. 
Generalising  \cite[Lemma 3]{mukaiADE}, we determine all the (finitely many) walls for slope semistability (Cor.~\ref{walls2}), and introduce the stability fan $\ST(S)$ in $H^2(S,\R)$, supported on the cone $\ma{E}$, determined by these walls (see Def.~\ref{wallsfan}). When the polarization $L$ varies in the interior of a cone of maximal dimension of the stability fan, the stability condition determined by $L$ is constant, and so is $M_{S,L}$.
When $L$ moves to a different cone in the stability fan, the moduli space $M_{S,L}$ undergoes a simple birational transformation. 
These results are presented in \S\ref{secmoduli}, and are mostly based on  Mukai's work \cite{mukaiADE}; see \ref{introsecmoduli} for  a more detailed overview.

The moduli spaces $M_{S,L}$ are Mori dream spaces (see \ref{sec_cones} and references therein for the notions of Mori dream space and of Mori chamber decomposition).
  This follows from Th.~\ref{Xintro} and  Castravet and Tevelev's result \cite[Th.~1.3]{CoxCT}, and  also 
from the log Fano property of $M_{S,L}$, see Cor.~\ref{MDS}. Thus the stability fan $\ST(S)$ in $H^2(S,\R)$ has a counterpart in $H^2(M_{S,L},\R)$, the fan $\MCD(M_{S,L})$ given by the Mori chamber decomposition,  defined via birational geometry.

In \S \ref{sec_determinant},
using the classical construction of determinant line bundles on the moduli space $M_{S,L}$, we define a group homomorphism
$$\rho\colon \Pic(S)\la\Pic(M_{S,L})$$
and study its properties, see \ref{introsec_determinant} for  a more detailed overview. 
The determinant map $\rho$
 provides the bridge between stability chambers in $H^2(S,\R)$ and cones of divisors in $H^2(M_{S,L},\R)$: in \S\ref{sec_blowup} we show the following.
\begin{thm}[see \ref{sec_cones}]\label{isointro}
Let  $L\in\Pic(S)$ ample, $L\in\Pi$. 
The  map $\rho\colon H^2(S,\R)\to H^2(M_{S,L},\R)$ is an isomorphism,
$\rho(\ma{E})$ is the cone of effective divisors $\Eff(M_{S,L})$,  and $\rho({\Pi})$ is the cone of movable divisors $\Mov(M_{S,L})$. 
Moreover,  $\rho$ yields an isomorphism between the stability fan $\ST(S)$ in $H^2(S,\R)$, and the fan $\MCD(M_{S,L})$ in  $H^2(M_{S,L},\R)$ given by the Mori chamber decomposition.
\end{thm}
This result relies on 
 the classical positivity properties of the determinant line bundle and on Th.~\ref{Xintro}.

We recall that a pseudo-isomorphism is a birational map which is an isomorphism in codimension one, and similarly we define a pseudo-automorphism.
When the polarization $L$ varies in the cone $\Pi$, we get finitely many pseudo-isomorphic moduli spaces $M_{S,L}$, related by sequences of flips. We show that in fact, when $L\in\Pi$, the moduli space $M_{S,L}$ determines the surface $S$; we addressed this question in analogy with  Bayer and Macr\`i's result \cite[Cor.~1.3]{bayermacri}
  on moduli of sheaves on K3 surfaces.
\begin{thm}[see \ref{sectorelli}]\label{pseudo}
Let $S_1$ and $S_2$ be del Pezzo surfaces of degree one, and $L_i\in\Pic(S_i)$ ample line bundles with $L_i\in\Pi_i\subset H^2(S_i,\R
)$, for $i=1,2$. Then $S_1\cong S_2$ if and only if $M_{S_1,L_1}$ and $M_{S_2,L_2}$ are pseudo-isomorphic.
\end{thm}
\noindent Note that the assumption that $L_i\in\Pi_i$ here is essential, because every degree one del Pezzo surface $S$ has a polarization $L_0$ such that $M_{S,L_0}\cong\pr^4$ (see Prop.~\ref{P^4}). 

We also describe the group of pseudo-automorphisms of $M_{S,L}$, when  $L\in\Pi$.
\begin{thm}[see \ref{secauto}]\label{psautM}
Let $L\in\Pic(S)$ ample, $L\in\Pi$. Then the group of pseudo-automorphisms of $M_{S,L}$ is isomorphic to the automorphism group $\Aut(S)$ of $S$, where $f\in\Aut(S)$ acts on $M_{S,L}$ as   $[F]\mapsto[(f^{-1})^*F]$. 
\end{thm}
\stepcounter{a}
\setcounter{subsection}{4}
\subsection{Geometry of the Fano model $Y$}
The anticanonical class $-K_S$ in $H^2(S,\R)$ belongs to the cone $\Pi$, and lies in the interior of a cone of the stability fan.
It follows again from the classical properties of the determinant line bundle that for the polarization $L=-K_S$, the moduli space $M_{S,L}$ is Fano. More precisely, we have the following.
\begin{proposition}[Prop.~\ref{fano} and \ref{numinv}, Lemma \ref{deform}]\label{numinvintro}
The moduli space $Y:=M_{S,-K_S}$ is a smooth, rational Fano $4$-fold with index one and  $b_2(Y)=9$, $b_3(Y)=0$, $h^{2,2}(Y)=b_4(Y)=45$, $(-K_Y)^4=13$,  $h^0(Y,-K_Y)=6$,  $h^0(Y,T_Y)=0$, and $h^1(Y,T_Y)=8$.
\end{proposition}
Let us notice that, except products of del Pezzo surfaces, there are very few known examples of Fano $4$-folds with $b_2\geq 7$. In particular, to the authors' knowledge, the family of Fano $4$-folds $Y$ is the only known example of Fano $4$-fold with $b_2\geq 9$ which is not a product of surfaces. 
It is a very interesting family, whose construction and 
study  was one of the motivations for this work.

By Th.~\ref{isointro}, the determinant map $\rho\colon H^2(S,\R)\to H^2(Y,\R)$ is an isomorphism and yields a   completely explicitly description of the relevant cones of divisors $\Eff(Y)$, $\Mov(Y)$, and $\Nef(Y)$, and more generally of the Mori chamber decomposition of $\Eff(Y)$. We give here a statement on the cone of effective curves, and refer the reader to \S \ref{fanosection} for the descriptions of the other relevant cones.
\begin{proposition}[see \ref{nefcone}]\label{NEintro}
The cone of effective curves $\NE(Y)$ is isomorphic to the cone of effective curves $\NE(S)$ of $S$, and it has $240$ extremal rays. Each extremal ray yields a small contraction\footnote{A \emph{contraction} is a surjective map with connected fibers $\ph\colon Y\to Z$, where $Z$ is normal and projective.} 
with exceptional locus a smooth rational surface. 

More generally, every contraction $\ph\colon Y\to Z$ with $\dim Z>0$ is birational with $\codim\Exc(\ph)\geq 2$, and  $\Nef(Y)\cap\partial\Mov(Y)=\{0\}$.
\end{proposition}
We also show that $S$ and $Y$ determine each other, and we determine $\Aut(Y)$.
\begin{thm}[see \ref{sectorelli}]\label{moduli}
 Let $S_1$ and $S_2$ be del Pezzo surfaces of degree $1$, and set $Y_i:=M_{S_i,-K_{S_i}}$ for $i=1,2$. Then $S_1\cong S_2$ if and only if $Y_1\cong Y_2$.
\end{thm}
\begin{thm}[see \ref{secauto}]\label{autY}
The map $\psi\colon\Aut(S)\to\Aut(Y)$ given by $\psi(f)[F]=[(f^{-1})^*F]$, for $f\in\Aut(S)$ and $[F]\in Y$, is a group isomorphism. In particular $\Aut(Y)$ is finite, and if $S$ is general,
then $\Aut(Y)=\{\Id_Y,\iota_Y\}$, where  $\iota_Y\colon Y\to Y$ is induced by the Bertini involution of $S$. 
\end{thm}
The description of the automorphism group of $Y$, and of its action on $H^2(Y,\R)$, is also used to show that $Y$ is \emph{fibre-like}, namely that it can appear as a fiber of a Mori fiber space, see \ref{MFS}.

Finally, motivated by the low values of $h^0(Y,-K_Y)$ and $(-K_Y)^4$ (see Prop.~\ref{numinvintro}), and also by the analogy with degree one del Pezzo surfaces, in \S \ref{anti}
we study the base loci of the anticanonical and bianticanonical linear systems of $Y$, and prove the following.
\begin{thm}[see \ref{anticanonical} and \ref{bianticanonical}]\label{system}
The linear system $|-K_Y|$ has a base locus of positive dimension, while
the linear system $|-2K_Y|$ is base point free.
\end{thm}
\stepcounter{a}
\setcounter{subsection}{9}
\subsection{The blow-up $X$ of $\pr^4$ at $8$ general points}
As we already recalled, association gives a bijection between (general) sets of $8$ points in $\pr^2$  and in $\pr^4$. This gives a natural correspondence between pairs $(S,h)$, where $S$ is a del Pezzo surface of degree one, and $h\in\Pic(S)$ defines a birational map $S\to\pr^2$, and blow-ups $X$ of $\pr^4$ at $8$ general points (see \ref{bijection}). The interplay between $S$, $X$, and $Y$ is the key point of this paper.
This also yields new results on the blow-up $X$ of $\pr^4$ at $8$ general points, which are  mostly treated in \S \ref{sezX}; let us give an overview.

First of all we describe explicitly
the relation among $X$ and $Y$: the Fano $4$-fold $Y$ is obtained from $X$ by flipping the transforms of the lines  in $\pr^4$ through $2$ blown-up points, and of the rational normal quartics through $7$  blown-up points (Lemma \ref{sequence}).

In particular, Th.~\ref{system} on the anticanonical and bianticanonical linear systems on $Y$ is proved using the birational map $X\dasharrow Y$ and studying the corresponding linear systems in $X$. We show that the base locus of $|-K_X|$ contains the transform $R$ of a smooth rational quintic curve in $\pr^4$ through the $8$ blown-up points, and that
the transform of $R$ in $Y$ is contained in the base locus of $|-K_Y|$
 (Cor.~\ref{conclusion} and Lemma \ref{images}, see also Rem.~\ref{macaulay}).
  
We also have the following direct consequence of Th.~\ref{Xintro} and \ref{pseudo}.
\begin{corollary}
Let $q_1^i,\dotsc,q_8^i\in\pr^2$ be such that $S_i:=\Bl_{q_1^i,\dotsc,q_8^i}\pr^2$ is a del Pezzo surface, for  $i=1,2$.
Let $p_1^i,\dotsc,p_8^i\in\pr^4$ be the associated points to $q_1^i,\dotsc,q_8^i\in\pr^2$, and  set $X_i:=\Bl_{p_1^i,\dotsc,p_8^i}\pr^4$, for $i=1,2$. Then $S_1\cong S_2$ if and only if $X_1$ and $X_2$ are pseudo-isomorphic.
\end{corollary}

The previous results also give a description of the group of pseudo-automorphisms of $X$; we show (Prop.~\ref{bertiniX}) that  $X$ has a unique non-trivial 
 pseudo-automorphism $\iota_X$, that 
we call 
the \emph{Bertini involution} of $X$.
Via the blow-up map $X\to\pr^4$, this also defines a birational involution $\iota_{\pr^4}\colon\pr^4\dasharrow\pr^4$.

The birational maps $\iota_X$ and $\iota_{\pr^4}$ are in fact classically known, as they can be   defined via the Cremona action of the Weyl group $W(E_8)$ on sets of $8$ points in $\pr^4$, see Dolgachev and Ortland \cite[Ch.~VI, \S 4, and p.~131]{dolgort} and  Du Val \cite[p.~199 and p.~201]{duval}. 
With the standard notation for divisors in $X$ (see \ref{notationP^4}), we have 
$$\iota_X^*H=49H-30\bigl(\sum_iE_i\bigr)$$
(see \cite[(11) on p.~199]{duval}), thus
 $\iota_{\pr^4}$ is defined by the linear system  $V\subset|\ol_{\pr^4}(49)|$ of hypersurfaces having multiplicity at least $30$ at $p_1,\dotsc,p_8$.
As noted in  \cite[p.~201]{duval} and \cite[p.~131]{dolgort}, the classical definitions of $\iota_X$ and $\iota_{\pr^4}$ do not give a geometrical description of these maps. Using the interpretation of $X$ as a moduli space of vector bundles on $S$, we give  a factorization of these maps as smooth blow-ups and blow-downs, see  Prop.~\ref{bertiniX}
and Cor.~\ref{bertiniP^4}.

Finally, as a direct application of Th.~\ref{Xintro} and \ref{isointro}, we  describe the fixed\footnote{A prime divisor $E$ is fixed if it is a fixed component of the linear system $|mE|$ for every $m\geq 1$.}  divisors of $X$ in terms of conics in $S$. It has been shown by Castravet and Tevelev \cite[Th.~2.7]{CoxCT} that
$\Eff(X)$  is generated by the classes of fixed divisors, which form an orbit under the action of the Weyl group $W(E_8)$ on $H^2(X,\Z)$ (see \ref{secisometry}). We get the following.
\begin{proposition}[see \ref{secfixeddiv}]\label{fixeddiv}
Let $X$ be the blow-up of $\pr^4$ at $8$ general points. Then the cone of effective divisors $\Eff(X)$ is generated by the classes of $2160$ fixed divisors, which are in bijection with the classes of conics in a del Pezzo surface $S$ of degree $1$. With the standard notation for divisors in $S$ and $X$ (see \ref{notationP^2} and \ref{notationP^4}), if $C\sim dh-\sum_{i}m_ie_i$ is such a conic, then the corresponding fixed divisor $E_C$ has class:
$$
E_C\sim \frac{1}{2}\bigl(\sum_{i}m_i-d\bigr)\bigl(H-\sum_{i}E_i\bigr)+\sum_{i}m_iE_i.
$$
\end{proposition}
\noindent We also give  generators for the semigroup of integral effective divisors of $X$ (Lemma \ref{integralgenerators}), by applying a result from \cite{CoxCT}.

We conclude by 
mentioning that our study of the blow-up $X$ and of its Fano model $Y$ is analogous to
the study in \cite{fanomodel} of the Fano model of 
the blow-up of $\pr^n$ in $n+3$ general points, for $n$ even. 

We work over the field of complex numbers.

\medskip

\noindent{\bf Acknowledgments. }
We would like to thank Paolo Cascini, Ana-Maria Castravet, Daniele Faenzi, Emanuele Macr\`i, John Ottem, Zsolt Patakfalvi, and Filippo Viviani for interesting discussions related to this work, 
 and the referees for useful comments.

The first-named author has been partially supported by the PRIN 2015 ``Geometria delle Variet\`a Algebriche''.
The second-named author has been supported by the FIRB 2012 ``Moduli spaces and their applications''.
The third-named author has been supported by the SNF grant ``Algebraic subgroups of the Cremona groups'' and the DFG grant ``Gromov-Witten Theorie, Geometrie und Darstellungen'' (PE 2165/1-2). The second-named and third-named authors are grateful to the University of Torino for the warm hospitality provided during part of the preparation of this work. 
\section{Preliminaries on del Pezzo surfaces of degree $1$ and association}\label{delpezzo}
\renewcommand{\thesubsection}{\arabic{section}.\arabic{subsection}}
\noindent Let  $S$ be a del Pezzo surface of degree $1$; we  always assume that $S$ is smooth.
In this section we collect the properties of $S$ that are needed in the sequel, fix the relevant notation and terminology, and introduce the convex rational polyhedral cones 
$$\ma{N}\subset\Pi\subset\ma{E}\subset\Nef(S)\subset H^2(S,\R)$$
that play a crucial role for the rest of the paper. We also recall the relation via association among $S$ and the blow-up $X$ of $\pr^4$ in $8$ points.

We refer the reader to \cite[Ch.~8]{dolgachevbook} for the classical properties of $S$.
 In particular, we recall that $H^2(S,\Z)$ has  a lattice structure given by the intersection form, and the sublattice $K_S^{\perp}$ is an $E_8$-lattice; we denote by $W_S\cong W(E_8)$ its Weyl group of automorphisms.

With a slight abuse of notation, we will often write $C\in H^2(S,\R)$ for the class of a curve $C\subset S$, and similarly for divisors in higher dimensional varieties.
\subsection{The cones $\NE(S)$ and $\Nef(S)$}
The cone of effective curves $\NE(S)\subset H^2(S,\R)$ is generated by the classes of the $240$ $(-1)$-curves, on which  $W_S$ acts transitively. 

A {\bf conic} $C$ on $S$ is a smooth rational curve  such that $-K_S\cdot C=2$ and $C^2=0$; every such conic yields a conic bundle $S\to\pr^1$ having $C$ as a fiber. There are $2160$ conics in $H^2(S,\Z)$, on which $W_S$ acts transitively \cite[\S 8.2.5]{dolgachevbook}.

We will repeatedly use the explicit description of $(-1)$-curves and conics of $S$ once a birational map $S\to\pr^2$ is fixed; see \cite[Prop.~8.2.19, \S 8.2.6, \S 8.8.1]{dolgachevbook}.

The dual cone of $\NE(S)$ is the cone of nef divisors $\Nef(S)$, which has two types of generators. 
 The first  are the conics, which lie on the boundary of $\NE(S)$, and correspond to conic bundles $S\to\pr^1$. 
The second type of generators are big and correspond to birational maps $\sigma\colon S\to\pr^2$, which realise $S$ as the blow-up of $\pr^2$ in $8$ distinct points; the corresponding generator of $\Nef(S)$ is $h:=\sigma^*\ol_{\pr^2}(1)$; we call such $h$ a {\bf   cubic}. 
There are  $17280$   cubics in $H^2(S,\Z)$, which form an orbit under the action of  $W_S$ (see \cite[\S 8.2.5, 8.2.6, 8.8.1]{dolgachevbook}). 
Summing-up we have: 
$$\Nef(S)=\langle C,h\,|\,C\text{ a conic and $h$ a   cubic}\rangle\subset H^2(S,\R),$$
where $\langle v_1,\dotsc,v_r\rangle$ denotes the convex cone generated by $v_1,\dotsc,v_r$ in a real vector space.
\subsection{Notation for $S\to\pr^2$}\label{notationP^2} Given a   cubic $h$, we use the following notation:
\begin{enumerate}[$\bullet$]
\item
 $\sigma\colon S\to \pr^2$ is the birational map defined by $h$ 
\item $q_1,\dotsc,q_8\in\pr^2$ are the points blown-up by $\sigma$
\item $e_i\subset S$ is the exceptional curve over $q_i$, for $i=1\dotsc,8$
\item $e:=e_1+\cdots+e_8$, so that $-K_S=3h-e$
\item $C_i\subset S$ is the  transform of a general line through $q_i$, so that $C_i\sim h-e_i$, for $i=1,\dotsc,8$
\item $\ell_{ij}\subset S$ is the transform of the line $\overline{q_iq_j}\subset\pr^2$, so that $\ell_{ij}\sim h-e_i-e_j$, for $1\leq i<j\leq 8$.
\end{enumerate}
\subsection{The cone $\ma{E}$}\label{secE}
We are interested in the subcone of $\Nef(S)$ generated  by the conics: 
$$
\ma{E}:=\langle C\,|\,C\text{ a conic}\rangle\subset H^2(S,\R).
$$ 
Since $\ma{E}\subsetneq\Nef(S)$, dually we have $\ma{E}^{\vee}\supsetneq\NE(S)$, for the dual cone $\ma{E}^{\vee}$ of $\ma{E}$. We have:
\begin{equation}\label{dualE}
\ma{E}^{\vee}=\langle \ell, 2h+K_S\,|\,\ell\text{ a $(-1)$-curve and $h$ a   cubic}\rangle.\end{equation}
Indeed, given a   cubic $h$, $(2h+K_S)^{\perp}\cap \ma{E}$ is a simplicial facet\footnote{A facet is a face of codimension one.} of $\ma{E}$, generated by the conics $C_i$ for $i=1,\dotsc,8$ (notation as in \ref{notationP^2}). On the other hand, given a $(-1)$-curve $\ell$, $\ell^{\perp}\cap\ma{E}$ is a non-simplicial facet of $\ma{E}$, generated by the $126$ conics disjoint from $\ell$. 

It follows from \eqref{dualE} that the cone $\ma{E}$ can equivalently be described as:
\begin{equation}\label{E}
\ma{E}=\{L\in \Nef(S)\,|\,L\cdot (2h+K_S)\geq 0\text{ for every   cubic }h\}.
\end{equation}
\begin{remark}\label{boundaryample}
Every $L\in\Pic(S)$ contained in the interior of $\ma{E}$ is ample.
If $L\in\Pic(S)$ is contained in the boundary of $\ma{E}$, then $L$ is ample if and only if $L$ is in the relative interior of a facet $(2h+K_S)^{\perp}\cap \ma{E}$, where $h$ is a   cubic. 
\end{remark}
\subsection{The cone $\ma{N}$}\label{sec_N} We set:
$$
\ma{N}=\{L\in H^2(S,\R)\,|\,L\cdot (2\ell+K_S)\geq 0\text{ for every $(-1)$-curve }\ell\},$$
equivalently $\ma{N}$ is defined via its dual cone:
$$\ma{N}^{\vee}=\langle 2\ell+K_S\,|\,\ell\text{ a $(-1)$-curve}\rangle.$$
The cone $\ma{N}^{\vee}$ has $240$ extremal rays, and is isomorphic to $\NE(S)$ via the automorphism of $H^2(S,\R)$ given by $\gamma\mapsto \gamma+(\gamma\cdot K_S)K_S$. This is a self-adjoint map and coincides with its transpose. Dually, $\Nef(S)$ is isomorphic to  $\ma{N}$  via the same linear map.
This gives a description of  the generators of $\ma{N}$:
\begin{equation}\label{N}
\ma{N}=\langle -2K_S+C,-3K_S+h\,|\,C\text{ a conic and $h$ a   cubic}\rangle.
\end{equation}
\begin{lemma}\label{nef}
The cone $\ma{N}^{\vee}$ contains all $(-1)$-curves; equivalently, every $L\in\ma{N}$ is nef.
\end{lemma}
\begin{proof}
 Let $\ell$ be a $(-1)$-curve, and $h$ a   cubic such that $\ell=e_1$ (notation as in \ref{notationP^2}). Let $\ell'$ be the $(-1)$-curve such that
$\ell'\sim 3h-2e_2-e_3-\cdots-e_8$.
Then $\ell\in\ma{N}^{\vee}$ because:
\begin{equation*}
\ell=e_1\sim e_2+\ell'+K_S=\frac{1}{2}\left(2e_2+K_S+2\ell'+K_S\right).
\qedhere\end{equation*}
\end{proof}
\subsection{The cone $\Pi$}\label{secPi}
We will also consider the following cone, defined in \cite[p.~8]{mukaiADE}:
$$\Pi=\{L\in \Nef(S)\,|\,L\cdot (2C+K_S)\geq 0\text{ for every conic }C\},$$
equivalently $\Pi$ is defined via its dual cone:
$$\Pi^{\vee}=\langle \ell, 2C+K_S\,|\,\ell\text{ is a $(-1)$-curve and $C$ is a conic}\rangle.$$
\begin{lemma}
We have: $\ma{N}\subset\Pi\subset\ma{E}\subset\Nef(S)$.
\end{lemma}
\begin{proof}
We show the dual inclusions $\ma{N}^{\vee}\supset\Pi^{\vee}\supset\ma{E}^{\vee}$. It
is easy to see that for every   cubic $h$, $2h+K_S\in\Pi^{\vee}$,
hence $\ma{E}^{\vee}\subset \Pi^{\vee}$ by \eqref{dualE}.
Similarly, given a conic $C$, it is easy to see that $2C+K_S\in\ma{N}^{\vee}$, so that $\Pi^{\vee}\subset\ma{N}^{\vee}$ by Lemma \ref{nef}.
\end{proof}
In \ref{movable} we will show that the one-dimensional faces of $\Pi$ contained in the interior of $\ma{E}$ are generated by  $-K_S+3h$, where
$h$ is
a    cubic.
\subsection{The Bertini involution}\label{bertiniS}
Let  $\iota_S\colon S\to S$ be the Bertini involution (see \cite[\S 8.8.2]{dolgachevbook}); for $S$ general, 
 $\iota_S$ is the unique non-trivial automorphism of $S$.
The pull-back $\iota_S^*$ acts on $\Pic(S)$ (and on $H^2(S,\R)$) by  fixing
 $K_S$ and acting as $-1$ on $K_S^{\perp}$. This yields:
\begin{equation}
\label{w_0}
\iota_S^*\gamma=2(\gamma\cdot K_S)K_S-\gamma\quad\text{ for every $\gamma\in H^2(S,\R)$.}
\end{equation}
\subsection{Other preliminary elementary properties of $S$}
\begin{remark}\label{positions}
Let $\ell,\ell'$ be $(-1)$-curves in $S$.
\begin{enumerate}[$(a)$]
\item
If $\ell\cdot\ell'=0$, then $\ell\cap\ell'=\emptyset$, and there exists a   cubic $h$  such that $h\cdot \ell=h\cdot\ell'=0$. 
\item
If $\ell\cdot\ell'=2$, then there exist (notation as in \ref{notationP^2}):
\begin{enumerate}[--]
\item  a   cubic $h$ such that $\ell=\ell_{12}\sim h-e_1-e_2$ and $\ell'\sim 2h-e_4-\cdots-e_8$
\item a   cubic $h'$ such that $\ell=e'_1$ and $\ell'\sim 3h'-2e_1'-e_2'-\cdots-e_7'$.
\end{enumerate}  
\item If $\ell\cdot\ell'\geq 3$, then $\ell\cdot\ell'= 3$ and $\ell'=\iota^*_S\ell\sim -2K_S-\ell$.
\end{enumerate}
\end{remark}
\begin{lemma}\label{cubics}
Let $h,h'$ be   cubics in $S$. Then $h\cdot h'\leq 17$, and equality holds if and only if $h'=\iota_S^*h\sim -6K_S-h$.
\end{lemma}
\begin{proof}
Consider the   cubic $h$;  notation as in \ref{notationP^2}.
We have
$h'\sim mh-\sum_{i} a_ie_i$
where $m,a_i\in\Z$ and $m=h\cdot h'$. Since $(h')^2=1$ and $-K_S\cdot h'=3$, we get
$$\sum a_i^2=m^2-1\quad\text{and}\quad \sum a_i=3(m-1).$$
Now the inequality $(\sum_{i}a_i)^2\leq 8\sum_{i}a_i^2$ (given by Cauchy-Schwarz  applied to $(a_1,\dotsc,a_8)$ and $(1,\dotsc,1)$)
yields $(m-1)(m-17)\leq 0$, hence $m\leq 17$.  If $m=17$, then $a_1=\cdots=a_8=6$, thus $h'\sim 17h-6(e_1+\cdots+e_8)=-6K_S-h=\iota_S^*h$ (see \eqref{w_0}).
\end{proof}
\begin{remark}\label{easy}
Let $L\in\Pic(S)$ be nef and such that $-K_S\cdot L=2$. Then 
 $L$ is one of the following classes: $\{-2K_S,C,-K_S+\ell\,|\,\text{$C$ a conic, $\ell$ a $(-1)$-curve}\}$.

Indeed by vanishing and Riemann-Roch we have $h^0(S,L)>0$. The semigroup of effective divisors of $S$ is generated by $(-1)$-curves and $-K_S$ \cite[Cor.~3.3]{batyrevpopov}, and since $-K_S\cdot L=2$, then $L$ is either  $-2K_S$, $-K_S+\ell$, or $\ell+\ell'$. In this last case, since $L$ is nef, we must have $\ell\cdot\ell'>0$. If $\ell\cdot\ell'=1$, then $\ell+\ell'$ is a conic, and if  $\ell\cdot\ell'\geq 3$, then $\ell+\ell'\sim -2K_S$ (see Rem.~\ref{positions}$(c)$). 
Finally, if $\ell\cdot\ell'=2$, then by Rem.~\ref{positions}$(b)$
there exists a   cubic $h$ 
such that $\ell+\ell'\sim 3h-e_1-e_2-e_4-\cdots-e_8\sim -K_S+e_3$ (notation as in \ref{notationP^2}).
\end{remark}
\subsection{Association}\label{introassociation}
We refer the reader to \cite{dolgort,GaleEP} for the definition and main properties of association, or Gale duality; here we just give a brief outline.

Consider the natural action of $\Aut(\pr^2)$ on $(\pr^2)^8$, and similarly of 
 $\Aut(\pr^4)$ on $(\pr^4)^8$. In both cases, every semistable element is also stable \cite[Ch.~II, Cor.\ on p.~25]{dolgort}.
Let us consider the GIT quotients $P^8_2:=((\pr^2)^8)^s/\Aut(\pr^2)$ and 
$P^8_4:=((\pr^4)^8)^s/\Aut(\pr^4)$.

  Association is an algebraic construction which  yields an isomorphism $a\colon P^8_2\cong P^8_4$
 \cite[Ch.~III, Cor.\ on p.~36]{dolgort}. In particular,
to every stable ordered set of $8$ points in $\pr^2$, we associate a stable ordered set of $8$ points in $\pr^4$, unique up to projective equivalence, and viceversa. Moreover, the same bijection can be given for non-ordered sets of points \cite[Ch.~III, \S 1]{dolgort}. We also need the following.
\begin{lemma}\label{generallinear}
Let $q_1,\dotsc,q_8\in\pr^2$ be in general linear position, and let $p_1,\dotsc,p_8\in\pr^4$ be the associated points. Then $p_1,\dotsc,p_8$ are in general linear position.
\end{lemma}
\begin{proof}
Let $A$ be a $3\times 8$ matrix with columns the coordinates of the points $q_i$'s, and similarly let $B$ be a $5\times 8$ matrix containing the coordinates of the points $p_j$'s; by the definition of association we have $A \cdot B^t=0$. Since $q_1,\dotsc,q_8\in\pr^2$ are in general linear position, 
every maximal minor of $A$ is non-zero.
For $I\subset\{1,\dotsc,8\}$ with $|I|=3$, let $a_I\in\C$ be the minor of $A$ given by the columns in $I$, and $b_I$ the minor of $B$ given by the columns \emph{not} in $I$. 

Let $I,J\subset \{1,\dotsc,8\}$ be such that  $|I|=|J|=3$ and $|I\cap J|=2$. It is shown in \cite[Ch.~III, Lemma 1]{dolgort} that $a_{I}b_{J}+a_{J}b_{I}=0$, thus $b_{I}=0$ if and only if $b_{J}=0$. Since  by construction $B$ has maximal rank, this shows that every maximal minor of $B$ is non-zero, hence the points $p_1,\dotsc,p_8$ are in general linear position.
\end{proof}
\begin{remark}\label{OK}
Let $q_1,\dotsc,q_8\in\pr^2$ be such that the blow-up of $\pr^2$ at  $q_1,\dotsc,q_8$ is a smooth del Pezzo surface (see \cite[Prop.~8.1.25]{dolgachevbook}). In 
 particular the $q_i$'s are in general linear position, and hence stable
\cite[Ch.~II, Th.~1]{dolgort}. This yields an open subset  $U_{\dP}\subset P^8_2$.
If $(p_1,\dotsc,p_8)\in a(U_{\dP})$, then 
 $p_1,\dotsc,p_8\in\pr^4$ are in general linear position by Lemma \ref{generallinear}.
\end{remark}
\subsection{Degree one del Pezzo surfaces and blow-ups of $\pr^4$ in $8$ points}\label{bijection}
Let $S$ be a del Pezzo surface of degree $1$, and
$h$ a   cubic in $S$. We associate to $(S,h)$ a blow-up $X$ of $\pr^4$ in $8$ points in  general linear position, as follows.

Let $q_1,\dotsc,q_8\in\pr^2$ be the points blown-up under the birational morphism 
$S\to\pr^2$ defined by $h$, and let $p_1,\dotsc,p_8\in\pr^4$ be the  associated points to $q_1,\dotsc,q_8\in\pr^2$ (which are in  general linear position by Rem.~\ref{OK}). Then we set $$X=X_h=X_{(S,h)}:=\Bl_{p_1,\dotsc,p_8}\pr^4.$$ 
We will always assume that $q_1,\dotsc,q_8\in\pr^2$ and  $p_1,\dotsc,p_8\in\pr^4$ are associated as \emph{ordered} sets of point.

Conversely, let $X$ be a blow-up of $\pr^4$ in $8$ general
points. Differently from the case of surfaces, the blow-up map $X\to\pr^4$ is unique, see \cite[p.~64]{dolgort}. Thus $X$ determines  $p_1,\dotsc,p_8\in\pr^4$ up to projective equivalence, which in turn determine 
$q_1,\dotsc,q_8\in\pr^2$  up to projective equivalence, and hence a pair $(S,h)$, such that $X\cong X_{(S,h)}$. The pair $(S,h)$ is unique up to isomorphism, therefore $S$ is determined up to isomorphism, and $h$ is determined up to
the action of $\Aut(S)$ on   cubics; in particular $X_{(S,h)}=X_{(S,\iota_S^*h)}$.
\subsection{Notation for the blow-up $X$ of $\pr^4$ at $8$ points}\label{notationP^4}
Let $p_1,\dotsc,p_8\in\pr^4$ be points in general linear position, and set $X:=\Bl_{p_1,\dotsc,p_8}\pr^4$. We use the following notation:
\begin{enumerate}[$\bullet$]
\item $E_i\subset X$ is the exceptional divisor over $p_i\in\pr^4$, for $i=1,\dotsc,8$
\item $H\in\Pic(X)$ is the pull-back of $\ol_{\pr^4}(1)$
\item $L_{ij}\subset X$ is the transform of the line $\overline{p_ip_j}\subset\pr^4$,  for $1\leq i<j\leq 8$
\item $e_i\subset E_i$ is a line, for $i=1,\dotsc,8$
\item $h\subset X$ is the transform of a general line in $\pr^4$
\item $\gamma_i\subset\pr^4$ is the rational normal quartic  through $p_1,\dotsc,\check{p}_i,\dotsc,p_8$, for $i=1,\dotsc,8$ ($\gamma_i$ exists and is unique, see for instance \cite[p.~14]{harris})
\item $\Gamma_i\subset X$ is the transform of  $\gamma_i\subset\pr^4$, for $i=1,\dotsc,8$.
\end{enumerate}
The notation $h,e_1,\dotsc,e_8$ is standard and will be used both in $S$ and in $X$ (see \ref{notationP^2}); it will be clear from the context whether we are referring to classes in $S$ or in $X$.
\section{Moduli of rank $2$ vector bundles on $S$ with $c_1=-K_S$ and $c_2=2$: non-emptyness, walls, special loci}\label{secmoduli}
\subsection{}\label{introsecmoduli} Let $S$ be a del Pezzo surface of degree $1$, and $L\in\Pic(S)$ an ample line bundle. Following \cite{mukaiADE}, 
in this section we introduce the moduli space $M_{S,L}$ of rank 2 torsion free sheaves $F$ on $S$, with $c_1(F)=-K_S$ and $c_2(F)=2$, semistable 
 with respect to $L$. We resume and expand the study made in \cite{mukaiADE} of this moduli space.

More precisely,
we determine explicitly all the walls for slope semistability, and introduce the stability fan $\ST(S)$ in $H^2(S,\R)$ determined by these walls.
We
describe the birational transformation occurring in $M_{S,L}$ when the polarization $L$ crosses a wall, and describe $M_{S,L}$ when $L$ belongs to some wall.
 Finally, for every   cubic $h$, we construct a chamber $\ma{C}_h$ such that for $L\in\ma{C}_h$, $M_{S,L}\cong\pr^4$. Using these results,
 we show Prop.~\ref{propertiesintro}.

Many results in this section are due to Mukai \cite{mukaiADE}.
Our new contributions are: the determination of the cone $\ma{E}$ of the polarizations for which the moduli space is non-empty (Cor.~\ref{nonempty2}), the completion of the description of the walls for slope semistability (Prop.~\ref{walls} and Cor.~\ref{walls2}), the description of the moduli space when the polarization is not in a chamber (Lemma \ref{specialL}), and the description of the exceptional locus of the morphism $\gamma\colon M_L\to M_L^{\mu}$ to the moduli space of slope semistable sheaves (Lemma \ref{slope}).
\subsection{The moduli space $M_{S,L}=M_L$} Let   $L\in\Pic(S)$ be ample,
and $F$ a rank $2$ torsion free sheaf on $S$.
By ``stable'' and ``semistable'' we mean stable or semistable in the sense of Gieseker-Maruyama; we will use $\mu$-stable or $\mu$-semistable for slope stability. We refer the reader to Huybrechts and Lehn's book \cite[\S 1.2]{huybrechtslehn} for these notions,
and recall that: \begin{center}
$\mu$-stable$\ \Rightarrow\ $ stable $\ \Rightarrow\ $ semistable
$\ \Rightarrow\ $ $\mu$-semistable.
\end{center}
 In particular, for a fixed polarization $L$, either the $4$ notions of stability and semistability above coincide, or there exists a strictly $\mu$-semistable sheaf.
\begin{definition}\label{M_L}
Given 
 $L\in\Pic(S)$ ample, $M_{S,L}$ is
the  moduli space  of torsion-free sheaves $F$ of rank $2$ on $S$, with $c_1(F)=-K_S$ and $c_2(F)=2$, semistable with respect to $L$. 
When
 the surface $S$ is fixed,  we will often write $M_L$
for $M_{S,L}$.
\end{definition}
\begin{remark}\label{stable}
Let $L\in\Pic(S)$ be ample, and let $F$ be a rank $2$ torsion-free sheaf  with $c_1(F)=-K_S$ and $c_2(F)=2$. Then either $F$ is stable, or $F$ is not semistable. Indeed by Riemann-Roch we have $\chi(S,F)=1$.
\end{remark}
The following is a standard application of Rem.~\ref{stable}, see for instance \cite[3.1]{BMW} and  \cite[Th.~4.5.4 and p.~115]{huybrechtslehn}.
\begin{corollary}\label{smoothness}
Let $L\in\Pic(S)$ be ample. If the moduli space $M_L$ is non-empty, then it is smooth, projective, of pure dimension $4$. 
\end{corollary}
 We will see in Cor.~\ref{connected} that $M_L$ is always irreducible; this is already known, see \cite[Prop.~3.11]{costamiroroig} and \cite[Th.~III]{gomez}.

Recall that in the ample cone of $S$, the polarizations $L$ for which there exists  a strictly $\mu$-semistable sheaf belong to a (locally finite) set of rational hyperplanes ({\bf walls}), which yield 
 a chamber decomposition of the ample cone (where a {\bf chamber} is a connected component of the complement of the walls in the ample cone). 
For any chamber $\mathcal{C}$ we have:
\begin{enumerate}[(1)]
\item for $L\in\ma{C}$, every $\mu$-semistable sheaf is also stable and $\mu$-stable;
\item the stability condition is the same for every $L\in\ma{C}$, and for $L\in\ma{C}$
the moduli spaces $M_L$ are all equal, hence they only depend on the chamber $\ma{C}$. 
\end{enumerate}
We will sometimes denote by $M_{\ma{C}}$ or $M_{S,\ma{C}}$ the moduli space $M_L$ for $L\in\ma{C}$.

First of all, Mukai  gives a necessary condition on the polarization $L$ for the existence of $\mu$-semistable sheaves with respect to $L$.
\begin{lemma}[\cite{mukaiADE}, p.~9]\label{nonempty}
Let $L\in\Pic(S)$ be ample. If there exists a $\mu$-semistable torsion-free sheaf $F$  of rank $2$ with $c_1(F)=-K_S$ and $c_2(F)=2$,   then $L\cdot (2h+K_S)\geq 0$ for every   cubic $h$, namely: $L\in\ma{E}$ (see \eqref{E}).
\end{lemma}
\subsection{Walls and special extensions}
 \cite[Lemma 3]{mukaiADE} describes all the walls that intersect the cone $\Pi$; we generalise it and describe every wall.
\begin{proposition}\label{walls}
Let $L\in\Pic(S)$ be ample, and suppose that there exists a strictly $\mu$-semistable  torsion-free sheaf $F$ of rank $2$ with $c_1(F)=-K_S$ and $c_2(F)=2$.
Then we have the following.
\begin{enumerate}[$(a)$]
\item
There exists a divisor $D$ such that 
$L\cdot (2D+K_S)=0$, and either  $D$ or $-K_S-D$ is linearly equivalent to  a
$(-1)$-curve, a conic, or a   cubic;
\item
 $F$ is locally free and there is
an exact sequence:
\begin{equation}\label{ext}
0\la\ol_S(D)\la F\la\ol_S(-K_S-D)\la 0.
\end{equation}
\item
If $D$ is effective, or if the extension is split, then $F$ is not stable.\\
If $-K_S-D$ is effective, and the extension is not split, then $F$ is stable.
\end{enumerate}
\end{proposition} 
For the proof of Prop.~\ref{walls} we need the following elementary result.
\begin{lemma}\label{int}
Let $L\in\Pic(S)$ be ample, $L\in\ma{E}$, and let $B$ be an effective divisor on $S$ such that:
$$B\cdot(K_S+B)=-2,\quad h^0(S,K_S+B)=0,\ \text{ and }\ L\cdot(2B+K_S)=0. $$
Then $B$ is linearly equivalent to either a $(-1)$-curve, or a conic, or a   cubic.
\end{lemma}
\begin{proof}
We can write $B\sim P+N$ where $P$ and $N$ are integral divisors, $P$ is nef, $P\cdot N=0$, and $N=\sum_{i=1}^rm_i\ell_i$ with $m_i\in\Z_{>0}$ and the $\ell_i$'s pairwise disjoint $(-1)$-curves (see \cite[Ex.~3.3]{BPS}).

If $P=0$, we have $B=m_1\ell_1+\cdots+m_r\ell_r$ and $-2=B\cdot(K_S+B)=-\sum_{i=1}^rm_i(m_i+1)$, hence $r=1$, $m_1=1$, and 
 $B$ is a $(-1)$-curve.

Suppose that $B$ has Iitaka dimension $1$, and let $S\to\pr^1$ be the conic bundle given by $P$, so that $P=m_0C$ where $C$ is a general fiber and $m_0\geq 1$. Then we have
$$-2=(m_0C+N)\cdot(K_S+m_0C+N)=
-\Bigl(2m_0+\sum_{i=1}^rm_i(m_i+1)\Bigr),$$
which yields $m_0=1$ and $r=0$, namely $B$ is linearly equivalent to a conic.

Finally, suppose that $B$ is big. We have $L\cdot(2h+K_S)\geq 0$ for every   cubic $h$, because $L\in\ma{E}$ (see \eqref{E}).
Since  $L\cdot(2B+K_S)=0$, if there exists some   cubic $h$ such that $B-h$ is effective, then $B=h$.

Since $P$ is nef and big, by vanishing we have $h^i(S,K_S+P)=0$ for $i=1,2$. Using the assumptions and Riemann-Roch, one gets $P\cdot(K_S+P)=-2$ and $N\cdot (K_S+N)=0$. This yields that $N=0$ and $B\sim P$ is nef and big, and moreover  using again Riemann-Roch and vanishing one gets
$h^0(S,B)=2+B^2$.
If $\sigma\colon S\to S'$ is the birational map given by $B$, we have $B=\sigma^*B'$ where $B'$ is ample and again $h^0(S',B')=2+(B')^2$, namely
the pair  $(S',B')$ has $\Delta$-genus zero (see \cite[\S 3.1]{beltrametti_sommese}). These pairs have been classified by Fujita, see for instance \cite[Prop.~3.1.2]{beltrametti_sommese}, and the possibilities are: $(\pr^2,\ol_{\pr^2}(1))$, $(\pr^2,\ol_{\pr^2}(2))$, $(\pr^1\times\pr^1,\ma{O}(1,a))$ with $a\geq 1$,  and $(\mathbb{F}_1,L)$ where $L$ is ample and $L_{|F}\cong\ol_{\pr^1}(1)$ for a fiber $F$ of the $\pr^1$-bundle on $\mathbb{F}_1$. In each case one can find a   cubic $h$ on $S$ such that $B-h$ is effective, so that $B=h$.
\end{proof}
\begin{proof}[Proof of Prop.~\ref{walls}]
Since $F$ is not $\mu$-stable with respect to $L$, there is an exact sequence
$$0\la\ol_S(D)\otimes\ma{I}_{Z_1}\la F\la\ol_S(-K_S-D)\otimes\ma{I}_{Z_2}\la 0$$
where $Z_1$ and $Z_2$ are $0$-dimensional subschemes of $S$ (see \cite[Ex.~1.1.16]{huybrechtslehn}) and $D$ is a divisor on $S$ with 
$L\cdot D=\frac{1}{2}L\cdot (-K_S)$, namely
$L\cdot (2D+K_S)=0$.
Using that $c_2(F)=2$ we get
$D\cdot (K_S+D)=l(Z_1)+l(Z_2)-2$,
where $l(Z_i)$ is the length of $Z_i$.
Since both $L$ and $-K_S$ are ample, we have $L\cdot(-K_S)>0$, and hence $L\cdot D>0$ and $L\cdot (D+K_S)<0$. This yields $h^0(S,K_S+D)=0$ and $h^2(S,K_S+D)=h^0(S,-D)=0$, hence $\chi(S,\ol_S(K_S+D))=-h^1(S,K_S+D)\leq 0$. On the other hand by Riemann-Roch:
$$\chi(S,\ol_S(K_S+D))=1+\frac{1}{2}D\cdot (K_S+D)=\frac{1}{2}\big(l(Z_1)+l(Z_2)
\big)\geq 0,$$
which yields $Z_1=Z_2=\emptyset$, $D\cdot (K_S+D)=-2$, and the exact sequence:
$$0\la\ol_S(D)\la F\la\ol_S(-K_S-D)\la 0;$$
in particular $F$ is locally free, and we have $(b)$.

We have $L\in\ma{E}$ by Lemma \ref{nonempty}.
By \cite[Lemma 2]{mukaiADE} the sheaf $F$ has a non-zero global section; thus by the sequence above, either $\ol_S(D)$ or $\ol_S(-K_S-D)$ must have a non-zero global section.
Now if $D$ is effective, then $D$ is either a $(-1)$-curve, or a conic, or a   cubic, by Lemma \ref{int}. Similarly, if $-K_S-D$ is  effective, then it is either a $(-1)$-curve, or a conic, or a   cubic, again by Lemma \ref{int}. This shows $(a)$.

For $(c)$, notice that if $\ell$ is a $(-1)$-curve, $C$ a conic, and $h$ a   cubic, we have:
$$ \chi(S,F)=1,\quad \chi(S,\ol_S(\ell))=1,\quad  \chi(S,\ol_S(C))=2,\quad  \chi(S,\ol_S(h))=3.$$

If $D$ is linearly equivalent to $\ell$, $C$, or $h$, then $\chi(S,\ol_S(D))>\frac{1}{2}\chi(S,F)=\frac{1}{2}$, thus $F$ is not stable.

Suppose that $-K_S-D$ is linearly equivalent to $\ell$, $C$, or $h$. If the extension is split, then we can exchange $D$ with $-K_S-D$, so that again $F$ is not stable. 
If instead the extension is not split, then  $\ol_S(D)$ is the unique locally free rank $1$ subsheaf $G$ of $F$ with $\mu_L(G)=\mu_L(F)$, where $\mu_L$ denotes the slope with respect to $L$
(see \cite[Ch.~4, Prop.~21]{friedmanbook}).  Since $\chi(S,\ol_S(D))=\chi(S,F)-\chi(S,\ol_S(-K_S-D))\in\{0,-1,-2\}$, we deduce that $F$ is stable.
\end{proof}
\begin{corollary}\label{walls2}
 Let $L\in\Pic(S)$ be ample, and suppose that there exists a strictly $\mu$-semistable  torsion-free sheaf $F$ of rank $2$ with $c_1(F)=-K_S$ and $c_2(F)=2$. 
Then $L$ belongs to a wall $(2\ell+K_S)^{\perp}$, $(2C+K_S)^{\perp}$, or $(2h+K_S)^{\perp}$, where $\ell$ is a $(-1)$-curve, $C$ is a conic, and $h$ is a   cubic.
\end{corollary}
\begin{definition}[the stability fan]\label{wallsfan}
 The hyperplanes  $(2h+K_S)^{\perp}$, $(2\ell+K_S)^{\perp}$, and $(2C+K_S)^{\perp}$
define a fan $\ST(S)$ in $H^2(S,\R)$, supported on the cone $\ma{E}$, that we call the {\bf stability fan}. 
The cones of maximal dimension of the fan are the closures $\overline{\mathcal{C}}$, where $\ma{C}\subset\ma{E}$ is a chamber. The cone  $\Pi$ is a union of cones of the fan, and the cone $\ma{N}$ belongs to the fan (see \ref{secPi} and \ref{sec_N}).
Notice that the hyperplanes  $(2h+K_S)^{\perp}$ cut the boundary of the cone $\ma{E}$, so that they do not separate different chambers for which the moduli space is non-empty (see Cor.~\ref{nonempty}).
\end{definition}
\begin{remark}\label{fanochamber}
The anticanonical class $-K_S$ does not belong to any wall and lies in the interior of the cone $\ma{N}$ (in particular $-K_S\in\Pi$ and $-K_S\in\ma{E}$), so that $\ma{N}$ is the closure of the chamber containing $-K_S$.
\end{remark}
We now study more in detail the sheaves arising from extensions as in \eqref{ext}.
\begin{remark}\label{dimext}
Given a $(-1)$-curve $\ell$, a conic $C$, and a   cubic $h$, we have:
\begin{align*}
h^1(S,K_S+2\ell)&=2, \quad &h^1(S,K_S+2C)=1, \qquad  &h^1(S,K_S+2h)=0,\\
h^1(S,-K_S-2\ell)&=3, \quad &h^1(S,-K_S-2C)=4, \qquad &h^1(S,-K_S-2h)=5.
\end{align*}
\end{remark}
\begin{lemma}\label{special}
Let $F$ be a  sheaf  on $S$  sitting in an extension
$$0\la\ol_S(D)\la F\stackrel{\ph}{\la}\ol_S(-K_S-D)\la 0$$
where $D$ or $-K_S-D$ is either a $(-1)$-curve, or a conic, or a   cubic.
We have:
\begin{enumerate}[$(a)$]
\item
$F$ is locally free with rank $2$, $c_1(F)=-K_S$, and $c_2(F)=2$;
\item
$\ph$ is unique up to non-zero scalar multiplication;
\item
either the extension is split, or $F$ is determined, up to isomorphism, by an element of $$\pr\big(\Ext^1(\ol_S(-K_S-D),\ol_S(D))\big)=\pr\big(H^1(S,\ol_S(K_S+2D))\big)$$
(here $\pr$ stands for the classical projectivization, namely that of $1$-dimensional linear subspaces);
\item
if $D$ is a   cubic, then the extension is split.
\end{enumerate}
\end{lemma}
\begin{proof}
Statement $(a)$ is straightforward. For $(b)$, let  $\ph'\colon F\to\ol_S(-K_S-D)$ be another surjective map. Then the restriction $\ph'_{|\ker\ph}\colon \ol_S(D)\to\ol_S(-K_S-D)$ must be trivial, because
 in all six possible cases for $D$, $-K_S-2D$ cannot be effective. Similarly, $\ph_{|\ker\ph'}=0$, thus $\ker\ph=\ker\ph'$ and hence $\ph'=\lambda\ph$. Statement $(c)$ follows from $(b)$, and $(d)$ holds because $h^1(S,K_S+2h)=0$ (see Rem.~\ref{dimext}).
\end{proof}
\begin{definition}[{\bf special loci in the moduli spaces $M_L$}]\label{defiloci}
Given a $(-1)$-curve $\ell$ and a conic $C$ in $S$, we set:
\begin{gather*}
P_\ell=\pr(\Ext^1(\ol_S(\ell),\ol_S(-K_S-\ell)))\cong\pr^2,\quad\,
Z_\ell=\pr(\Ext^1(\ol_S(-K_S-\ell),\ol_S(\ell)))\cong\pr^1,\\
E_C=\pr(\Ext^1(\ol_S(C),\ol_S(-K_S-C)))\cong\pr^3,
\end{gather*}
and we denote by $F_C$ the unique sheaf on $S$ sitting in a non-split extension:
$$0\la\ol_S(C)\la F_C\la\ol_S(-K_S-C)\la 0.$$

These loci are crucial in the description of the birational transformation occurring in $M_L$ when $L$ crosses a wall; following Mukai \cite[p.~8]{mukaiADE}, we describe this in detail.
\end{definition}
\subsection{Crossing the wall $(2\ell+K_S)^{\perp}$}\label{flip}
Let us fix a  $(-1)$-curve $\ell$.
Let $\ma{C}$ be a chamber lying in the halfspace
$$(K_S+2\ell)^{>0}:=\{\gamma\in H^2(S,\R)\,|\,\gamma\cdot (K_S+2\ell)>0\},$$ and such that
$\overline{\ma{C}}\cap(2\ell+K_S)^{\perp}$ intersects the ample cone.
Then for any non-trivial extension
$$0\la\ol_S(-K_S-\ell)\la F\la\ol_S(\ell)\la 0,$$
$F$ is stable with respect to $L\in\ma{C}$ \cite[Ch.~II, Th.~1.2.3]{Qin1993}. 
By Lemma \ref{special}$(c)$, these sheaves $F$
 are parametrized by 
$P_\ell\cong\pr^2$, and we have $P_\ell\subset M_L$ 
\cite[Ch.~II, Cor.~1.2.4]{Qin1993}.

Similarly, let  $\ma{C}'$ be a chamber  lying in the halfspace
$(K_S+2\ell)^{<0}$, and such that 
$\overline{\ma{C}'}\cap(2\ell+K_S)^{\perp}$ intersects the ample cone.
Then for any non-trivial extension
$$0\la\ol_S(\ell)\la F'\la\ol_S(-K_S-\ell)\la 0,$$
$F'$ is stable with respect to $L'\in\ma{C}'$; these sheaves $F$ are parametrized by $Z_\ell\cong\pr^1$, and we have $Z_\ell\subset M_{L'}$.

Suppose moreover that $\overline{\ma{C}}$ and $\overline{\ma{C}'}$ share a common facet (which must lie on the hyperplane $(K_S+2\ell)^{\perp}$). Then the moduli spaces $M_L$ and $M_{L'}$ are birational and isomorphic in codimension $1$ under the natural map $[F]\mapsto[F]$. More precisely, the normal bundles of $P_\ell$ and $Z_\ell$ are 
$\ma{N}_{P_\ell/M_L}\cong\ol_{\pr^2}(-1)^{\oplus 2}$ and
$\ma{N}_{Z_\ell/M_{L'}}\cong\ol_{\pr^1}(-1)^{\oplus 3}$
\cite[Prop.~3.6 and Lemma 3.2(ii)]{FriedmanQin}, 
 the birational map $M_L\dasharrow M_{L'}$ is a $K$-negative flip,
 and factors as 
$M_L\leftarrow \widehat{M}\to M_{L'}$, where $\widehat{M}\to M_L$ is the blow-up of $P_\ell$, and $\widehat{M}\to M_{L'}$ is the blow-up of $Z_\ell$ (see \cite[Th.~2]{Qin1993} and
\cite[Th.~3.9]{FriedmanQin}).
\subsection{Crossing the wall $(2C+K_S)^{\perp}$}\label{blowup}
Let us fix a conic $C$.
Let $\ma{C}$ be a chamber  lying in the halfspace
$(K_S+2C)^{>0}$, such that 
$\overline{\ma{C}}\cap(2C+K_S)^{\perp}$ intersects the ample cone.
Then for any non-trivial extension
$$0\la\ol_S(-K_S-C)\la F\la\ol_S(C)\la 0,$$
$F$ is stable with respect to $L\in\ma{C}$ \cite[Ch.~II, Th.~1.2.3]{Qin1993}. By Lemma \ref{special}$(c)$, these sheaves $F$ are parametrized by $E_C\cong\pr^3$, and we have $E_C\subset M_L$ 
\cite[Ch.~II, Cor.~1.2.4]{Qin1993}.

Similarly, let  $\ma{C}'$ be a chamber 
 lying in the halfspace
$(K_S+2C)^{<0}$,
and such that 
$\overline{\ma{C}'}\cap(2C+K_S)^{\perp}$ intersects the ample cone.
Then for any non-trivial extension
$$0\la\ol_S(C)\la F'\la\ol_S(-K_S-C)\la 0,$$
$F'$ is stable with respect to $L'\in\ma{C}'$; then $F'\cong F_C$, which yields a point $[F_C]\in M_{L'}$.

Suppose moreover that $\overline{\ma{C}}$ and $\overline{\ma{C}'}$ share a common facet (which must lie on the hyperplane $(K_S+2C)^{\perp}$). Then the natural map $[F]\mapsto[F]$ yields a birational morphism  $M_L\to M_{L'}$. More precisely, the
 normal bundle of $E_C$ is
$\ma{N}_{E_C/M_L}\cong\ol_{\pr^3}(-1)$ \cite[Prop.~3.6]{FriedmanQin}, and
 $M_L\to M_{L'}$ is just the blow-up of the smooth point $[F_C]$, with exceptional divisor $E_C$ (see \cite[Th.~2]{Qin1993} and
\cite[Th.~3.9]{FriedmanQin}).
\subsection{The moduli space for the outer chamber $\ma{C}_h$}\label{outer}
Let us fix a   cubic $h$. Recall from \ref{secE} that the wall $(2h+K_S)^{\perp}$ cuts the facet $\tau_h=\langle C_1,\dotsc,C_8\rangle$ of $\ma{E}$ (notation as in \ref{notationP^2}). We have $C_1+\cdots+C_8\sim -K_S+5h$, so $-K_S+5h$ belongs to the relative interior of $\tau_h$.

All the walls different from $(2h+K_S)^{\perp}$ cut the cone  $\tau_h$ along a proper face. Thus there is a unique chamber $\ma{C}_h\subset\ma{E}$ such that $\overline{\ma{C}}_h$ intersects the relative interior of $\tau_h$; in particular $\tau_h$ is a facet of  $\overline{\ma{C}}_h$. 
We have $\overline{\ma{C}}_h=\langle C_1,\dotsc,C_8,-K_S+3h\rangle$ (see Lemma \ref{L_t} and the following discussion).
\begin{proposition}[\cite{mukaiADE}, p.~9]\label{P^4}
For every $[F]\in M_{\ma{C}_h}$ the sheaf  $F$ is locally free, and
$M_{\ma{C}_h}\cong\pr(\Ext^1(\ol_S(h),\ol_S(-K_S-h)))=\pr^4$.
\end{proposition}
\subsection{The moduli space $M_L$ when $L$ belongs to some wall}
\begin{lemma}\label{specialL}
Let $L\in\Pic(S)$ be ample, $L\in\ma{E}$. 
We have the following.
\begin{enumerate}[$(a)$]
\item
There exists a unique chamber
$\ma{C}$ such that $L\in\overline{\ma{C}}$ and $\ma{C}\subset (K_S+2D)^{>0}$ for every $(-1)$-curve, conic, or   cubic $D$ with $L\cdot(2D+K_S)=0$.
\item We have  $M_L=M_{\ma{C}}$, and $L\in\Pi$ if and only if $\ma{C}\subset\Pi$.
\item If $L$ belongs to the boundary of $\ma{E}$, then $L$ belongs to a unique wall $(2h+K_S)^{\perp}$ where $h$ is  a   cubic;
$L$ is in the relative interior of the facet 
$\tau_h=(2h+K_S)^{\perp}\cap\ma{E}$, $\ma{C}=\ma{C}_h$, and $M_L=M_{\ma{C}_h}\cong\pr(\Ext^1(\ol_S(h),\ol_S(-K_S-h)))=\pr^4$.
\item Suppose that $L$ is in the interior of $\ma{E}$, and that it is contained in the walls $(2\ell_i+K_S)^{\perp}$ and $(2C_j+K_S)^{\perp}$,  $\ell_i$  a $(-1)$-curve and $C_j$ a conic, for $i=1,\dotsc,r$ and $j=1,\dotsc,s$, with $r\geq 0$ and $s\geq 0$.
 Then $M_L$ contains the special loci $P_{\ell_i}\cong\pr^2$ and $E_{C_j}\cong\pr^3$ for all $i,j$, and these loci are pairwise disjoint in $M_L$.
\end{enumerate}
\end{lemma}
\begin{proof}
The anticanonical class $-K_S$ belongs to $\ma{E}$  and is not contained in any wall (see Rem.~\ref{fanochamber}). Thus 
we can choose  $m\gg 0$ such that the class $L':=-K_S+mL$ is contained in a chamber, and such that no wall contains a convex combination of $L$ and $L'$ different from $L$.
Then the chamber $\ma{C}$ containing $L'$ satisfies $(a)$, because for every $D$ as in $(a)$ we have $L'\cdot (K_S+2D)=-K_S\cdot (K_S+2D)>0$.

It follows from standard arguments (see for instance \cite[Ch.~4, Prop.~22]{friedmanbook}; the argument for torsion free sheaves is the same) that:
\begin{equation}\label{comparison}
\text{$L$-$\mu$-stable}\ \Rightarrow\ \text{$\ma{C}$-stable}\ \Rightarrow\ \text{$L$-$\mu$-semistable}.\end{equation}

Suppose that $L$ is on the boundary of $\ma{E}$. Since $L$ is ample, by Rem.~\ref{boundaryample} $L$ is contained in the relative interior of a facet 
$\tau_h=(2h+K_S)^{\perp}\cap\ma{E}$, where $h$ is a   cubic. By  \ref{outer} we see that $(2h+K_S)^{\perp}$ is the unique wall containing $L$, and that $\ma{C}=\ma{C}_h$.

We have described $\ma{C}_h$-stable sheaves in Prop.~\ref{P^4}, they concide with non-split extensions
$$0\la\ol_S(-K_S-h)\la F\la\ol_S(h)\la 0.$$
In particular such sheaves are all $L$-stable by Prop.~\ref{walls}$(c)$.

Conversely, by Prop.~\ref{walls} and Lemma \ref{special}$(d)$, the strictly $L$-$\mu$-semistable sheaves are the sheaves appearing in an extension as above, and the split case $\ol_S(-K_S-h)\oplus\ol_S(h)$; this last one is not $L$-stable.

Together with \eqref{comparison}, this shows that $L$-stability coincides with $\ma{C}_h$-stability, and we get $(c)$ by Prop.~\ref{P^4}. Notice also that there are no $L$-$\mu$-stable sheaves.

Suppose now that $L$ is not on the boundary of $\ma{E}$, as in $(d)$.
By Prop.~\ref{walls}, the set of strictly $L$-$\mu$-semistable sheaves is given by:
\begin{enumerate}[(1)]
\item
 the sheaves in non-split extensions $0\la\ol_S(-K_S-D)\la F\la\ol_S(D)\la 0$
\item the sheaves in  extensions $0\la\ol_S(D)\la F\la\ol_S(-K_S-D)\la 0$
\end{enumerate}
where $D\in\{\ell_1,\dotsc,\ell_r,C_1,\dotsc,C_s\}$. The sheaves in (1) are $L$-stable by Prop.~\ref{walls}$(c)$, and $\ma{C}$-stable as explained in \ref{flip} and \ref{blowup}. On the other hand, the sheaves in (2) are neither $L$-stable (again by Prop.~\ref{walls}$(c)$) nor $\ma{C}$-stable (because $L'\cdot (K_S+2D)>0$). 

Together with \eqref{comparison}, this shows that $L$-stability coincides with $\ma{C}$-stability, and we get $(b)$. The sheaves in (1) yield the loci $P_{\ell_1},\dotsc,P_{\ell_r},E_{C_1},\dotsc,E_{C_s}$ in $M_L$. Finally,  given a strictly $L$-$\mu$-semistable sheaf $F$  as in (1), the subsheaf $\ol_S(-K_S-D)$ is unique (see \cite[Ch.~4, Prop.~21]{friedmanbook}), so these loci are pairwise disjoint, and
we get $(d)$.
\end{proof}
\subsection{Properties of $M_L$}\label{properties}
\begin{corollary}\label{birational}
Let $L_1,L_2\in \Pic(S)$  be ample, $L_1,L_2\in\ma{E}$. Then  there exist non-empty dense open subsets $U_1\subseteq M_{S,L_1}$ and  $U_2\subseteq M_{S,L_2}$ parametrizing the same objects; this yields a natural birational 
 map $\ph\colon M_{S,L_1}\dasharrow M_{S,L_2}$ given by $\ph([F])=[F]$.

Suppose that $L_1\in\Pi$. Then $\ph$ is a contracting\footnote{A birational map $\ph\colon X_1\dasharrow X_2$ is \emph{contracting} if there exist open subsets $U_1\subseteq X_1$ and $U_2\subseteq X_2$ such that $\ph$ yields an isomorphism between $U_1$ and $U_2$, and $\codim (X_2\smallsetminus U_2)\geq 2$.} 
birational map, and  $\ph$ is a pseudo-isomorphism if and only if
 $L_2\in\Pi$.
\end{corollary}
\begin{proof}
By Lemma \ref{specialL} we can assume that $L_1$ and $L_2$ do not belong to any wall. Then the first statement follows from the explicit description of wall crossings in \ref{flip} and \ref{blowup}. If $L_1\in\Pi$, then $\ph$ is a compositions of flips and blow-downs, so it is contracting. Moreover, $\ph$ is a pseudo-isomorphism if and only if it can be factored as a sequence of flips, if and only if 
 $L_2\in\Pi$. See \cite[p.~9]{mukaiADE}.
\end{proof}
The following are  straightforward consequences of Cor.~\ref{nonempty},  Cor.~\ref{birational} and Prop.~\ref{P^4}.
\begin{corollary}\label{nonempty2}
Let $L\in\Pic(S)$ be ample. The following are equivalent:
\begin{enumerate}[$(i)$]
\item there exists a $\mu$-semistable  torsion-free sheaf $F$ of rank $2$ with $c_1(F)=-K_S$ and $c_2(F)=2$;
\item $M_L\neq\emptyset$;
\item $L\in\ma{E}$, namely
$L\cdot (2h+K_S)\geq 0$ for every   cubic $h$.
\end{enumerate}
\end{corollary}
\begin{corollary}\label{connected}
Let $L\in\ma{E}$ be an ample line bundle. Then $M_L$ is irreducible. 
\end{corollary}
\begin{corollary}\label{locallyfree}
Let $L\in\ma{E}$ be an ample line bundle, and $F$ a $\mu$-semistable  torsion-free sheaf $F$ of rank $2$ with $c_1(F)=-K_S$ and $c_2(F)=2$.
 Then $F$ is locally free.
\end{corollary}
\begin{proof} This holds for $L\in\ma{C}_h$  by Prop.~\ref{P^4}. By the explicit description of wall crossings given in  \ref{flip} and \ref{blowup}, the statement stays true whenever $L$ is in a chamber. Finally if $L$ does not belong to a chamber, the statement follows from Lemma \ref{specialL}$(b)$.
\end{proof}
\subsection{The moduli space of  $\mu$-semistable   sheaves}\label{secslope}
We consider now the moduli space $M_L^{\mu}$ of $\mu$-semistable  rank $2$ torsion-free sheaves $F$ on $S$, with $c_1(F)=-K_S$ and $c_2(F)=2$, see \cite[\S 8.2]{huybrechtslehn} and references therein. 
We recall that by Cor.~\ref{locallyfree}, every such sheaf $F$ is locally free.

 By  \cite[Def.\ and Th.~8.2.8]{huybrechtslehn}  $M_L^{\mu}$ is a projective scheme (which we consider with its reduced scheme structure), endowed with a natural morphism 
$\gamma\colon M_L\to M_L^{\mu}$,
with the following properties:
\begin{enumerate}[$(a)$]
\item the points of $M_L^{\mu}$ are in bijection with equivalence classes of 
$\mu$-semistable  rank $2$ locally free sheaves $F$ on $S$, with $c_1(F)=-K_S$ and $c_2(F)=2$;
\item by \cite[Def.~8.2.10 and Th.~8.2.11]{huybrechtslehn},  if $F_1$ and $F_2$ are not isomorphic, then they are equivalent if and only if they are strictly $\mu$-semistable, and in the 
 exact sequences given by Prop.~\ref{walls}:
 $$0\la \ol_S(D_i)\la F_i\la\ol_S(-K_S-D_i)\to 0\qquad i=1,2$$
we have either $\ol_S(D_1)\cong\ol_S(D_2)$ or  $\ol_S(D_1)\cong\ol_S(-K_S-D_2)$;
\item for $[F]\in M_L$, $\gamma([F])$ is the class of $F$ in $M_L^{\mu}$, and
$\gamma$ is an isomorphism on the open subsets of $\mu$-stable, locally free sheaves \cite[Cor.~8.2.16]{huybrechtslehn}.
\end{enumerate}

Thus if the polarization $L$ is in a chamber, it follows from $(c)$ and Cor.~\ref{locallyfree} that $\gamma$ is an isomorphism and the two moduli spaces coincide.
In general we have the following.
\begin{lemma}\label{slope}
Let $L\in\Pic(S)$ be ample, $L\in\ma{E}$. 

 If $L$ belongs to the boundary of $\ma{E}$, then $M_L^{\mu}$ is a point.

Suppose that $L$ is not on the boundary of $\ma{E}$, and that it is contained in the walls $(2\ell_i+K_S)^{\perp}$ and $(2C_j+K_S)^{\perp}$,  $\ell_i$  a $(-1)$-curve and $C_j$ a conic, for $i=1,\dotsc,r$ and $j=1,\dotsc,s$, with $r\geq 0$ and $s\geq 0$.
Then $\gamma$ is birational, $\Exc(\gamma)=P_{\ell_1}\cup\cdots\cup P_{\ell_r}\cup 
E_{C_1}\cup\cdots\cup E_{C_s}$, and the image of $\Exc(\gamma)$ is given by $r+s$ distinct points.
\end{lemma}
\begin{proof}
The statement follows from $(b)$ above and from Lemma \ref{specialL} (and its proof).
\end{proof} 
\section{The determinant map}\label{sec_determinant}
\subsection{}\label{introsec_determinant} In this section we recall the classical construction of determinant line bundles on the moduli space $M_L$, and use it to define a group homomorphism
$\rho\colon \Pic(S)\to\Pic(M_L)$,
that we  call the determinant map.
We study the properties of $\rho$,  together with its real extension $\rho\colon H^2(S,\R)\to H^2(M_L,\R)$ and its dual map
$\zeta:=\rho^t\colon \N(M_L)\la\N(S)=H^2(S,\R)$,
where $\N(M_L)$ is the vector space of real $1$-cycles in $M_L$ up to numerical equivalence, and similarly for $S$.
Using the classical positivity properties of the determinant line bundle, we determine the images via $\zeta$ of the lines in the exceptional loci in $M_L$
(Cor.~\ref{facet_ell} and \ref{facet_C}). We also compare the maps $\rho$ for different $L$'s (Lemma \ref{cod1}), and show that $\rho$ is equivariant for the action of the group of automorphisms of $S$ which fix the polarization $L$ (Prop.~\ref{equivariant}).
Finally we deduce that the moduli space $M_{-K_S}$ is a Fano variety (Prop.~\ref{fano}), and more generally that $M_L$ is always a Mori dream space (Cor.~\ref{MDS}).
\subsection{Determinant line bundles}
We refer the reader to \cite[Ch.\ 8]{huybrechtslehn} and references therein, in particular Le Potier  \cite{lepotierdet} and Li \cite{JunLi}, for the construction and properties of  determinant line bundles; let us give a brief outline.

Let $L\in\Pic(S)$ be ample, $L\in\ma{E}$. To simplify the notation, set $M:=M_L$.
Let $K(S)$ be the Grothendieck group of coherent sheaves on $S$, and similarly for the moduli space $M$ and the product $S\times M$; since $S$ and $M$ are smooth projective varieties, their Grothendieck groups of coherent sheaves are naturally isomorphic to the Grothendieck groups of locally free sheaves. Moreover, being $S$ a rational surface, we have a group isomorphism
\begin{equation}\label{groupiso}
(\rk,c_1,\chi)\colon K(S)\la\Z\oplus \Pic(S)\oplus \Z.
\end{equation} 
Under this isomorphism, the class $[\ol_{pt}]$ corresponds to $(0,\ol_S,1)$, in particular it does not depend on the point.

Let $\mathfrak{f}\in K(S)$ be the class of a rank $2$ torsion-free sheaf $F$ on $S$ with $c_1(F)=-K_S$ and $c_2(F)=2$; for every $\mathfrak{c}\in K(S)$ we have
$$
\chi(\mathfrak{f}\otimes \mathfrak{c})=2\chi(\mathfrak{c})-K_S\cdot c_1(\mathfrak{c})-\rk \mathfrak{c}.
$$
We are interested in the subgroup $\mathfrak{f}^{\perp}=\{\mathfrak{c}\in K(S)\,|\,
\chi(\mathfrak{f}\otimes \mathfrak{c})=0\}$. Notice that since $\chi(\mathfrak{f}\otimes [\ol_{pt}])=2$, for every $\mathfrak{c}\in K(S)$ we have $2\mathfrak{c}-\chi(\mathfrak{f}\otimes \mathfrak{c})[\ol_{pt}]\in \mathfrak{f}^{\perp}$.

By Rem.~\ref{stable} and Cor.~\ref{locallyfree}, for every $[F]\in M$ the sheaf $F$ is locally free and stable, thus there exists a universal vector bundle $\ma{U}$ over $S\times M$, which is unique up to twists with pull-backs of line bundles on $M$. We denote by $\pi_S\colon S\times M\to S$ and 
$\pi_{M}\colon S\times M\to M$ the two projections.
One defines a group homomorphism
$\lambda_{\ma{U}}\colon K(S)\to\Pic(M)$ \cite[Def.~8.1.1]{huybrechtslehn} as the composition:
$$\lambda_{\ma{U}}\colon K(S)\stackrel{\pi_S^*}{\la} K(S\times M)\stackrel{\otimes[\ma{U}]}{\la}K(S\times M)\stackrel{(\pi_{M})_!}{\la} 
K(M)\stackrel{\det}{\la}\Pic(M),$$
where we recall that $(\pi_{M})_!=\sum_i(-1)^iR^i(\pi_{M})_*$, thus
\begin{equation}\label{deflambda}
\lambda_{\ma{U}}(\mathfrak{c})=\det\bigl((\pi_{M})_!(\pi_S^*\mathfrak{c}\otimes[\ma{U}])\bigr).
\end{equation}
Given a line bundle $N\in\Pic(M)$, the vector bundle $\ma{U}\otimes\pi_{M}^*N$ is another universal family which yields another group homomorphism $\lambda_{\ma{U}\otimes\pi_{M}^*N}\colon K(S)\to\Pic(M)$. We have: 
$$\lambda_{\ma{U}\otimes\pi_{M}^*N}(\mathfrak{c})=\lambda_{\ma{U}}(\mathfrak{c})\otimes[N^{\otimes\chi(\mathfrak{f}\otimes \mathfrak{c})}]\quad\text{ for every $\mathfrak{c}\in K(S)$.}$$
In order to have an intrinsic map, one restricts to the subgroup $\mathfrak{f}^{\perp}$, and  set:
$$\lambda:=(\lambda_\ma{U})_{|\mathfrak{f}^{\perp}}\colon  \mathfrak{f}^{\perp}\la\Pic(M).$$
\begin{thm}[\cite{huybrechtslehn}, Th.~8.2.8]\label{thmL_1}
Let $H\in|L|$ be a curve,
and set $\mathfrak{h}:=[\ol_H]\in K(S)$.
Consider the class 
$$u_1=-2\mathfrak{h}+\chi(\mathfrak{f}\otimes \mathfrak{h})[\ol_{pt}]\in \mathfrak{f}^{\perp}\subset K(S),$$
and set $\ma{L}_1:=\lambda(u_1)\in\Pic(M_L)$ \cite[Def.~8.1.9]{huybrechtslehn}. Then $\ma{L}_1$ is semiample, and some positive multiple of $\ma{L}_1$ defines the map $\gamma\colon M_L\to M_L^{\mu}$ (see \eqref{secslope}).
\end{thm}
\subsection{The map $\rho$} 
We define a map $\mathfrak{u}\colon\Pic(S)\to  \mathfrak{f}^{\perp}\subset K(S)$, and use it to define a map  
$\rho:=\lambda\circ\mathfrak{u}\colon\Pic(S)\to\Pic(M)$, as follows.
We set, for every $P\in\Pic(S)$:
\begin{equation}
\label{u(P)}
\mathfrak{u}(P):=2\bigl([P^{\otimes(-1)}]-[\ol_S]\bigr)-\chi\Bigl(\mathfrak{f}
\otimes 
\bigl([P^{\otimes(-1)}]-[\ol_S]\bigr)\Bigr)[\ol_{pt}]\in \mathfrak{f}^{\perp}\subset K(S),
\end{equation}
where $\chi(\mathfrak{f}
\otimes 
([P^{\otimes(-1)}]-[\ol_S]))=P\cdot(P+2K_S)$;
one can check that
\begin{equation}\label{numerics}
(\rk \mathfrak{u}(P),c_1(\mathfrak{u}(P)),\chi(\mathfrak{u}(P)))=(0,P^{\otimes(-2)},-K_S\cdot P).
\end{equation}
Since  the map \eqref{groupiso} is an isomorphism, \eqref{numerics} implies that $\mathfrak{u}$ is a group homomorphism. 
The following is the key property of $\mathfrak{u}$; it will be used in the proof of Cor.~\ref{determinant}.
\begin{remark}\label{keyu}
Let $P\in\Pic(S)$ be effective,
 $H\in|P|$ a curve, and set $\mathfrak{h}:=[\ol_H]\in K(S)$. 
The exact sequence
$$0\la\ol_S(-H)\la\ol_S\la\ol_H\la 0$$
yields $\mathfrak{h}=[\ol_S]-[P^{\otimes(-1)}]$ in $K(S)$, hence
 $\mathfrak{u}(P)=-2\mathfrak{h}+\chi(\mathfrak{f}\otimes \mathfrak{h})[\ol_{pt}]$. 
\end{remark}

Finally we define 
$\rho:=\lambda\circ\mathfrak{u}\colon\Pic(S)\to\Pic(M)$, namely 
$\rho(P)=\lambda(\mathfrak{u}(P))$ for every $P\in\Pic(S)$
(this is in fact twice the map defined in \cite[(1.8)]{JunLi}).
\begin{remark}\label{same}
The maps $\lambda$ and $\rho$ depend only on the moduli space $M$, so that $\rho$ is the same for every polarization $L$ which yields the same stability condition and hence the 
same moduli space (like in the situation of Lemma \ref{specialL}).
\end{remark}
Whenever we need to highlight the dependence on the polarization, we will write
$\rho=\rho_L\colon \Pic(S)\to\Pic(M_L)$ or $\rho=\rho_{\ma{C}}\colon \Pic(S)\to\Pic(M_{\ma{C}})$ if $\ma{C}$ is a chamber.
\begin{corollary}\label{determinant}
Let $L\in\Pic(S)$ be ample, $L\in\ma{E}$, and consider $\rho(L)\in\Pic(M_L)$. Then  $\rho(L)=\ma{L}_1$ (notation as in Th.~\ref{thmL_1}), and the following hold:
\begin{enumerate}[$(a)$]
\item
$\rho(L)$ is semiample, and some positive multiple of $\rho(L)$ defines the map $\gamma\colon M_L\to M_L^{\mu}$ (see \eqref{secslope});
\item $\rho(L)$ is
 big if $L$ lies in the interior of the cone $\ma{E}$, while 
 $\rho(L)=0$ if $L$ lies on the boundary of $\ma{E}$;
\item
 $\rho(L)$ is ample if and only if  $L$ belongs to a chamber.
\item  Suppose that $L$ is in the interior of $\ma{E}$, and that it is contained in the walls $(2\ell_i+K_S)^{\perp}$ and $(2C_j+K_S)^{\perp}$,  $\ell_i$  a $(-1)$-curve and $C_j$ a conic, for $i=1,\dotsc,r$ and $j=1,\dotsc,s$, with $r\geq 0$ and $s\geq 0$.
If $\Gamma\subset M_L$ is an irreducible curve, then
$\rho(L)\cdot \Gamma=0$ if and only if $\Gamma\subset P_{\ell_1}\cup\cdots\cup P_{\ell_r}\cup E_{C_1}\cup\cdots\cup E_{C_s}$.
\end{enumerate} 
\end{corollary}
\begin{proof}
Rem.~\ref{keyu} and Th.~\ref{thmL_1}
yield $\rho(L)=\ma{L}_1$, and $(a)$. Then the remaining statements follow from Lemma \ref{slope}.
\end{proof}
In particular, if $\ma{C}\subset\ma{E}$ is a chamber, it follows from $(c)$ above that $\rho(L)$ is ample on $M_{\ma{C}}$ for every $L\in\ma{C}$.
\begin{lemma}\label{cod1}
Let $L,L'\in\ma{E}$ be ample line bundles, 
and let $\ph\colon M_L\dasharrow M_{L'}$ be the natural birational map (see Cor.~\ref{birational}).
Let $U\subseteq M_L$ and $U'\subseteq M_{L'}$ be the open subsets over which $\ph$ is an isomorphism. Then for every $P\in\Pic(S)$ we have
$$\rho_L(P)_{|U}\cong \ph^*\bigl(\rho_{L'}(P)_{|U'}\bigr).$$
If moreover $L$ and $L'$ belong to $\Pi$, then we have a commutative diagram:
{\footnotesize$$\xymatrix{&{\Pic(S)}\ar[dl]_{\rho_L}\ar[dr]^{\rho_{L'}}&\\
{\Pic(M_L)}&&{\Pic(M_{L'}).}\ar[ll]_{\ph^*}
}$$}
\end{lemma}
\begin{proof}
For simplicity set $M:=M_L$ and $M':=M_{L'}$. We denote by $\pi_{\scriptscriptstyle A\times B,A}$ the projection $A\times B\to A$.

We can choose universal vector bundles $\ma{U}$ on $S\times M$ and $\mathcal{U}'$ on $S\times M'$ such that $\ma{U}_{|S\times U}= (\Id_S\times\ph)^*\ma{U}'_{|S\times U'}$.
Let $F$ be a locally free sheaf on $S$. We have
$$
\bigl((\pi_{\scriptscriptstyle S\times M,S})^*F\otimes\ma{U}\bigr)_{|S\times U}=(\pi_{\scriptscriptstyle S\times U,S})^*F\otimes 
\ma{U}_{|S\times U}
=(\Id_S\times\ph)^*\bigl((\pi_{\scriptscriptstyle S\times U',S})^*F\otimes 
\ma{U}'_{|S\times U'}\bigr)$$
and hence by base change:
\begin{multline*}
\bigl(R^i(\pi_{\scriptscriptstyle S\times M,M})_*((\pi_{\scriptscriptstyle S\times M,S})^*F\otimes\ma{U})\bigr)_{|U}=R^i(\pi_{\scriptscriptstyle S\times U,U})_*\bigl((\pi_{\scriptscriptstyle S\times M,S})^*F\otimes\ma{U}\bigr)_{|S\times U}\\
\shoveright{=
R^i(\pi_{\scriptscriptstyle S\times U,U})_*(\Id_S\times\ph)^*\bigl((\pi_{\scriptscriptstyle S\times U',S})^*F\otimes 
\ma{U}'_{|S\times U'}\bigr)}\\
\cong\ph^*R^i(\pi_{\scriptscriptstyle S\times U',U'})_*\bigl((\pi_{\scriptscriptstyle S\times U',S})^*F\otimes 
\ma{U}'_{|S\times U'}\bigr)=\ph^*\bigl(R^i(\pi_{\scriptscriptstyle S\times M',M'})_*((\pi_{\scriptscriptstyle S\times M',S})^*F\otimes 
\ma{U}')\bigr)_{|U'}.
\end{multline*}
Taking the determinant  commutes with restricting to an open subset, therefore we get: 
$$\bigl(\det(R^i(\pi_{\scriptscriptstyle S\times M,M})_*((\pi_{\scriptscriptstyle S\times M,S})^*F\otimes\ma{U})\bigr)_{|U}
\cong \ph^*\det\bigl(R^i(\pi_{\scriptscriptstyle S\times M,M})_*((\pi_{\scriptscriptstyle S\times M,S})^*F\otimes 
\ma{U}')\bigr)_{|U'},$$
and hence
$$\lambda_{\ma{U}}([F])_{|U}\cong \ph^*\bigl(\lambda_{\ma{U}'}([F])_{|U'}\bigr),$$
which yields the first statement.
Now if $L$ and $L'$ belong to $\Pi$, then
by Cor.~\ref{birational} the complements of $U$ and $U'$ both have codimension $>1$, and we deduce that $\lambda_{\ma{U}}([F])=\ph^*(\lambda_{\ma{U}'}([F]))$ in $\Pic(M)$, which concludes the proof.
\end{proof}
Let $\ma{C}\subset\Pi$ be a chamber, and $L\in\overline{\ma{C}}$. 
Cor.~\ref{determinant} and Lemma \ref{cod1} yield the following positivity property of $\rho(L)$ on $M_{\ma{C}}$.
\begin{lemma}\label{rho(L)}
Let $\ma{C}\subset\Pi$ be a chamber and $L\in\Pic(S)$ ample, $L\in\overline{\ma{C}}$.
Suppose that
$L$ is contained in the walls $(2\ell_i+K_S)^{\perp}$ and $(2C_j+K_S)^{\perp}$,  $\ell_i$  a $(-1)$-curve and $C_j$ a conic, for $i=1,\dotsc,r$ and $j=1,\dotsc,s$, with $r\geq 0$ and $s\geq 0$.
Suppose also that $\ma{C}$ is contained in $(2\ell_i+K_S)^{>0}$ for $i=1,\dotsc,h$, and in
$(2\ell_i+K_S)^{<0}$ for $i=h+1,\dotsc,r$, with $h\in\{0,\dotsc,r\}$.

Then $M_{\ma{C}}$ contains the loci $P_{\ell_1},\dotsc,P_{\ell_h},Z_{\ell_{h+1}},\dotsc,Z_{\ell_r},E_{C_1}\dotsc,E_{C_s}$.
Moreover $\rho(L)\in\Pic(M_{\ma{C}})$ is nef and big, and if $\Gamma\subset M_{\ma{C}}$ is an irreducible curve, then
$\rho(L)\cdot \Gamma=0$ if and only if $\Gamma\subset 
P_{\ell_1}\cup\cdots\cup P_{\ell_h}\cup Z_{\ell_{h+1}}\cup\cdots\cup Z_{\ell_r}\cup E_{C_1}\cup\cdots\cup E_{C_s}$.
\end{lemma}
\begin{proof}
Notice first of all that since $\ma{C}\subset\Pi$, we have $\ma{C}\subset(2C_j+K_S)^{>0}$ for $j=1,\dotsc,s$, so the first statement follows from \ref{flip} and \ref{blowup}.

Let $\ma{C}'$ be the chamber such that $L\in\overline{\ma{C'}}$ and 
 $\ma{C}'\subset(2\ell_i+K_S)^{>0}$, $\ma{C}'\subset(2C_j+K_S)^{>0}$  for every $i=1,\dotsc,r$ and  $j=1,\dotsc,s$, as in Lemma \ref{specialL}$(a)$. 
We consider both determinant maps
$$\rho_{\ma{C}}\colon \Pic(S)\la\Pic(M_{\ma{C}})\quad\text{and}\quad
 \rho_{\ma{C}'}\colon \Pic(S)\la\Pic(M_{\ma{C}'}).$$

We have $M_L=M_{\ma{C}'}$ by Lemma \ref{specialL}$(b)$.
Thus by Cor.~\ref{determinant} (see also Rem.~\ref{same}), $\rho_{\ma{C}'}(L)\in \Pic({M}_{\ma{C}'})$ is nef and big, and has intersection zero precisely with curves contained in the loci $P_{\ell_i}$ and $E_{C_j}$, for $i=1,\dotsc,r$ and $j=1,\dotsc,s$.

Going from $\ma{C}'$ to $\ma{C}$, we cross the walls $(K_S+2\ell_i)^{\perp}$ 
for $i=h+1,\dotsc,r$; correspondingly (see \ref{flip}) the natural birational map 
$\ph\colon M_{\ma{C}'}\dasharrow M_{\ma{C}}$ is an isomorphism in codimension one and flips $P_{\ell_{i}}\cong\pr^2$
to  $Z_{\ell_{i}}\cong\pr^1$, for every $i=h+1,\dotsc,r$. 

Notice that in $M_{\ma{C}'}$ the loci $P_{\ell_1},\dotsc,P_{\ell_r},E_{C_1},\dotsc,E_{C_s}$ are pairwise disjoint 
(see Lemma \ref{specialL}$(d)$), so that $P_{\ell_1},\dotsc,P_{\ell_h},E_{C_1},\dotsc,E_{C_s}$
 are contained in the open subset where $\ph$ is an isomorphism.
Therefore $(\ph^*)^{-1}(\rho_{\ma{C}'}(L))\in\Pic(M_{\ma{C}})$ is nef, and has intersection zero only with curves contained in $P_{\ell_1},\dotsc,P_{\ell_h},Z_{\ell_{h+1}},\dotsc,Z_{\ell_r},E_{C_1}\dotsc,E_{C_s}$.
 On the other hand $(\ph^*)^{-1}(\rho_{\ma{C}'}(L))=\rho_{\ma{C}}(L)$  by Lemma \ref{cod1}, so we get the statement.
\end{proof}
\begin{lemma}\label{-K}
Let $\ma{C}\subset\Pi$ be a chamber. Then $\rho(-K_S)=-K_{M_{\ma{C}}}\in\Pic(M_{\ma{C}})$.
\end{lemma}
\begin{proof}
Consider the moduli space $Y:=M_{-K_S}$, and the associated map 
$\rho_{-K_S}\colon\Pic(S)\to\Pic(Y)$. By Cor.~\ref{determinant} and \cite[Th.~8.3.3]{huybrechtslehn} we have $\rho_{-K_S}(-K_S)=-K_Y$. On the other hand, since $-K_S\in\Pi$ and $\ma{C}\subset\Pi$, there is a pseudo-isomorphism $\ph\colon M_{\ma{C}}\dasharrow Y$ (see Cor.~\ref{birational}), and $\ph^*(-K_Y)=-K_{M_{\ma{C}}}$.
The statement follows from Lemma \ref{cod1}.
\end{proof}
\subsection{The map $\zeta$}
Let us consider now the transpose map of $\rho$:
$$\zeta:=\rho^t\colon \N(M_{\ma{C}})\la\N(S)=H^2(S,\R).$$
It follows from Lemmas \ref{rho(L)} and \ref{-K} that we can determine the images, via  $\zeta$, of the lines in the exceptional loci of $M_{\ma{C}}$.
\begin{corollary}\label{facet_ell}
Let $\ma{C}\subset\Pi$ be a chamber such that $\overline{\ma{C}}$ has a facet on the wall $(2\ell+K_S)^\perp$, where $\ell$ is a $(-1)$-curve.
Let $\Gamma_\ell\subset M_{\ma{C}}$ be an irreducible curve defined as follows:
\begin{enumerate}[$\bullet$]
\item if  $\ma{C}\subset (2\ell+K_S)^{>0}$, then  $\Gamma_\ell$ is a line in  $P_\ell\cong\pr^2\subset M_{\ma{C}}$;
\item if $\ma{C}\subset (2\ell+K_S)^{<0}$, then  $\Gamma_\ell=Z_\ell\subset M_{\ma{C}}$.
\end{enumerate}
Then $\zeta(\Gamma_\ell)=\begin{cases}\ \ 2\ell+K_S\quad&\text{if   $\ma{C}\subset (2\ell+K_S)^{>0}$;}\\
-2\ell-K_S\quad&\text{if $\ma{C}\subset (2\ell+K_S)^{<0}$.}\end{cases}$
\end{corollary}
\begin{proof}
If $\ma{C}\subset (2\ell+K_S)^{>0}$ we have $\ma{N}_{P_\ell/M_{\ma{C}}}\cong\ol_{\pr^2}(-1)^{\oplus 2}$ (see \ref{flip}), thus $-K_{M_{\ma{C}}}\cdot\Gamma_\ell=1$. Similarly, if $\ma{C}\subset (2\ell+K_S)^{<0}$, we get
  $-K_{M_{\ma{C}}}\cdot\Gamma_\ell=-1$.

Let $L\in\Pic(S)$ be in the relative interior of the facet $\overline{\ma{C}}\cap(2\ell+K_S)^{\perp}$. The boundary of $\Pi$ intersects 
the facet $\overline{\ma{C}}\cap(2\ell+K_S)^{\perp}$ along proper faces, hence $L$ 
is in the interior of $\Pi$, and it is ample by Rem.~\ref{boundaryample}.
Consider $\rho(L)\in\Pic(M_{\ma{C}})$. Then by Lemma \ref{rho(L)}, $\rho(L)\cdot\Gamma_\ell=0$. Since the linear span of the facet is the hyperplane $(2\ell+K_S)^{\perp}$, this yields 
$$\rho((2\ell+K_S)^{\perp})\subseteq \Gamma_\ell^{\perp}\subset H^2(M_{\ma{C}},\R),$$
and dually $\zeta(\Gamma_\ell)=a (2\ell+K_S)$ for some $a\in\R$. 

On the other hand we have $\rho(-K_S)=-K_{M_{\ma{C}}}$ by Lemma \ref{-K}, so that 
$$-K_{M_{\ma{C}}}\cdot\Gamma_\ell=\rho(-K_S)\cdot\Gamma_\ell=-K_S\cdot
\zeta(\Gamma_\ell)=a(-K_S)\cdot(2\ell+K_S)=a.$$
This yields the statement.
\end{proof}
With a similar proof, one shows the following.
\begin{corollary}\label{facet_C}
Let $\ma{C}\subset\Pi$ be a chamber such that $\overline{\ma{C}}$ has a facet on the wall $(2C+K_S)^\perp$, where $C\subset S$ is a conic.
Let $\Gamma_C\subset M_{\ma{C}}$ be a line in $E_C\cong\pr^3\subset
M_{\ma{C}}$.
Then
$\zeta(\Gamma_C)=2C+K_S$.
\end{corollary}
\begin{remark}\label{cod1curves}
Let $\ma{C}_1$ and $\ma{C}_2$ be two chambers contained in $\Pi$,
and consider the maps $\zeta_{\ma{C}_1}\colon\N(M_{\ma{C}_1})\to H^2(S,\R)$ and 
$\zeta_{\ma{C}_2}\colon\N(M_{\ma{C}_2})\to H^2(S,\R)$, with the obvious notation.
Let $\ph\colon M_{\ma{C}_1}\dasharrow M_{\ma{C}_2}$ be the natural birational map (see Cor.~\ref{birational}), $\Gamma\subset M_{\ma{C}_1}$ an irreducible curve contained in the open subset where $\ph$ is an isomorphism, and $\Gamma':=\ph(\Gamma)$. Then
$\zeta_{\ma{C}_1}(\Gamma)=\zeta_{\ma{C}_2}(\Gamma')\in H^2(S,\R)$.
Indeed if $L\in\Pic(S)$, using Lemma \ref{cod1}, we have:
$$\zeta_{\ma{C}_2}(\Gamma')\cdot L=\Gamma'\cdot \rho_{\ma{C}_2}(L)=
\Gamma\cdot \ph^*\rho_{\ma{C}_2}(L)=\Gamma\cdot \rho_{\ma{C}_1}(L)=\zeta_{\ma{C}_1}(\Gamma)\cdot L.$$
\end{remark}
\subsection{$M_L$ is a Mori dream space}
Mori dream spaces are projective varieties with an especially nice 
behaviour with respect to
 birational geometry and the Minimal Model Program (see  \cite{hukeel} for more details). 
Fano varieties, and more generally log Fano varieties, are Mori dream spaces. We recall that a smooth projective variety $M$ is log Fano if there exists an effective $\Q$-divisor $\Delta$ on $M$ such that $-(K_M+\Delta)$ is ample, and the pair $(M,\Delta)$ is klt. We refer the reader to \cite[Def.~2.34]{kollarmori} for the notion of klt pair; since $M$ is smooth, this is a condition on the singularities of $\Delta$, which is automatically satisfied when $\Delta=0$. 
 
The following is another important consequence of Lemma \ref{-K}; see \cite[Prop.~3.3]{BMW} for a related result.
\begin{proposition}[the Fano model $Y$]\label{fano}
The moduli space $Y:=M_{-K_S}$ is a smooth Fano $4$-fold; we have $M_L=Y$ for every ample line bundle $L\in\ma{N}$.
\end{proposition}
\begin{proof}
The moduli space $Y$ is a smooth projective $4$-fold  by 
Cor.~\ref{smoothness} and \ref{connected}. Moreover  $-K_{Y}=\rho(-K_S)$ by Lemma \ref{-K}, so it is ample by Cor.~\ref{determinant}$(c)$, and $Y$ is Fano.
The second statement follows from
 Rem.~\ref{fanochamber} and 
Lemma \ref{specialL}.
\end{proof} 
\begin{corollary}\label{MDS}
For every $L\in\Pic(S)$ ample, $L\in\ma{E}$, the moduli space $M_L$ is log Fano and a Mori dream space.
\end{corollary}
\noindent  Via the relation with the blow-up of $\pr^4$ in $8$ general points (Th.~\ref{Xintro}), the Corollary above is already known by \cite[Th.~1.3]{CoxCT}; see also \cite[Th.~1.3]{araujo_massarenti}.  
\begin{proof}
By Prop.~\ref{fano}, the moduli space $Y=M_{-K_S}$ is Fano. Let now $L\in\Pic(S)$ be ample, $L\in\ma{E}$. Consider the natural birational map $\ph\colon Y\dasharrow M_L$ (see Cor.~\ref{birational}). Since $-K_S\in\Pi$, the map $\ph$ is contracting. Thus $M_L$ is log Fano by \cite[Lemma 2.8]{prokshok}. Finally, log Fano varieties are Mori dream spaces by \cite[Cor.~1.3.2]{BCHM}.
\end{proof}
\subsection{Automorphisms}
Let $\Aut(S)_L$ be the subgroup of $\Aut(S)$ of automorphisms fixing $L$ in $\Pic(S)$. There is  a natural homomorphism $\psi\colon \Aut(S)_L\to\Aut(M_L)$ defined as follows:
$$
\psi(f)([F])=[(f^{-1})^*F]\qquad\text{for every $f\in\Aut(S)_L$ and $[F]\in M_L$.}
$$   
\begin{proposition}\label{equivariant}
The determinant map $\rho\colon \Pic(S) \to \Pic(M_L)$
is equivariant for the action of $\Aut(S)_L$, namely for every $f\in\Aut(S)_L$ and $P\in \Pic(S)$ we have:
$$\rho(f^*P)=\psi(f)^*(\rho(P))\in \Pic(M_L).$$
\end{proposition}
\begin{proof}
For simplicity set $M:=M_L$.
Let $f\in\Aut(S)_L$, $g:=\psi(f)$ the induced automorphism of $M$, and $h:=(f,g)$ acting diagonally on $S\times M$.
 
Let $\mathcal{U}$ be a universal vector bundle over $M$.
Let us check that $h^*\ma{U}$ is again a universal family. 
 To this end, 
we compute the restriction over $S\times [F]$, for $[F]\in M$:
$$
(h^*\mathcal{U})_{|S\times [F]}=(h_{|S\times [F]})^*(\mathcal{U}_{|S\times g([F])})
\cong f^*(f^{-1})^*F\cong F.$$

We have a commutative diagram
{\footnotesize$$\xymatrix{{S\times M}\ar[r]^h\ar[d]^{\pi_M}&{S\times M}\ar[d]^{\pi_M}\\
M\ar[r]^{g}&M,
}$$}where the horizontal maps are isomorphism, hence
for every locally free sheaf ${G}$ on $S\times M$ and every $i\geq 0$ we have 
$$R^i(\pi_M)_*(h^*{G})\cong g^*(R^i(\pi_M)_*{G}),$$
thus in $\Pic(M)$:
$$\det\bigl((\pi_M)_![h^*{G}]\bigr)\cong g^*\bigl(\det((\pi_M)_![{G}])\bigr).$$

Therefore given a locally free sheaf $V$ on $S$ we have, by \eqref{deflambda}:
\begin{gather*}
\lambda_{h^*\ma{U}}([f^*V])=\det\bigl((\pi_M)_![\pi_S^*f^*V\otimes h^*\ma{U}]\bigr)\cong
\det\bigl((\pi_M)_![h^*(\pi_S^*V\otimes \ma{U})]\bigr)\\
\cong g^*\bigl(\det((\pi_M)_![\pi_S^*V\otimes \ma{U}])\bigr)= g^*\lambda_{\ma{U}}([V]).
\end{gather*}

The pullback via $f$ of vector bundles on $S$ induces a natural group automorphism $f^*\colon K(S)\to K(S)$, which preserves the subgroup $\mathfrak{f}^{\perp}$. Thus the equality above yields
$\lambda(f^*\mathfrak{c})=g^*\lambda(\mathfrak{c})$ for every $\mathfrak{c}\in \mathfrak{f}^{\perp}$.
To conclude, we remark that $\mathfrak{u}(f^*P)=f^*(\mathfrak{u}(P))$  for every $P\in\Pic(S)$  (see \eqref{u(P)}). 
\end{proof}
\section{The relation with the blow-up $X$ of $\pr^4$ at $8$ points -- identification of the stability fan with the Mori chamber decomposition}\label{sec_blowup} 
\subsection{Polarizations in the linear span of $h$ and $-K_S$}\label{secL_t}
In this preliminary subsection, for a fixed   cubic $h$, we
describe the chambers that intersect the plane in $H^2(S,\R)$ spanned by $h$ and $-K_S$ (see  \cite[p.~9]{mukaiADE}).
 This will be needed to describe the birational maps that relate the associated moduli spaces.
Notation as in \ref{notationP^2}; set 
$$h':=\iota_S^*h=17h-6e=-6K_S-h,$$ so that $h'$ is another   cubic (see Lemma \ref{cubics}).
Since both $h$ and $h'$ are nef and non-ample, the cone $\Nef(S)$ intersects the plane in $H^2(S,\R)$ spanned by $h$ and $-K_S$ in the cone $\langle h,h'\rangle$. For every  $t\in (0,1)\cap\Q$ consider the ample class:
$$L_t:=(1-t)h+th'=-6tK_S+(1-2t)h.$$ 
\begin{lemma}\label{L_t}
We have:
$$L_t\in\ma{E}\ \Leftrightarrow\
t\in\Big[\frac{1}{32},\frac{31}{32}\Big],\qquad L_t\in\Pi\ \Leftrightarrow\
t\in\Big[\frac{1}{20},\frac{19}{20}\Big],\qquad
L_t\in\ma{N}\ \Leftrightarrow\
t\in\Big[\frac{1}{4},\frac{3}{4}\Big],$$
and $L_t$ belongs to a wall if and only if $t\in\left\{\frac{1}{32},\frac{1}{20},\frac{1}{8},\frac{1}{4},\frac{3}{4},\frac{7}{8},\frac{19}{20},\frac{31}{32}\right\}$. More precisely:
\begin{enumerate}[$\bullet$]
\item
$L_{\frac{1}{32}}=\frac{3}{16}(-K_S+5h)\in(2h+K_S)^{\perp}$
\item
$L_{\frac{1}{20}}=\frac{3}{10}(-K_S+3h)\in(2C_i+K_S)^{\perp}$, $i=1,\dotsc,8$
\item $L_{\frac{1}{8}}=\frac{3}{4}(-K_S+h)\in(2\ell_{ij}+K_S)^{\perp}$, $1\leq i<j\leq 8$
\item $L_{\frac{1}{4}}=\frac{1}{2}(-3K_S+h)\in(2e_{i}+K_S)^{\perp}$, $i=1,\dotsc,8$
\item $L_{\frac{3}{4}}\in(2\ell+K_S)^{\perp}$ when $\ell$ is a $(-1)$-curve such that
$\ell\sim 6h-2e-e_i$, for $i=1,\dotsc,8$
\item $L_{\frac{7}{8}}\in(2\ell+K_S)^{\perp}$ when $\ell$ is a $(-1)$-curve such that
$\ell\sim 5h-2e+e_i+e_j$, with
$1\leq i<j\leq 8$ 
\item
$L_{\frac{19}{20}}\in(2C'_i+K_S)^{\perp}$, $C'_i$ the conic such that
$C_i'\sim11h-4e+e_i$, for
$i=1,\dotsc,8$
\item
$L_{\frac{31}{32}}\in(2h'+K_S)^{\perp}$
\end{enumerate}
and no other wall contains some $L_t$.
\end{lemma}
\begin{figure}[h]\caption{Chambers intersecting the plane spanned by $h$ and $-K_S$}\label{figura}\footnotesize
\begin{tikzpicture}
\draw [dotted] (0,0) -- (1,0);
\draw [dotted] (11.5,0) -- (12.5,0);
\draw (1,0) -- (2.5,0);\draw  (10,0) -- (11.5,0);
\draw [thick] (2.5,0) -- (10,0);
\draw (0,-0.1,0) -- (0,0.1);\draw (1,-0.1,0) -- (1,0.1);\draw (2.5,-0.1,0) -- (2.5,0.1);
\draw (4,-0.1,0) -- (4,0.1);\draw (5.5,-0.1,0) -- (5.5,0.1);\draw (7,-0.1,0) -- (7,0.1);
\draw (8.5,-0.1,0) -- (8.5,0.1);\draw (10,-0.1,0) -- (10,0.1);\draw (11.5,-0.1,0) -- (11.5,0.1);
\draw (12.5,-0.1,0) -- (12.5,0.1);
\node [above] at (0,0.15) {$h$};
\node [above] at (1,0) {$L_{\frac{1}{32}}$};
\node [above] at (2.5,0) {$L_{\frac{1}{20}}$};
\node [above] at (4,0) {$L_{\frac{1}{8}}$};
\node [above] at (5.5,0) {$L_{\frac{1}{4}}$};
\node [above] at (7,0) {$L_{\frac{3}{4}}$};
\node [above] at (8.5,0) {$L_{\frac{7}{8}}$};
\node [above] at (10,0) {$L_{\frac{19}{20}}$};
\node [above] at (11.5,0) {$L_{\frac{31}{32}}$};
\node [above] at (12.5,0.15) {$h'$};
\node [below] at (1.75,0) {$\ma{C}_h$};
\node [below] at (3.25,0) {$\ma{B}_h$};
\node [below] at (4.75,0) {$\ma{F}_h$};
\node [below] at (6.25,0) {$\ma{N}$};
\node [below] at (7.75,0) {$\ma{F}_{h'}$};
\node [below] at (9.25,0) {$\ma{B}_{h'}$};
\node [below] at (10.75,0) {$\ma{C}_{h'}$};
\draw[snake=brace,segment amplitude=15pt] (2.5,0.7) -- (10,0.7);
\node [above] at (6.25,1.2) {$\Pi$};
\draw[snake=brace,segment amplitude=15pt] (1,1.5) -- (11.5,1.5);
\node [above] at (6.25,2) {$\ma{E}$};
\end{tikzpicture}
\end{figure}
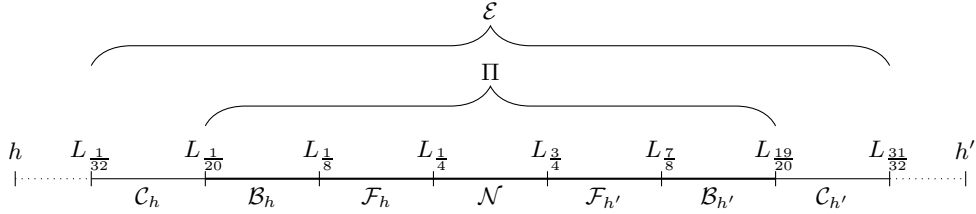
We introduce some notation for the $7$ chambers containing some $L_t$ (see Fig.~\ref{figura}). 
First we
notice that $L_{\frac{1}{32}}=\frac{3}{16}(-K_S+5h)$ is in the relative interior of the facet $\tau_h$ of $\ma{E}$  cut by $(2h+K_S)^{\perp}$ (see \ref{outer}), 
so that $L_t$ belongs to the outer chamber $\ma{C}_h$ for
$t\in(\frac{1}{32},\frac{1}{20})$. In this chamber the moduli space is isomorphic to $\pr^4$, as shown in Prop.~\ref{P^4}.
\begin{notation}[the chambers $\ma{B}_h$ and $\ma{F}_h$]\label{chamb}
Given a   cubic $h$, we denote by $\ma{B}_h$ the chamber containing $-K_S+2h=\frac{7}{3}L_{\frac{1}{14}}$, and by $\ma{F}_h$ the chamber containing 
$-2K_S+h=\frac{5}{3}L_{\frac{1}{5}}$ (see Fig.~\ref{figura}).

Notice that $\ma{B}_h\subset\Pi$. It is not difficult to see that $-K_S+3h$ and $-K_S+h$ generate one-dimensional faces of $\overline{\ma{B}}_h$, and that the hyperplanes $(2C_i+K_S)^{\perp}$ and $(2\ell_{jk}+K_S)^{\perp}$ intersect $\overline{\ma{B}}_h$ along a facet, for every
 $i=1,\dotsc,8$
and $1\leq j<k\leq 8$.
\end{notation}
\begin{proof}[Proof of Lemma \ref{L_t}]
We use Cor.~\ref{walls2}. If $\tilde{h}$ is a   cubic, and $m:=h\cdot\tilde{h}$, we have
$L_t\cdot(2\tilde{h}+K_S)=4(9-m)t+2m-3$.
Using Lemma \ref{cubics} and \eqref{E}, we get the statement for  $\ma{E}$ and
the walls 
$(2\tilde{h}+K_S)^{\perp}$.
The computation is completely analogous for ${\Pi}$ and the walls 
$(2C+K_S)^{\perp}$, and for $\ma{N}$ and the walls 
$(2\ell+K_S)^{\perp}$.
\end{proof}
\subsection{The blow-up $X$ of $\pr^4$ in $8$ points}
Let $h$ be a   cubic in $S$, and recall from \ref{bijection} that we associate to $(S,h)$ a blow-up $X=X_{(S,h)}$ of $\pr^4$ in $8$ points in general linear position.
The following  is a refined version of Th.~\ref{Xintro};  notation as in  \ref{notationP^2} and \ref{notationP^4}. 
\begin{thm}[\cite{mukaiADE}, p.~9]\label{X}
Let $S$ be a del Pezzo surface of degree $1$ and
$h$ a   cubic in $S$.  Then there is an isomorphism 
 $f\colon M_{S,\ma{B}_h}\to X_{(S,h)}$, and moreover:
$$f(E_{C_i})=E_i\quad\text{and}\quad f(Z_{\ell_{jk}})=L_{jk}\quad\text{for every }
i=1,\dotsc,8\text{ and }1\leq j<k\leq 8.$$
 \end{thm}
We can see from Fig.~\ref{figura} and Lemma \ref{L_t} that, to go from the chamber $\ma{C}_h$ to $\ma{B}_h$, one has to cross the walls $(2C_i+K_S)^{\perp}$ for $i=1,\dotsc,8$; this gives the blow-up map $M_{\ma{B}_h}\to M_{\ma{C}_h}\cong\pr^4$, see Prop.~\ref{P^4} and \ref{blowup}.
\begin{corollary}\label{SfromX}
Let  $X$ be a blow-up of $\pr^4$ at $8$ general points. Then there exists a smooth del Pezzo surface $S$ of degree one, and a   cubic $h$ on $S$, such that $X\cong  M_{S,\ma{B}_h}$.
\end{corollary}
From now on we will identify $M_{\ma{B}_h}$ and $X$ via the isomorphism $f$.
\begin{proposition}\label{explicit}
 The maps $\rho\colon H^2(S,\R)\to H^2(X,\R)$ and 
$\zeta\colon\N(X)\to H^2(S,\R)$ are isomorphisms of real vector spaces, and we have (notation as in \ref{notationP^2} and \ref{notationP^4}):
\begin{align*}
\rho(h)&=\sum_{j=1}^8E_j-H,\quad &\rho(e_i)&=-2E_i+\sum_{j=1}^8E_j-H\ \text{ for }i=1,\dotsc, 8,\\
\zeta(h)&=e-h=2h+K_S,\quad &\zeta(e_{i})&=-2e_i+e-h\ \text{ for }i=1,\dotsc, 8.
\end{align*}
\end{proposition}
\begin{proof}
The cone $\overline{\ma{B}}_h$ is contained in $\Pi$, intersects the walls $(2C_i+K_S)^{\perp}$ and $(2\ell_{ij}+K_S)^{\perp}$ along a facet, and $\ma{B}_h\subset(2\ell_{ij}+K_S)^{<0}$ (see \ref{chamb}). Thus by Th.~\ref{X} and Cor.~\ref{facet_C} and \ref{facet_ell} we have:
$$\zeta(e_{i})=2C_i+K_S=-2e_i+e-h,\quad
\zeta(h-e_i-e_j)=
\zeta(L_{ij})=-2\ell_{ij}-K_S=h-e+2e_i+2e_j
$$
which easily yields the statement.
\end{proof}
\subsection{Relating the intersection product in $S$ to Dolgachev's pairing in $X$}\label{secisometry}
We recall that $H^2(X,\Z)$ has a natural pairing, related to the action of the symmetric group $S_8$ and of the standard Cremona map, see 
\cite{dolgachev,dolgort} and also \cite{mukaiXIV}.
This pairing is defined by imposing that $H,E_1,\dotsc,E_8$ is an orthogonal basis, $H^2=3$, and $E_i^2=-1$ for $i=1,\dotsc,8$ (notation as in \ref{notationP^4}).  The sublattice 
 $K_X^{\perp}$ is an $E_8$-lattice; we denote by $W_X\cong W(E_8)$ its Weyl group of automorphisms.

Using Prop.~\ref{explicit}, we see that $\rho^{-1}\colon H^2(X,\R)\to H^2(S,\R)$ coincides with $\frac{1}{2}$ the linear map $\ph$ defined by Mukai in \cite[p.~7]{mukaiADE}. In particular, we get the following.
\begin{lemma}[\cite{mukaiADE}, p.\ 7]\label{isometry}
 Set $\tilde{\rho}:=\frac{1}{2}\rho\colon H^2(S,\R)\to H^2(X,\R)$. We have $\tilde{\rho}(K_S^{\perp}\cap H^2(S,\Z))=K_X^{\perp}\cap H^2(X,\Z)$, and the restriction $\tilde{\rho}\colon K_S^{\perp}\cap H^2(S,\Z)\to K_X^{\perp}\cap H^2(X,\Z)$
is an isometry.
Moreover
there is a group isomorphism $\phi\colon W_S\to W_X$ such that $\rho$ and $\tilde{\rho}$ are equivariant with respect to $\phi$.
\end{lemma}
\begin{remark}\label{integral}
The relation between integral points in $H^2(S,\R)$ and $H^2(X,\R)$ via $\tilde{\rho}$ is the following: 
$$\tilde{\rho}^{-1}(H^2(X,\Z))=\{L\in H^2(S,\Z)\,|K_S\cdot L\text{ is even}\}.$$ 
Indeed, we have $\tilde{\rho}(K_S)=\frac{1}{2}K_X$ by Lemma \ref{-K}. 
Let  $L\in H^2(S,\R)$ and set $m:=K_S\cdot L$; we have 
$L-mK_S\in K_S^{\perp}$ and  $\tilde{\rho}(L)=\tilde{\rho}(L-mK_S)+\frac{1}{2}mK_X$. 

If $L$ is integral and $m$
 is even, then  $\tilde{\rho}(L-mK_S)\in H^2(X,\Z)$ by Lemma \ref{isometry}, so  $\tilde{\rho}(L)\in H^2(X,\Z)$. 

Conversely, suppose that $\tilde{\rho}(L)$ is integral. It is not difficult to check that $K_X$ and $K_X^{\perp}\cap H^2(X,\Z)$ generate $H^2(X,\Z)$ as a group, and since $\tilde{\rho}(L-mK_S)\in K_X^{\perp}$, we see that $m$ must be an even integer and $\tilde{\rho}(L-mK_S)$ must be integral. By Lemma \ref{isometry}, $L-mK_S$ is integral, and hence $L$ is integral.
\end{remark}
\subsection{Cones of divisors and chamber decompositions}\label{sec_cones}
We recall that if $M$ is a Mori dream space, then in $\Nu(M)$ the convex cones $\Eff(M)$,  $\Mov(M)$, and $\Nef(M)$
 (respectively of effective, movable, and nef divisors)  are all closed and polyhedral. Moreover there is a fan $\MCD(M)$, supported on $\Eff(M)$, called the {\bf Mori chamber decomposition} (see  \cite{hukeel,okawa_MCD}); the cones of maximal dimension of the fan are in bijection with contracting birational maps $\ph\colon M\dasharrow M'$ (up to isomorphism of the target), where $M'$ is projective, normal and $\Q$-factorial. The cone corresponding to $\ph$ is $\ph^*\Nef(M')+\langle E_1,\dotsc,E_r\rangle$, where $E_1,\dotsc,E_r\subset M$ are the exceptional prime divisors of $\ph$. 
In particular,  $\ph\colon M\dasharrow M'$ is a pseudo-isomorphism if and only if the corresponding cone is contained in $\Mov(M)$. 

Let $\ma{C}\subset\Pi$ be a chamber. After Cor.~\ref{MDS}, we know that $M_{\ma{C}}$ is a Mori dream space. Applying the previous results,
we relate the Mori chamber decomposition $\MCD(M_{\ma{C}})$ to the stability fan $\ST(S)$ in $S$ (see \ref{wallsfan}),
via the determinant map $\rho\colon H^2(S,\R)\to H^2(M_{\ma{C}},\R)$; this is
 our main result in this section.
\begin{thm}\label{iso}
Let $\ma{C}$
be a chamber contained in $\Pi$. We have the following:
\begin{enumerate}[$(a)$]
\item
the determinant map $\rho\colon H^2(S,\R)\to H^2(M_{\ma{C}},\R)$ is an isomorphism;
\item $\rho(\overline{\ma{C}})=\Nef(M_{\ma{C}})$, $\rho({\Pi})=\Mov(M_{\ma{C}})$, and $\rho(\ma{E})=\Eff(M_{\ma{C}})$;
\item $\rho$ yields an isomorphism between the stability fan $\ST(S)$ in $H^2(S,\R)$ (see \ref{wallsfan}), and the fan $\MCD(M_{\ma{C}})$ in  $H^2(M_{\ma{C}},\R)$ given by the Mori chamber decomposition.
\end{enumerate}
\end{thm}
Before proving Th.~\ref{iso}, we need a preliminary property.
\begin{definition}[the divisor $E_C$]\label{E_C}
Let $C$ be a conic and $\ma{C}\subset\Pi$  a chamber. We generalise Def.~\ref{defiloci} and define a fixed
prime divisor $E_C\subset M_{\ma{C}}$, as follows.

If  $\overline{\ma{C}}\cap (2C+K_S)^{\perp}$ is a facet of $\overline{\ma{C}}$ and intersects the ample cone of $S$, then by
\ref{blowup} 
 $M_{\ma{C}}$ contains the divisor $E_C\cong\pr^3$, which is the exceptional divisor of a blow-up of a point. 

In general, we choose a chamber $\ma{C}'\subset\Pi$ such that $\overline{\ma{C}'}\cap (2C+K_S)^{\perp}$  is a facet of $\overline{\ma{C}'}$ and intersects the ample cone of $S$, so that 
 $M_{\ma{C}'}$ contains $E_C$ as an exceptional divisor.
 Let $\ph\colon M_{\ma{C}}\dasharrow M_{\ma{C}'}$ be the natural pseudo-isomorphism (see Cor.~\ref{birational}). Then we still denote by $E_C\subset M_{\ma{C}}$ the strict transform of $E_C$ under $\ph$.
\end{definition}
\begin{lemma}\label{E_C=}
In the setting of Def.~\ref{E_C}, we have $\rho(C)=2E_C$ in $H^2(M_{\ma{C}},\R)$.
\end{lemma}
\begin{proof}
There exists a   cubic $h$ such that $C=C_1$ (notation as in \ref{notationP^2}). Consider the chamber $\ma{B}_h\subset\Pi$ and $\rho_{\ma{B}_h}\colon H^2(S,\R)\to H^2(M_{\ma{B}_h},\R)$.
By Prop.~\ref{explicit} and Th.~\ref{X} we have
$\rho_{\ma{B}_h}(C)=\rho_{\ma{B}_h}(h-e_1)=2E_1=2E_{C}$, and this yields $\rho_{\ma{C}}(C)=2E_C$ by Lemma \ref{cod1}.
\end{proof}
\begin{proof}[Proof of Th.~\ref{iso}]
We first show $(a)$ and that  $\rho(\ma{E})=\Eff(M_{\ma{C}})$.
Let $h$ be a   cubic, and consider  the chamber $\ma{B}_h\subset\Pi$. By Cor.~\ref{birational} there is a pseudo-isomorphism $\ph\colon M_{\ma{C}}\dasharrow M_{\ma{B}_h}$, and $\ph^*\colon H^2(M_{\ma{B}_h},\R)\to H^2(M_{\ma{C}},\R)$ is an isomorphism which preserves the effective cone.  Thus, by Lemma \ref{cod1}, it is enough to prove the statements
for the chamber $\ma{B}_h$. Now $\rho_{\ma{B}_h}$ is isomorphism by Prop.~\ref{explicit}, so we have $(a)$.

Let $X=X_{(S,h)}$. The cone $\Eff(X)$ is generated by the orbit $W_X\cdot E_1$, see \cite[Th.~2.7]{CoxCT}. On the other hand
 $\ma{E}$ is generated by conics, namely by the orbit $W_S\cdot C_1$ (notation as in \ref{notationP^2}).
  Since  $\rho_{\ma{B}_h}(C_1)=2E_1$ by Lemma \ref{E_C=}, we have 
  $\rho_{\ma{B}_h}(\ma{E})=\Eff(X)$ by Lemma \ref{isometry}. Thus $\rho_{\ma{C}}(\ma{E})=\Eff(M_{\ma{C}})$.

\smallskip

Let now $\ma{C}'\subset\ma{E}$ be a chamber and $\ph'\colon M_{\ma{C}}\dasharrow M_{\ma{C}'}$ the
natural birational map (see  Cor.~\ref{birational}). Since $\ma{C}\subset\Pi$, the map $\ph'$ is contracting, so it determines a Mori chamber in $H^2(M_{\ma{C}},\R)$.
Let us show that $\rho_{\ma{C}}(\ma{C}')$ is contained in this Mori chamber, namely that:
\begin{equation}\label{chambers}
\rho_{\ma{C}}(\ma{C}')\subseteq (\ph')^*\Nef  (M_{\ma{C}'})+\langle E_1,\dotsc,E_r\rangle\quad\text{in }H^2( M_{\ma{C}},\R),\end{equation}
where $E_1,\dotsc,E_r\subset M_{\ma{C}}$ are the exceptional prime divisors of $\ph'$.

Let $L\in\Pic(S)$, $L\in\ma{C}'$. By Cor.~\ref{determinant}$(c)$, $\rho_{\ma{C'}}(L)$ is ample on $M_{\ma{C}'}$, so $(\ph')^*(\rho_{\ma{C'}}(L))\in (\ph')^*\Nef(M_{\ma{C}'})$. On the other hand, if $U\subseteq M_{\ma{C}}$ is the open subset where $\ph'$ is an isomorphism, we have
$\rho_{\ma{C}}(L)_{|U}\cong  (\ph')^*(\rho_{\ma{C'}}(L))_{|U}$
by Lemma \ref{cod1}, and hence $\rho_{\ma{C}}(L)\cong(\ph')^*(\rho_{\ma{C'}}(L))+\sum_{i=1}^ra_i E_i$, where $a_i\in\Z$. 
We show that $a_i> 0$ for every $i=1,\dotsc,r$, which yields \eqref{chambers}.

Let $L_0\in\ma{C}$, and consider the segment in $H^2(S,\R)$ joining $L_0$ and $L$. By varying the polarization along the segment, we factor $\ph'\colon M_{\ma{C}}\dasharrow M_{\ma{C}'}$ in flips and blow-downs as described in \ref{flip} and \ref{blowup}:
{\footnotesize$$\xymatrix{
{M_{\ma{C}}=M_{\ma{C}_0}}\ar@{-->}[r]_(0.6){\sigma_1}\ar@{-->}@/^1pc/[rrrr]^{\ph'}&
{M_{\ma{C}_1}}\ar@{-->}[r]&{\cdots\cdots}\ar@{-->}[r]&{M_{\ma{C}_{t-1}}}\ar@{-->}[r]_(0.4){\sigma_t}&
{M_{\ma{C}_t}=M_{\ma{C}'}}
}$$}where $t\geq r$ and the sequence contains precisely $r$ blow-downs $\sigma_{i_1},\dotsc,\sigma_{i_r}$, with exceptional divisors (the transforms of) $E_1,\dotsc,E_r$.

When $\sigma_i\colon M_{\ma{C}_{i-1}}\dasharrow M_{\ma{C}_i}$ is a flip, we have
$\sigma_i^*(\rho_{\ma{C}_i}(L))=\rho_{\ma{C}_{i-1}}(L)$ by Lemma \ref{cod1}. 

Consider a blow-up $\sigma_{i_j}\colon M_{\ma{C}_{i_j-1}}\to M_{\ma{C}_{i_j}}$.
Again by Lemma \ref{cod1} we have $\rho_{\ma{C}_{i_j-1}}(L)=\sigma_{i_j}^*\rho_{\ma{C}_{i_j}}(L)+b_jE_j$, with $b_j\in\Z$.
Let 
$\Gamma_j\subset E_j\cong\pr^3\subset M_{\ma{C}_{i_j}-1}$ be a line, and let $C_j\subset S$ be the conic such that $\sigma_{i_j}$ corresponds to crossing the wall $(2C_j+K_S)^{\perp}$; by construction  $L\cdot(2C_j+K_S)<0$.
 Then by Cor.~\ref{facet_C} and the projection formula we have
$$L\cdot (2C_j+K_S)=L\cdot \zeta_{\ma{C}_{i_j-1}}(\Gamma_j)=\rho_{\ma{C}_{i_j-1}}(L)\cdot\Gamma_j=\bigl(\sigma_{i_j}^*\rho_{\ma{C}_{i_j}}(L)+b_jE_j\bigr)\cdot\Gamma_j=-b_j,$$
thus $b_j>0$. 
This shows  that $a_i>0$ for $i=1,\dotsc,r$.

\smallskip

Let us show that different maximal cones in $\ST(S)$ yield different Mori chambers. Let $\ma{C}''\subset\ma{E}$ be another chamber, and $\ph''\colon M_{\ma{C}}\dasharrow M_{\ma{C}''}$ the associated contracting birational map. If $\ph'$ and $\ph''$ have the same Mori chamber, then there exists an isomorphism $\chi\colon M_{\ma{C}'}\to M_{\ma{C}''}$ such that $\ph''=\chi\circ\ph'$. Then, by 
Cor.~\ref{birational}, there exists an open subset $U'\subseteq M_{\ma{C}'}$ such that $\codim(M_{\ma{C}'}\smallsetminus U')\geq 2$, and $\chi([F])=[F]$ for every $[F]\in U'$. By comparing two universal families for $M_{\ma{C}'}$ and $M_{\ma{C}''}$, we see that  $\chi([F])=[F]$ for every $[F]\in
M_{\ma{C}'}$. Therefore $\ma{C}'$ and $\ma{C}''$ yield the same stability condition, and hence  $\ma{C}'=\ma{C}''$.

\smallskip

Thus the image of a chamber $\ma{C}\subset\ma{E}$ is contained in a unique Mori chamber. Then, using  \eqref{chambers}
and $\rho_{\ma{C}}(\ma{E})=\Eff(M_{\ma{C}})$, it is not difficult to see that
$$\rho_{\ma{C}}(\overline{\ma{C}'})=(\ph')^*\Nef  M_{\ma{C}'}+\langle E_1,\dotsc,E_r\rangle$$
for every chamber $\ma{C}'\subset\ma{E}$,
and hence $(c)$. For $\ma{C}'=\ma{C}$ we get $\rho_{\ma{C}}(\overline{\ma{C}})=\Nef(M_{\ma{C}})$.
Finally, by Cor.~\ref{birational} $\ma{C}'\subset\Pi$ if and only if $\ph'$ is a pseudo-isomorphism, if and only if $\rho_{\ma{C}}(\ma{C'})\subset\Mov(M_{\ma{C}})$, so we get $(b)$.
\end{proof}
\begin{proof}[Proof of  Th.~\ref{isointro}]
It is
a direct consequence of  Lemma \ref{specialL} and Th.~\ref{iso}. 
\end{proof}
\subsection{From the blow-up $X$ to the Fano model $Y$}\label{secxi}
Let $h$ be a   cubic, and
let us go back to the chambers intersecting the plane spanned by $h$ and $-K_S$, described in \ref{secL_t}. We have a diagram:
{\footnotesize$$
\xymatrix{{X\cong M_{\ma{B}_h}}\ar@{-->}[r]\ar@{-->}@/^1pc/[rr]^{\xi}
&
{M_{\ma{F}_h}}\ar@{-->}[r]& Y=M_{-K_S},
}$$}where the birational maps are the natural ones (see Cor.~\ref{birational}), and we denote by $\xi\colon X\dasharrow Y$ the composition $X\stackrel{f^{-1}}{\to} M_{\ma{B}_h}\dasharrow Y$. We will occasionally write $\xi_h\colon X_h\dasharrow Y$, when we need to stress that $X_h$ and $\xi_h$ depend on the chosen   cubic $h$ (while $Y$ does not).
Notation as in \ref{notationP^4}.
\begin{lemma}\label{sequence}
The birational map $\xi\colon X\dasharrow Y$ is the composition of $36$ ($K$-positive) flips: first the flips of  $L_{ij}$ for $1\leq i<j\leq 8$, and then the flips of   $\Gamma_k$ for $k=1,\dotsc,8$. There is a commutative diagram:
{\footnotesize$$\xymatrix{&{\widehat{X}}\ar[dl]\ar[dr]&\\
{X}\ar@{-->}[rr]^{\xi}&&Y}$$}where $\widehat{X}\to X$ is the blow-up of the curves $L_{ij}$  and $\Gamma_k$, with every exceptional divisor isomorphic to $\pr^1\times\pr^2$ with normal bundle $\ma{O}(-1,-1)$, and $\widehat{X}\to Y$ is the blow-up of $36$ pairwise disjoint smooth rational surfaces.
\end{lemma}
\begin{proof}
Firstly, to go from the chamber $\ma{B}_h$ to  $\ma{F}_h$, we have to cross the 
 $\binom{8}{2}=28$ walls $(2\ell_{ij}+K_S)^{\perp}$ (see Fig.~\ref{figura} and Lemma \ref{L_t}). Moreover, by Th.~\ref{X}, the loci $Z_{\ell_{ij}}$ correspond to the curves $L_{ij}\subset X$. Therefore 
the map $X\dasharrow M_{\ma{F}_h}$ is the composition of  $28$  flips, each  replacing $L_{ij}$
with $P_{\ell_{ij}}\cong\pr^2$. 

Secondly, to go from the chamber $\ma{F}_h$ to  $\ma{N}$, we have to cross the $8$ walls  $(2e_{k}+K_S)^{\perp}$ (see again Fig.~\ref{figura} and Lemma \ref{L_t}).
Thus the second map $M_{\ma{F}_h}\dasharrow Y$ is a composition of $8$ flips, each replacing 
$Z_{e_k}\cong\pr^1$ with $P_{e_k}\cong\pr^2$, for $k=1,\dotsc,8$. 

Let $k\in\{1,\dotsc,8\}$. We claim that $Z_{e_k}\subset M_{\ma{F}_h}$ is the transform of $\Gamma_k\subset X$.  
Indeed, consider the maps $\zeta_{\ma{F}_h}\colon \N(M_{\ma{F}_h})\to H^2(S,\R)$ and
$\zeta_{\ma{B}_h}\colon \N(X)\to H^2(S,\R)$.
 The curve $\Gamma_k$ is contained in the open subset where the birational map  $X\dasharrow M_{\ma{F}_h}$ is an isomorphism, 
and if we denote by $\Gamma_k'\subset M_{\ma{F}_h}$ the transform of $\Gamma_k$,
 Rem.~\ref{cod1curves} and Prop.~\ref{explicit} yield that 
$$\zeta_{\ma{F}_h}(\Gamma_k')=\zeta_{\ma{B}_h}(\Gamma_k)=
\zeta_{\ma{B}_h}(4h-e+e_k)=
-2e_k-K_S.$$ 
On the other hand, $\overline{\ma{F}}_h$ intersects $(2e_{k}+K_S)^{\perp}$ along a wall, and by Cor.~\ref{facet_ell} we also have
$\zeta_{\ma{F}_h}(Z_{e_k})=-2e_k-K_S$.
Since $\zeta_{\ma{F}_h}$ is an isomorphism by Th.~\ref{iso}$(a)$, we deduce that 
 $\Gamma'_k$ and $Z_{e_k}$ are numerically equivalent; the class of $Z_{e_k}$ generates an extremal ray of $\NE(M_{\ma{F}_h})$ whose locus is just 
$Z_{e_k}$, hence   $\Gamma_k'=Z_{e_k}$.

Finally, the factorization of $\xi$ as a sequence of smooth blow-ups and blow-downs follows from the explicit description of the flips in \ref{flip}.
\end{proof}
\begin{corollary}\label{negativecurves}
Let $X$ be the blow-up of $\pr^4$ at $8$ general points. If $C\subset X$ is an irreducible curve with $-K_X\cdot C\leq 0$, then either $C=L_{ij}$ or $C=\Gamma_k$ for some $1\leq i<j\leq 8$, $k=1,\dotsc,8$.
\end{corollary}
\begin{proof}
This follows from Lemma \ref{sequence} and \cite[Lemma 2.8(2)]{blowup}.
\end{proof}
\section{Geometry of the Fano model $Y$}\label{fanosection}
\noindent Let $S$ be a del Pezzo surface of degree one; 
in this section we study the Fano $4$-fold 
 $Y=M_{S,-K_S}$ (see Prop.~\ref{fano}). 
\begin{proposition}[numerical invariants]\label{numinv}
We have $b_2(Y)=9$, $b_3(Y)=0$, $h^{2,2}(Y)=b_4(Y)=45$, $(-K_{Y})^4=13$, 
and 
$h^0(Y,-K_Y)=6$. Moreover $Y$ has index one. 
\end{proposition}
\begin{proof}
Let $h$ be a   cubic in $S$, and let us consider $X=X_{(S,h)}$ (see \ref{bijection}) and the birational map $\xi\colon X\dasharrow Y$ (see \ref{secxi}). 
Since $X$ is a blow-up of $\pr^4$ in $8$ points, one computes 
that $b_2(X)=b_4(X)=h^{2,2}(X)=9$, $b_3(X)=0$, and  $(-K_X)^4=625-8\cdot 81=-23$.
By the explicit description of $\xi$ as a sequence of smooth blow-ups given in Lemma \ref{sequence},
this yields the Betti and Hodge numbers of $Y$ (see for instance \cite[Th.~7.31]{voisin}). Moreover
 $(-K_{Y})^4=(-K_X)^4+36=13$, and $h^0(Y,-K_Y)=h^0(\pr^4,-K_{\pr^4})-15\cdot 8=6$ (see for instance \cite[Cor.~3.10 and Prop.~3.3]{blowup}).
Finally, it is not difficult to see that $Y$ contains curves of anticanonical degree $1$, for instance a line in a smooth rational surface in the indeterminacy locus of  $\xi^{-1}\colon Y\dasharrow X$. Therefore $Y$ has index one.
\end{proof}
We are now going to describe the relevant cones of curves and divisors in $\N(Y)$ and $H^2(Y,\R)$, using the following direct consequence of
Th.~\ref{iso}. 
\begin{corollary}\label{isoY}
The determinant map $\rho\colon H^2(S,\R)\to H^2(Y,\R)$ yields an isomorphism between:
$$\ma{N}\subset\Pi\subset\ma{E}\subset H^2(S,\R)\quad\text{and}\quad
\Nef(Y)\subset\Mov(Y)\subset\Eff(Y)\subset H^2(Y,\R).$$
Dually, the map $\zeta\colon \N(Y)\to H^2(S,\R)$ yields an isomorphism between:
$$\Mov_1(Y)\subset\NE(Y)\subset\N(Y)\quad\text{and}\quad \ma{E}^{\vee}\subset\ma{N}^{\vee}\subset H^2(S,\R).$$
\end{corollary}
\noindent Here $\Mov_1(Y)=\Eff(Y)^{\vee}$ is the convex cone generated by classes
of curves moving in a family of curves covering $Y$ (see \cite[\S 11.4.C and references therein]{lazII}).

In the following subsections we give a geometric description of the extremal rays and the facets of these cones in terms of special divisors and curves in $Y$,
using the explicit descriptions  given in \S\ref{delpezzo} of the cones $\ma{N}$, $\Pi$, $\ma{E}$ and their duals.
\subsection{The cone of effective curves and the nef cone}\label{nefcone}
The cone $\NE(Y)$  
has $240$ extremal rays, and is isomorphic to  $\NE(S)$ (see \ref{sec_N}).
 If $\ell$ is a $(-1)$-curve, the corresponding extremal ray of $\NE(Y)$ is generated by the class of a line $\Gamma_{\ell}$ in $P_{\ell}\cong\pr^2\subset Y$ (see Cor.~\ref{facet_ell}).
The corresponding elementary contraction is a
 small contraction, sending $P_{\ell}$  to a point. 
For completeness let us state here the following lemma on the relative positions of the special surfaces $P_\ell$ in $Y$; it will be proved in  \S \ref{sezX}. Recall that if  $\ell,\ell'\subset S$ are $(-1)$-curves, then $\ell\cdot\ell'\leq 3$, with equality if and only if $\ell'=\iota_S^*\ell$, see Rem.~\ref{positions}$(c)$.
\begin{lemma}\label{excplanes}
Let $\ell,\ell'\subset S$ be distinct $(-1)$-curves.
If $\ell\cdot\ell'\leq 1$, then $P_\ell\cap P_{\ell'}=\emptyset$ in $Y$.
If $\ell\cdot\ell'=2$, then $P_\ell$ and $P_{\ell'}$ intersect transversally in one point  in $Y$.

Suppose that $S$ is general. If
 $\ell\cdot\ell'=3$,
then 
$P_\ell$ and $P_{\ell'}$ intersect transversally in $3$ points  in $Y$.
\end{lemma}

The cone $\Nef(Y)$ is isomorphic to $\Nef(S)$.
It has $19440=2160+17280$ extremal rays, one for each conic $C$ and   cubic $h$ of $S$; the corresponding generators are $\rho(-2K_S+C)$ and $\rho(-3K_S+h)$ (see \eqref{N}).

Recall that extremal rays of $\Nef(Y)$ correspond to contractions $f\colon Y\to Z$ with $\rho_Z=1$. Let us describe these contractions using Lemma \ref{specialL}$(d)$ and Cor.~\ref{determinant}.

Given a   cubic $h$, the line bundle $-3K_S+h$ is contained in the walls $(2e_i+K_S)^{\perp}$, for $i=1,\dotsc,8$ (notation as in \ref{notationP^2}; see Lemma \ref{L_t}). 
Thus the contraction given by $\rho(-3K_S+h)$ is birational, small, and has exceptional locus  $P_{e_1}\cup\cdots\cup P_{e_8}$, where the $P_{e_i}$'s are pairwise disjoint. Correspondingly, the classes of the curves $\Gamma_{e_1},\dotsc,\Gamma_{e_8}$ span a simplicial facet of $\NE(Y)$.

Given a conic $C$, the line bundle $-2K_S+C$ is contained in 14 walls $(2\ell+K_S)^{\perp}$, where $\ell$ is a $(-1)$-curve such that $C\cdot\ell=0$, namely  $\ell$ is a component of a reducible conic linearly equivalent to $C$.
 Thus the contraction given by
$\rho(-2K_S+C)$ is birational, small, and has exceptional locus the disjoint union of $14$ $P_{\ell}$'s (all contained in the divisor $E_C$, see \ref{effcone}). This yields a non-simplicial facet of $\NE(Y)$.

This shows Prop.~\ref{NEintro} from the Introduction.
\subsection{The cone of effective divisors}\label{effcone}
The cone $\Eff(Y)$ has 2160 extremal rays, each generated by a fixed divisor $E_C$, where $C\subset S$ is a conic.
 Each such divisor comes, up to pseudo-isomorphism, from the blow-up of a smooth point.

The divisor $E_C\subset Y$ is smooth and is isomorphic to the blow-up of $\pr^3$ in $14$ points (in a special position). This can be seen by choosing a   cubic $h$ such that $C=C_1\sim h-e_1$ (notation as in \ref{notationP^2}), so that $E_C$ is the transform of the exceptional divisor $E_1\cong\pr^3\subset X_h=X$ under $\xi\colon X\dasharrow Y$ (see Th.~\ref{X}). By the explicit description of the map 
$\xi$ given in Lemma \ref{sequence}, we see that $E_1$ is blown-up in the $14$ points of intersection with $L_{12},\dotsc,L_{18},\Gamma_2,\dotsc,\Gamma_8$.

Recall that for every $(-1)$-curve $\ell$, $\Gamma_\ell\subset P_{\ell}\cong\pr^2\subset Y$ generates an extremal ray of $\NE(Y)$. By Lemmas \ref{E_C=} and \ref{facet_ell} we have
\begin{equation}\label{El}
E_C\cdot\Gamma_\ell=\frac{1}{2}\rho(C)\cdot\Gamma_\ell
=\frac{1}{2}C\cdot\zeta(\Gamma_\ell) =\frac{1}{2}C\cdot (2\ell+K_S)=C\cdot\ell-1,
\end{equation}
and there are $14$ special loci $P_\ell$ (given by the $(-1)$-curves $\ell$ with $C\cdot\ell=0$) contained in $E_C$; these are in $E_C$ the exceptional divisors of the blow-up $E_C\to\pr^3$.
\subsection{The cone of movable divisors and the divisors $H_{Y,h}$}\label{movable}
The cone $\Mov(Y)$ is isomorphic to the cone $\Pi\subset H^2(S,\R)$ via $\rho$, and it has two types of facets, cut by $\zeta^{-1}(\ell)^{\perp}$ and  $\zeta^{-1}(2C+K_S)^{\perp}$ for every $(-1)$-curve $\ell$ and conic $C$ in $S$ (see \ref{secPi}). The class $\zeta^{-1}(2\ell)$ is the class of a moving curve on $Y$, we will describe it in \ref{moving}. The class  $\zeta^{-1}(2C+K_S)$ is the class of the transform $\Gamma_C\subset E_C$ of a general line in $\pr^3$ under the blow-up $E_C\to\pr^3$, see Cor.~\ref{facet_C}.
\begin{definition}[the map $\eta_h$]\label{eta}
Let $h$ be a   cubic in $S$, 
and consider the outer chamber $\ma{C}_h$ introduced in \ref{outer}; by Prop.~\ref{P^4}, we have $M_{\ma{C}_h}\cong\pr^4$.
Thus the natural contracting birational map $Y=M_{-K_S}\dasharrow M_{\ma{C}_h}$ (see Cor.~\ref{birational}) yields a map
$\eta_h\colon Y\dasharrow \pr^4$.
By varying the polarization from $-K_S$ to $\ma{C}_h$ along the plane spanned by $-K_S$ and $h$ (see Fig.~\ref{figura}), we factor $\eta_h$ as
{\footnotesize$$
\xymatrix{Y\ar@{-->}[r]_{\xi_h^{-1}}
\ar@{-->}@/^1pc/[rr]^{\eta_h}&{X_h}\ar[r]&{\pr^4}}
$$}where $\xi_h$ is described in Lemma \ref{sequence} and $X_h\to\pr^4$ is the blow-up of $8$ points. It follows from Lemma \ref{sequence} (and its proof) that the indeterminacy locus of $\eta_h$ is the union of the surfaces $P_{e_i}$ and $P_{\ell_{jk}}$ for $i=1,\dotsc,8$ and $1\leq j<k\leq 8$, and these surfaces are pairwise disjoint.
\end{definition}
\begin{proposition}\label{H_Y}
Let $h$ be a   cubic, and set
 $H_{Y,h}:=\frac{1}{2}\rho(-K_S+3h)\in H^2(Y,\R)$. Then $H_{Y,h}\in\Pic(Y)$ and is a movable class. Its complete linear system defines the contracting birational map
$\eta_h\colon Y\dasharrow\pr^4$,
with exceptional divisors $E_{C_1},\dotsc,E_{C_8}$.
The images $\eta_h(E_{C_1}),\dotsc,\eta_h(E_{C_8})$ are $8$ points in $\pr^4$, associated to the points $q_1,\dotsc,q_8\in\pr^2$ (notation as in \ref{notationP^2}).
\end{proposition}
\begin{proof}
After Prop.~\ref{explicit} we have $\rho_{\ma{B}_h}(-K_S+3h)=2H\in\Pic(X_h)$ (notation as in \ref{notationP^4}), so 
Lemma \ref{cod1} and the definition of $\eta_h$ yield
$$\rho_{-K_S}(-K_S+3h)=(\xi_h^{-1})^*(\rho_{\ma{B}_h}(-K_S+3h))=(\xi_h^{-1})^*(2H)
=\eta_h^*\ol_{\pr^4}(2).$$
Thus $H_{Y,h}=\eta_h^*\ol_{\pr^4}(1)$, and the rest of the statement follows from Th.~\ref{X}.
\end{proof}
\begin{lemma}\label{bigray}
The divisor $H_{Y,h}$ generates an extremal ray of $\Mov(Y)$, contained in the interior of $\Eff(Y)$.

Conversely, let $\tau$ be an extremal ray  of $\Mov(Y)$ lying in the interior of $\Eff(Y)$. Then there exists a   cubic $h'\subset S$ such that $H_{Y,h'}\in\tau$.
\end{lemma}
\begin{proof}
The divisor $-K_S+3h$ belongs to $\Pi$, and it lies on the facets $\Pi\cap(2C_i+K_S)^{\perp}$ of $\Pi$, for $i=1,\dotsc,8$ (notation as in \ref{notationP^2}). Since the classes $2C_i+K_S$, for $i=1,\dotsc,8$, are linearly independent, we see that $-K_S+3h$ generates an extremal ray of $\Pi$, and via the isomorphism $\rho$ we see that $H_{Y,h}$ generates an extremal ray of $\Mov(Y)$. Moreover
 $H_{Y,h}$ is big by Prop.~\ref{H_Y}. 

For the second statement,
a large enough integral divisor $D\in\tau$ defines a  contracting birational map $f\colon Y\dasharrow Y'$, where $Y'$ is $\Q$-factorial with $\rho_{Y'}=1$. The prime exceptional divisors of $f$  generate a simplicial facet of $\Eff(Y)$. By Th.~\ref{iso}$(b)$ we have $\Eff(Y)\cong\ma{E}$, and every simplicial facet of $\ma{E}$ has the form $(2h'+K_S)^{\perp}\cap\ma{E}$ for a   cubic $h'$, see \ref{secE}. The corresponding facet of $\Eff(Y)$ is generated by $E_{C'_1},\dotsc,E_{C'_8}$, thus $f\colon Y\dasharrow Y'$ and $\eta_{h'}\colon Y\dasharrow\pr^4$ have the same exceptional divisors. This means that the composition $f\circ\eta_{h'}^{-1}\colon \pr^4\dasharrow Y'$ is a pseudo-isomorphism, and hence an isomorphism. Therefore $H_{Y,h'}\in\tau$.
\end{proof}
Let us notice that the birational map $\eta_h\colon Y\dasharrow\pr^4$ allows to reconstruct the surface $S$ from $Y$: indeed, by Prop.~\ref{H_Y}, it determines the points $q_1,\dotsc,q_8\in\pr^2$ blown-up by $S\to\pr^2$.
  This is the key point for the proofs of Theorems 
 \ref{pseudo}, \ref{psautM},
\ref{moduli}, and  \ref{autY}.
\subsection{The cone of moving curves}\label{moving}  The cone 
$\Mov_1(Y)$ is isomorphic, via $\zeta^{-1}$, to $\ma{E}^{\vee}\subset H^2(S,\R)$. Thus it has $17520=17280+240$ extremal rays, generated by $\zeta^{-1}(2h+K_S)$ and $\zeta^{-1}(\ell)$ for every   cubic $h$ and $(-1)$-curve $\ell$ on $S$ (see \eqref{dualE}).
Let us describe some families of curves whose classes generate $\Mov_1(Y)$.

Let $h$ be a   cubic, and consider the birational map $\eta_h\colon Y\dasharrow\pr^4$. It follows from Prop.~\ref{explicit} and Rem.~\ref{cod1curves} that
 $\zeta^{-1}(2h+K_S)\in\N(Y)$ is the class of the transform under $\eta_h$ of a general line in $\pr^4$. The corresponding facet of $\Eff(Y)$ is simplicial, generated by the exceptional divisors $E_{C_1},\dotsc,E_{C_8}$ of $\eta_h$.

In order to describe the extremal ray generated by $\zeta^{-1}(\ell)$, we need the following.
\begin{remark}\label{P^1bundle}
Let $X$ be the blow-up of $\pr^4$ at $8$  points $p_1,\dotsc,p_8$ in general linear position.
Fix $i\in\{1,\dotsc,8\}$ and let $\ma{P}_i\subset\pr^3$ be the image of the set $\{p_1,\dotsc,\check{p}_i,\dotsc,p_8\}$ under the projection $\pi_{p_i}\colon\pr^4\dasharrow\pr^3$ from $p_i$. Let $T$ be the blow-up of $\pr^3$ at the $7$ points in $\ma{P}_i$.
There is a pseudo-isomorphism $X\dasharrow X_i$ and a $\pr^1$-bundle $X_i\to T$ extending $\pi_{p_i}$ (see \cite[Ex.~1]{mukaiADE} and \cite[Rem.~4.8]{fanomodel}). Let $\pi_i\colon X\dasharrow T$ be the composite map. Then the general fiber of $\pi_i$ is the transform in $X$ of a general line in $\pr^4$ through $p_i$, so it has class $h-e_i\in\N(X)$. 
\end{remark} 
\begin{remark}\label{-KP^3}
In the situation of Rem.~\ref{P^1bundle}, $-K_T$ is nef and big  \cite[Prop.~2.9]{blanclamy}, and
 $-K_T= 2 Q$, $Q\in\Pic(T)$. Using vanishing and Riemann-Roch one computes that $h^0(T,Q)=3$. 

Set $D:=\pi_i^*Q\in\Pic(X)$. Then $h^0(X,D)=3$, and by Prop.~\ref{explicit} we have $D\sim 2H-\sum_{j=1}^8 E_j-E_i=\frac{1}{2}\rho(-K_S+e_i)$.
\end{remark}
Let now $\ell\subset S$ be a $(-1)$-curve. Choose a   cubic $h$ such that  $\ell=e_1$ (notation as in \ref{notationP^2}), and consider $X=X_h$ and the $\pr^1$-bundle $X_1\to T$ described in Rem.~\ref{P^1bundle}. Set 
$Y_\ell:=X_1$ and $T_\ell:=T$. Then there is a pseudo-isomorphism $Y\dasharrow Y_\ell$ and a $\pr^1$-bundle $Y_\ell\to T_\ell$,   and we claim that  the general fiber $f$ of the composite map $Y\dasharrow T_\ell$ has class $2\zeta^{-1}(\ell)\in\N(Y)$.
Indeed by Prop.~\ref{explicit} we have $\zeta_{\ma{B}_h}(h-e_1)=2e_1\in H^2(S,\R)$, and
Rem.~\ref{cod1curves} yields
$\zeta_{-K_S}(f)=\zeta_{\ma{B}_h}(h-e_1)=2e_1=2\ell$.
 See \cite[p.~9-10]{mukaiADE} for a modular description of the map  $Y\dasharrow T_\ell$. 
The facet of $\Eff(Y)$ cut by $\zeta^{-1}(\ell)^{\perp}$ is generated by the 126 $E_C$'s such that $C$ is a conic disjoint from $\ell$.

In particular, we notice that $Y$ has $240$ distinct dominating families of rational curves of anticanonical degree $2$.
\subsection{Torelli type results}\label{sectorelli}
\begin{proof}[Proof of Th.~\ref{moduli}]
One implication is clear. For the other, let $f\colon Y_1\to Y_2$ be an isomorphism, and let $h_2$ be a   cubic on $S_2$.

Consider the divisor class $H_{Y_2,h_2}$ on $Y_2$. By Lemma \ref{bigray},
 $f^*H_{Y_2,h_2}$ generates an extremal ray of $\Mov(Y_1)$, lying in the interior of $\Eff(Y_1)$. Again by Lemma \ref{bigray}, there exists a   cubic $h_1$ on $S_1$ such that $H_{Y_1,h_1}$ is a positive multiple of $f^*H_{Y_2,h_2}$. Therefore we have a commutative diagram:
{\footnotesize$$
\xymatrix{{Y_1}\ar@{-->}[d]_{\eta_{h_1}}\ar[r]^{f}&{Y_2}\ar@{-->}[d]^{\eta_{h_2}}\\
{\pr^4}\ar[r]^{f'}&{\pr^4}}
$$}where $f'$ is a projective transformation. 

For $i=1,2$ let $p_1^i,\dotsc,p_8^i\in\pr^4$ be the images of the exceptional divisors of $\eta_{h_i}$, 
 $\sigma_i\colon S_i\to\pr^2$  the map induced by $h_i\in\Pic(S_i)$, and $q_1^i,\dotsc,q_8^i\in\pr^2$ the points blown-up by $\sigma_i$. Then  $p_1^i,\dotsc,p_8^i\in\pr^4$ and $q_1^i,\dotsc,q_8^i\in\pr^2$ are associated point sets by Prop.~\ref{H_Y}, and 
$p_1^1,\dotsc,p_8^1$ are projectively equivalent to $p_1^2,\dotsc,p_8^2$ by the diagram above.  We conclude that $q_1^1,\dotsc,q_8^1$ and  $q_1^2,\dotsc,q_8^2$ are projectively equivalent, and hence that $S_1\cong S_2$.
\end{proof}
\begin{proof}[Proof of Th.~\ref{pseudo}]\label{proofpseudo}
If $S_1\cong S_2$, then $M_{S_1,L_1}$ and $M_{S_2,L_2}$ are pseudo-isomorphic by Cor.~\ref{birational}. 

Conversely, suppose that  $M_{S_1,L_1}$ and $M_{S_2,L_2}$ are pseudo-isomorphic, and set $Y_i:=M_{S_i,-K_{S_i}}$ for $i=1,2$. Then, again by Cor.~\ref{birational}, there is a pseudo-isomorphism $f\colon Y_1\dasharrow Y_2$. Since $Y_1$ and $Y_2$ are Fano, $f$ must be an isomorphism (because $f^*(-K_{Y_2})=-K_{Y_1}$), hence $S_1\cong S_2$ by Th.~\ref{moduli}. 
\end{proof}
\subsection{Automorphisms and pseudo-automorphisms}\label{secauto}
\begin{proof}[Proof of Th.~\ref{autY}]\label{proofautY}
We have a natural group homomorphism $\Aut(S)\to\Aut(H^2(S,\R))$, given by $f\mapsto (f^{-1})^*$, and similarly for $Y$. Moreover, the isomorphism $\rho\colon H^2(S,\R)\to H^2(Y,\R)$ induces an isomorphism $\Aut(H^2(S,\R))\to \Aut(H^2(Y,\R))$, given by $\ph\mapsto \rho\circ \ph\circ \rho^{-1}$. These maps are related by the following diagram, which is commutative by Prop.~\ref{equivariant}:
\begin{equation}\label{c3}\begin{gathered}{\footnotesize
\xymatrix{{\Aut(S)}\ar[r]\ar[d]_{\psi}&{\Aut(H^2(S,\R))}\ar[d]^{\wr}\\
{\Aut(Y)}\ar[r]&{\Aut(H^2(Y,\R))}
}}\end{gathered}\end{equation}
The map  $\Aut(S)\to\Aut(H^2(S,\R))$ is injective \cite[Prop.~8.2.39]{dolgachevbook}, thus  $\psi$ is injective.
 To show that $\psi$ is also surjective, 
let $g\colon Y\to Y$ be an automorphism. 

The first step is to show that, up to multiply $g$ for an element in the image of $\psi$, we can assume that $g^*\colon H^2(Y,\R)\to H^2(Y,\R)$ fixes the ray $\R_{\geq 0} H_{Y,h}$ for some   cubic $h$ of $S$. 
As in the proof of Th.~\ref{moduli} (see \ref{sectorelli}), we find two   cubics $h,h'$ on $S$ such that $g^*H_{Y,h}$ is a multiple of $H_{Y,h'}$, and  a commutative diagram
\begin{equation}\label{c}\begin{gathered}{\footnotesize
\xymatrix{{Y}\ar@{-->}[d]_{\eta_{h'}}\ar[r]^{g}&{Y}\ar@{-->}[d]^{\eta_{h}}\\
{\pr^4}\ar[r]^{g'}&{\pr^4}}}\end{gathered}
\end{equation}
where $g'$ is a projective transformation. Moreover, 
if $\sigma\colon S\to\pr^2$ 
and $\sigma'\colon S\to\pr^2$ are the morphisms
induced respectively by $h$ and $h'$, 
there is a projective transformation $f'\colon\pr^2\to\pr^2$ sending the 
 the points blown-up by $\sigma'$ to the points blown-up by $\sigma$. 

 By the uniqueness of the blow-up, there exists an automorphism $f\colon S\to S$ such that the following diagram commutes:
\begin{equation}\label{c2}\begin{gathered}{\footnotesize
\xymatrix{{S}\ar[d]_{\sigma'}\ar[r]^{f}&{S}\ar[d]^{\sigma}\\
{\pr^2}\ar[r]^{f'}&{\pr^2}}}\end{gathered}
\end{equation}
and $f^*h=h'$.
Now by Prop.~\ref{equivariant} we get
$$
\psi(f)^*H_{Y,h}=\psi(f)^*\Bigl(\frac{1}{2}\rho(-K_S+3h)\Bigr)=
\frac{1}{2}\rho\bigl(f^*(-K_S+3h)\bigr)=\frac{1}{2}\rho(-K_S+3h')=H_{Y,h'},
$$
hence $(g\circ\psi(f^{-1}))^*H_{Y,h}$ is a positive multiple of $H_{Y,h}$.

We can now assume that $g^*$ fixes the ray $\R_{\geq 0}H_{Y,h}$, so that $h=h'$ in \eqref{c} and $\sigma=\sigma'$ in \eqref{c2}. Since $g^*$ also induces an automorphism of $H^2(Y,\Z)\subset H^2(Y,\R)$, we actually have $g^*H_{Y,h}=H_{Y,h}$.

Let 
$E_{C_1},\dotsc,E_{C_8}\subset Y$ be the exceptional  divisors of $\eta_h$,
$p_1,\dotsc,p_8\in\pr^4$ their images,  and $q_1,\dotsc,q_8\in\pr^2$ the (ordered) associated points, namely the points blown-up by $\sigma$.

The projective transformation $g'\colon \pr^4\to\pr^4$ fixes the set $\{p_1,\dotsc,p_8\}$, hence 
permutes the points $p_i$; let us call $\tau$ this permutation, so that
 $g^*E_{C_i}=E_{C_{\tau(i)}}$ for $i=1,\dotsc,8$.

We note that the map $\Aut(Y)\to\Aut(H^2(Y,\R))$ is injective. Indeed, suppose that $g^*=\Id_{H^2(Y,\R)}$. Then $\tau$ is the identity, and $g'$ fixes $p_i$ for every $i=1,\dots,8$. Since $p_1,\dotsc,p_8$ are in general linear position (see Rem.~\ref{OK}), we get $g'=\Id_{\pr^4}$ and hence $g=\Id_Y$ by \eqref{c}.

We carry on with the proof that $g\in\im(\psi)$.
Since $p_1,\dotsc,p_8$ and $p_{\tau(1)},\dotsc,p_{\tau(8)}$ are projectively equivalent, and $p_{\tau(1)},\dotsc,p_{\tau(8)}$ (as an ordered set of points) 
is associated to  $q_{\tau(1)},\dotsc,q_{\tau(8)}$ \cite[Ch.\ III, \S1]{dolgort},
 we conclude that also $q_1,\dotsc,q_8$ and $q_{\tau(1)},\dotsc,q_{\tau(8)}$
 are projectively equivalent. Let us call $k'$ the projective transformation of $\mathbb{P}^2$ which maps $q_i$ in $q_{\tau(i)}$ for $i=1,\dotsc,8$. This induces an automorphism $k$ of $S$.
We claim that $\psi(k)=g$; by what precedes,
 it is enough to show that $\psi(k)^*=g^*$.

Notice that $H_{Y,h},E_{C_1},\dotsc,E_{C_8}$ is a basis of $H^2(Y,\R)$, and $g^*E_{C_i}=E_{C_{\tau(i)}}$ for $i=1,\dotsc,8$. On the other hand $k^*h=h$, $k^*K_S=K_S$, and $k^*e_i=e_{\tau(i)}$ for $i=1,\dotsc,8$, hence $k^*C_i=C_{\tau(i)}$ for $i=1,\dotsc,8$. This easily implies, using Prop.~\ref{equivariant} and the map $\rho\colon H^2(S,\R)\to H^2(Y,\R)$, that $\psi(k)^*H_{Y,h}=H_{Y,h}$ and $\psi(k)^*E_{C_i}=E_{C_{\tau(i)}}$ for $i=1,\dotsc,8$, and finally that $\psi(k)^*=g^*$.

Therefore $\psi$ is an isomorphism; in particular $\Aut(Y)$ is finite, see \cite[\S 8.8.4]{dolgachevbook}.
\end{proof}
\begin{proof}[Proof of Th.~\ref{psautM}]\label{proofpsautM}
By Cor.~\ref{birational}, there is a pseudo-isomorphism $\ph\colon M_{S,L}\dasharrow Y:=M_{S,-K_S}$, which induces an isomorphism between the group of pseudo-automorphisms of $M_{S,L}$ and that of $Y$. On the other hand, being $Y$ Fano, every pseudo-automorphism of $Y$ is an automorphism. Thus the statement follows from Th.~\ref{autY}.
\end{proof}
Given a chamber $\ma{C}\subset\Pi$, it is not difficult to see that under the isomorphism given by Th.~\ref{psautM}, $\Aut(M_{S,\ma{C}})\cong\{f\in\Aut(S)\,|\,f^*\ma{C}=\ma{C}\}$. In particular, when $S$ is general, $\Aut(M_{S,\ma{C}})=\{\text{Id}\}$ unless $M_{S,\ma{C}}=Y$, because the Bertini involution $\iota_S$ fixes only the central chamber $\ma{N}$ (see \ref{bertiniS}).
\begin{definition}[the Bertini involution in $Y$]
The Bertini involution $\iota_S$ of $S$ induces an involution $\iota_Y=\psi(\iota_S)$ of $Y$, which we still call the Bertini involution; explicitly $\iota_Y\colon Y\to Y$ is given by
$\iota_Y([F])=[\iota_S^*F]$.
By Prop.~\ref{equivariant}, we have a commutative diagram:
\begin{equation}\label{bertini}\begin{gathered}{\footnotesize
\xymatrix{{H^2(S,\R)}\ar[r]^{\iota_S^*}\ar[d]_{\rho}&{H^2(S,\R)}\ar[d]^{\rho}\\{H^2(Y,\R)}\ar[r]^{\iota_Y^*}&{H^2(Y,\R).}
 }}\end{gathered}\end{equation}
\end{definition}
\subsection{Fibre-likeness}\label{MFS}
A Fano variety is  \emph{fibre-like} if it can appear as a fiber of a Mori fiber space;
this notion has been introduced and studied in \cite{CFST}. Every Fano variety with $b_2=1$ is fibre-like, while fibre-likeness becomes a rather strong condition on Fano varieties with $b_2>1$.

In the case of the Fano $4$-fold $Y$, 
our analysis of the automorphisms yields that 
the invariant part of 
$H^2(Y,\R)$ by the action of the Bertini involution $\iota_Y$ is $\R K_Y$ (see \eqref{bertini} and \ref{bertiniS}). By \cite[Th.~1.2]{CFST}, this implies the following.
\begin{proposition}
The Fano $4$-fold $Y$ is fibre-like.
\end{proposition}
 According to the authors' knowledge, this is the first explicit example of higher-dimensional non-toric smooth Fano variety with this property, which is not a product of lower dimensional varieties.
 The
symmetries of numerical cones of Fano varieties with high Picard rank were one of our original motivations for this work.
\subsection{Deformations}
\begin{lemma}\label{deform}
We have $h^0(Y,T_Y)=0$ and $h^1(Y,T_Y)=8$.
\end{lemma}
\noindent Thus $Y=M_{S,-K_S}$ varies in an $8$-dimensional family, like $S$.
For the proof, we need the following general formula.
\begin{lemma}\label{RR}
Let $Z$ be a smooth Fano $4$-fold. Then
$$h^0(Z,T_Z)-h^1(Z,T_Z)=27-5h^0(-K_Z)+K_Z^4+3b_2(Z)-h^{1,2}(Z)-h^{2,2}(Z)+3h^{1,3}(Z).$$
\end{lemma}
\begin{proof}
Since $Z$ is Fano, by Nakano vanishing we have $h^i(T_Z)=0$ for $i\geq 2$, so by Riemann-Roch $$h^0(T_Z)-h^1(T_Z)=\chi(T_Z)=\frac{1}{12}\left(2K^4-5K^2\cdot c_2-5K\cdot c_3-2\chi_{top}\right)+4.$$
Riemann-Roch for $\ol_Z(-K_Z)$ gives $K^2\cdot c_2=2(6h^0(-K)-K^4-6)$, and Riemann-Roch for $\Omega^1_Z$ gives $K\cdot c_3=2(2h^{1,2}-4b_2+h^{2,2}-4h^{1,3}-22)$; this
 yields the statement.
\end{proof}
\begin{proof}[Proof of Lemma \ref{deform}]
By Lemma \ref{RR} and  Prop.~\ref{numinv} we have $h^0(T_Y)-h^1(T_Y)=-8$. On the other hand, $\Aut(Y)$ is finite by Th.~\ref{autY}, hence $h^0(T_Y)=0$, and  $h^1(T_Y)=8$.
\end{proof}
\subsection{Other models}
Let us mention two other  interesting projective $4$-folds that are pseudo-isomorphic to $Y$.

The first is   the blow-up $W$ of $(\pr^1)^4$ in $5$ general points. There exists a pseudo-isomorphism $W\dasharrow X$, where $X$ is a blow-up pf $\pr^4$ at $8$ general points, see \cite[Remark at the end of \S 1]{mukaiTshaped}. Thus by Cor.~\ref{SfromX} and Th.~\ref{iso} there exist a del Pezzo surface $S$ of degree $1$, and a chamber $\ma{C}\subset\Pi\subset H^2(S,\R)$, such that $W\cong
M_{S,\ma{C}}$, and $W$ is pseudo-isomorphic to $Y=M_{S,-K_S}$.

For the second, let $G$ be the variety of lines contained in a smooth complete intersection of two quadric hypersurfaces in $\pr^6$. Then $G$ is a smooth Fano $4$-fold with $b_2(G)=8$, and $G$ is pseudo-isomorphic to a blow-up of $\pr^4$ in $7$ general points (see \cite{fanomodel} and references therein). Let 
$\Bl_p G$ be the blow-up of $G$ at
 a general point. As for $W$ above, there exist a del Pezzo surface $S$ of degree $1$, and  a chamber $\ma{C}'\subset\Pi\subset H^2(S,\R)$, such that 
$\Bl_p G\cong M_{S,\ma{C}'}$, and $\Bl_p G$ is pseudo-isomorphic to $Y=M_{S,-K_S}$. There is also a chamber $\ma{C}''\subset\ma{E}\subset H^2(S,\R)$ such that $G\cong M_{S,\ma{C}''}$, however  $\ma{C}''\not\subset\Pi$.
\section{Anticanonical and bianticanonical linear systems}\label{anti}
\subsection{The anticanonical linear system}\label{anticanonical}
\noindent Let $S$ be a degree one del Pezzo surface, and $Y=M_{S,-K_S}$ the associated Fano $4$-fold. In this subsection we show the first part of 
Th.~\ref{system}, namely that
the linear system $|-K_Y|$ has a base locus of positive dimension.

It is enough to prove this statement when 
$Y=M_{S,-K_S}$ is general, {\em i.e.} when
$S$ is a general del Pezzo surface of degree $1$; we will assume this throughout the subsection.
Let us also fix for the whole subsection a   cubic $h\subset S$, the corresponding blow-up $X=X_h$ of $\pr^4$ at $8$ points, and the  birational map
$\xi\colon X\dasharrow Y$ (see \ref{secxi}). We keep the notation as in \ref{notationP^4}. Notice that since $S$ is general, $X$ is a blow-up of $\pr^4$ at $8$ general points.

We analyse the base locus of $|-K_X|$. 
This contains the curves $L_{ij}$ for $1\leq i<j\leq 8$ and $\Gamma_k$ for $k=1,\dotsc,8$, because they have negative intersection with $-K_X$ (see Cor.~\ref{negativecurves}). We will show (Cor.~\ref{conclusion} and Lemma \ref{images}) that $\Bs|-K_X|$ also contains the transform $R$ of a  smooth rational quintic curve $R_4\subset\pr^4$   through $p_1,\dotsc,p_8$.

Let us recall that an elliptic normal quintic in $\pr^4$ is a smooth curve of genus one, degree $5$, not contained in a hyperplane.
\begin{lemma}[\cite{ranestadschreyer},\cite{dolgannali}]\label{scroll}
 Let $p_1,\dotsc,p_8\in\pr^4$ be general points. 
Then there is a pencil  of elliptic normal quintics in $\pr^4$ through $p_1,\dotsc,p_8$, which sweeps out a cubic scroll $W\subset\pr^4$.

Let moreover 
$q_1,\dotsc,q_8\in\pr^2$ be the associated points to  $p_1,\dotsc,p_8\in\pr^4$. Then there is a birational map $\alpha\colon W\to\pr^2$ such that  $\alpha(p_i)=q_i$ for $i=1,\dotsc,8$, $\alpha$ sends the pencil  of elliptic normal quintics to the pencil of plane cubics through $q_1,\dotsc,q_8$, and $\alpha$ is the blow-up of the ninth base point $q_0\in\pr^2$ of the pencil of plane cubics.
\end{lemma}
\begin{proof}
The first statement is \cite[Prop.~5.2]{ranestadschreyer}. 
Let $B\subset W$ be an elliptic normal quintic through $p_1,\dotsc,p_8$, and $\Lambda\subset \pr^4$ a general hyperplane. By \cite[2.4]{dolgannali}, the complete linear system $|p_1+\cdots+p_8-\Lambda_{|B}|$ on $B$ yields a map $B\to\pr^2$, embedding $B$ as a plane cubic, and sending $p_i$ to $q_i$ for $i=1,\dotsc,8$.

Recall that $W$ is isomorphic to the blow-up of $\pr^2$ at a point. Let $e\subset W$ be the $(-1)$-curve, and $f\subset W$ a fiber of the $\pr^1$-bundle on $W$. Then the blow-up $\alpha\colon W\to\pr^2$ is the map associated to the complete linear system $|e+f|$ on $W$. In $W$ we have $\Lambda_{|W}\sim e+2f$, $B\sim 2e+3f=-K_W$, and $B^2=8$, so that if $B'$ is another quintic of the pencil, $B'_{|B}=p_1+\cdots+p_8$. Thus on $B$ we have $p_1+\cdots+p_8-\Lambda_{|B}\sim(B'-\Lambda)_{|B}\sim (e+f)_{|B}$.
Moreover it is not difficult to see that the restriction map $H^0(W,\ol_W(e+f))\to
H^0(B,\ol_B((e+f)_{|B}))$ is an isomorphism;
 this yields the statement.
\end{proof}
Let $W'\subset X$ be the transform of the cubic scroll $W\subset\pr^4$. We have a diagram:
\begin{equation}\label{W}\begin{gathered}{\footnotesize
\xymatrix{&{W'\subset X}\ar[dl]_{\eta}\ar[dr]\ar[dd]^{\alpha'}&\\
{W\subset\pr^4}\ar[dr]_{\alpha}&&S\ar[dl]^{\sigma}\\
&{\pr^2}&
}}\end{gathered}\end{equation}
where $\eta\colon W'\to W$ is the blow-up of $p_1,\dotsc,p_8$, so the composition $\alpha':=\alpha\circ
\eta\colon W'\to\pr^2$ is the blow-up of $q_0,\dotsc,q_8$. Thus
$W'$ is isomorphic to the blow-up of $S$ in the base point of  $|-K_S|$, and there is an elliptic fibration $\pi\colon W'\to\pr^1$,  where the smooth fibers are the transforms of the elliptic normal quintics through
$p_1,\dotsc,p_8$ in $\pr^4$.
\begin{lemma}\label{locus}
The surface $W'\subset X$ is disjoint from $L_{ij}$ for $1\leq i<j\leq 8$ and from $\Gamma_k$ for $k=1,\dotsc,8$, and $W'$ is contained in the open subset where $\xi\colon X\dasharrow Y$ is an isomorphism. 
\end{lemma}
\begin{proof}
Consider the rational normal quartic
$\gamma_1\subset\pr^4$  through $p_2,\dotsc,p_8$, so that $\Gamma_1\subset X$ is the transform of $\gamma_1$.
To show that $W'$ is disjoint from $\Gamma_1$, we show that  $W\cap\gamma_1=\{p_2,\dotsc,p_8\}$ and that the intersection is transverse.

Let $V\subset\pr^4$ be the cone over $\gamma_1$ with vertex $p_8$. Then the $0$-cycle given by the intersection of $V$ and $W$ has degree $9$ and contains $p_2+\cdots+p_7+3p_8$, so it is $p_2+\cdots+p_7+3p_8$. Thus  set-theoretically $W\cap V=\{p_2,\dotsc,p_8\}$, and the intersection is transverse at $p_2,\dotsc,p_7$.

This shows that set-theoretically $W\cap\gamma_1=\{p_2,\dotsc,p_8\}$, and that the intersection is transverse at $p_2,\dotsc,p_7$. By considering the cone over $\gamma_1$ with vertex $p_7$, we see that the intersection is transverse also at $p_8$. 
Thus $W'\cap\Gamma_1=\emptyset$, and similarly $W'\cap\Gamma_k=\emptyset$ for $k=1,\dotsc,8$.

To show that $W'\cap L_{ij}=\emptyset$ for every $1\leq i<j\leq 8$, one proceeds in a similarly way, by considering in $\pr^4$ the intersection of $W$ with a plane through $3$ points among $p_1,\dotsc,p_8$, and showing that $W$ intersects the line $\overline{p_ip_j}$ only in  $p_i,p_j$, and that the intersection is transverse.

The last statement follows from Lemma \ref{sequence}.
\end{proof}
\begin{lemma}\label{restriction}
We have $(-K_X)_{|W'}=\ol_{W'}(R+2F)$ and $R=\Bs|(-K_X)_{|W'}|$,
where $F\subset W'$ is a fiber of the elliptic fibration, and $R\subset W'$ is a $(-1)$-curve and a section of the elliptic fibration. 
\end{lemma}
\begin{proof}
We have $-K_X=5H-3\sum_{i=1}^8E_i$, and $E_{i|W'}$ is a $(-1)$-curve in $W'$ for $i=1,\dotsc,8$. Thus $((-K_X)_{|W'})^2=25(H_{|W'})^2+9\sum_i(E_{i|W'})^2=75-72=3$.

Let $F$ be a smooth fiber of the elliptic fibration $\pi\colon W'\to\pr^1$. Since $F$ is the transform of an elliptic normal quintic in $\pr^4$ through $p_1,\dotsc,p_8$, we have
$$-K_X\cdot F=\Bigl(5H-3\sum_{i=1}^8E_i\Bigr)\cdot F=25-24=1.$$
Since $-K_{W'}\sim F$, by Riemann-Roch we get $\chi(W',(-K_X)_{|W'})=3$.

Notice that $(-K_X)_{|W'}$ has positive intersection with every curve in $W'$. Indeed,
by 
Lemma \ref{negativecurves}, there are finitely many irreducible curves in $X$ having non-positive intersection with $-K_X$, and  by Lemma  \ref{locus} these curves are disjoint from $W'$. 

By Nakai's criterion,
 $(-K_X)_{|W'}$ is ample on $W'$, and $-K_{W'}$ is nef, so by Kodaira vanishing we have $h^i(W',(-K_X)_{|W'})=h^i(W',K_{W'}-K_{W'}+(-K_X)_{|W'})=0$ for $i=1,2$, and  $h^0(W',(-K_X)_{|W'})=3$. In particular the linear system $|(-K_X)_{|W'}|$ is non-empty. 

We have $F\in|\pi^*\ol_{\pr^1}(1)|$, and every fiber of $\pi$ is integral. Thus for every irreducible curve $C\subset W'$ we have $F\cdot C\geq 0$, and $F\cdot C=0$ if and only if $C\sim F$ and $C$ is a fiber of the elliptic fibration.

Let $D\in |(-K_X)_{|W'}|$. Then $1=-K_X\cdot F=D\cdot F$ (where the first intersection is in $X$, and the second in $W'$), thus we must have  $D=R+\sum_im_iF_i$ where $R$ is an irreducible curve such that  $R\cdot F=1$, and $F_i$ are fibers of the elliptic fibration. In particular $(-K_X)_{|W'}\sim R+mF$, where $m=\sum_im_i$. 

Since $R$ is a section of $\pi$, we have $R\cong\pr^1$;  moreover $-K_{W'}\cdot R=F\cdot R=1$, hence $R$ is a $(-1)$-curve.
Since $((-K_X)_{|W'})^2=3$, we get $m=2$, hence $D=R+F_1+F_2$ (with possibly $F_1=F_2$).

 Now if  $D'\in |(-K_X)_{|W'}|$ is another divisor, we have $D'=R'+F_1'+F_2'$ as above. Since all fibers of the elliptic fibration are linearly equivalent, we get $R\sim R'$ and hence  $R=R'$, because $R$ is a $(-1)$-curve. Thus $R=\Bs|(-K_X)_{|W'}|$.
\end{proof}
\begin{corollary}\label{conclusion}
The base locus of $|-K_X|$ contains the smooth rational curve $R$, and the base locus of $|-K_Y|$ contains the
 smooth rational curve $\xi(R)$. 
\end{corollary}
\begin{proof}
It follows from Lemma \ref{restriction} that $\Bs|-K_X|\supseteq\Bs|(-K_X)_{|W'}|=R$. Moreover, by Lemma \ref{locus}, $R$ is contained in the open subset where the pseudo-isomorphism $\xi\colon X\dasharrow Y$ is an isomorphism, thus $\xi(R)$ is contained in the base locus of $|-K_Y|$.
\end{proof}
We describe the images of the curve $R\subset X$ in $\pr^4$ and in $\pr^2$ in the following lemma, whose proof is not difficult and is left to the reader.
\begin{lemma}\label{images}
Let $R_4\subset \pr^4$ and $R_2\subset \pr^2$ be the images of $R$ under $\eta\colon W'\subset X\to W\subset \pr^4$ and $\alpha'\colon W'\to\pr^2$ respectively (see \eqref{W}).

Then $R_4$ is a smooth rational quintic curve through $p_1,\dotsc,p_8$, and $R_2$  is a rational plane quartic  containing $q_1,\dotsc,q_8$ and having a triple point in $q_0$.
\end{lemma}
\begin{remark}[communicated to the authors by Daniele Faenzi and John Christian Ottem]\label{macaulay}
Faenzi and Ottem have computed
with Macaulay2 the dimension and the degree of the base locus  of the linear system of quintics in $\pr^4$ having multiplicity at least $3$ at $8$ general points; it turns out that this base locus has dimension $1$ and degree $65$. On the other hand the base locus contains the $28$ lines $\overline{p_ip_j}$, the $8$ quartics $\gamma_i$, and the quintic $R_4$, whose degrees sum to $65$. This shows that the base locus of $|-K_Y|$ is given by the smooth rational curve   $\xi(R)$, possibly union some zero-dimensional components.
\end{remark}
\begin{remark}
 The quartic $R_2\subset\pr^2$ is classically known and described in \cite[p.\ 252]{coble} and \cite[(6)]{moody}; it has the following geometrical description. Consider the pencil of plane cubics through $q_0,\dotsc,q_8$, and
let $C_\lambda$ be an element of the pencil. Consider the (projective) tangent line $T_{q_0}C_{\lambda}$  of  $C_\lambda$ at $q_0$, and let $p_\lambda$ be the third point of intersection among $C_\lambda$ and $T_{q_0}C_{\lambda}$. Then $R_2$ is the locus of the points $p_\lambda$ when $\lambda$ varies. 

The point $p_\lambda$ is related to the Bertini involution $\iota_{\pr^2}$ defined by the pencil and $q_0$, because if $x\in C_\lambda$ is a general point, then $\iota_{\pr^2}(x)$ is the third point of intersection of the line $\overline{p_\lambda x}$ with $C_\lambda$.

It is not difficult to see that, if $F_1$ and $F_2$ are the equations of two cubics of the pencil, and $L_i:=\sum_{j=0}^2\frac{\partial F_i}{\partial x_i}(q_0)x_j$ is the equation of the tangent line  at $q_0$ of the cubic defined by $F_i$, for $i=1,2$, then $R_2$ has equation
$F_1L_2-F_2L_1=0$.
\end{remark}
\begin{remark}
The curves $R$ and $F$ (notation as in Lemma \ref{restriction})  satisfy $-K_X\cdot R=-K_X\cdot F=1$ and $E_i\cdot R=E_i\cdot F=1$ for every $i=1,\dotsc,8$, so $R\equiv F$ in $X$, and since $W'$ is contained in the domain of $\xi$, we also have $\xi(R)\equiv \xi(F)$ in $Y$. 
Under the map $\zeta\colon\N(X)\to H^2(S,\R)$, we have $\zeta(R)=-K_S$ (see Prop.~\ref{explicit}).
\end{remark}
\begin{remark}
It is shown in \cite[Th.~3.1]{DP} that a nef divisor in $X$ is always base point free, and an ample divisor in $X$ is always very ample. It is interesting to note that these properties are not preserved from $X$ to $Y$.
\end{remark}
\subsection{The bianticanonical linear system}\label{bianticanonical}
\noindent Let $S$ be a degree one del Pezzo surface, and $Y=M_{S,-K_S}$ the associated Fano $4$-fold. In this subsection  we show the second part of 
Th.~\ref{system}, namely that
the linear system $|-2K_Y|$ is base point free.

Let us consider a fixed divisor $E_C\subset Y$, where 
 $C\subset S$ is a conic. Then  $E_C+\iota_Y^*E_C\in|-2K_Y|$ (see Du Val \cite[p.~201]{duval}).  Indeed we have $E_C=\frac{1}{2}\rho(C)$ by Lemma \ref{E_C=}, and $\iota_S^*C\sim-4K_S-C$ (see \eqref{w_0}), so using \eqref{bertini} we get
$$\iota_Y^*E_C=\iota_Y^*\Bigl(\frac{1}{2}\rho(C)\Bigr)=
\rho\Bigl(\frac{1}{2}\iota_S^*C\Bigr)=
\rho\Bigl(-2K_S-\frac{1}{2}C\Bigr)=-2K_Y-E_C.$$
We are going to use the divisors  $E_C+\iota_Y^*E_C$ to prove 
the statement. First we need the following intermediate result.
\begin{lemma}\label{partial}
Let $h\subset S$ be a   cubic; notation as in \ref{notationP^2}. Then
$\Bs|-2K_Y|\subseteq\bigcup_{i=1}^8 (P_{e_i}\cup P_{\iota_S^*e_i})$.
\end{lemma}
\begin{proof}
Since $E_{C_i}+E_{\iota_S^*C_i}\in|-2K_Y|$ for $i=1,\dotsc,8$, we have
\begin{equation}\label{bl}
\Bs|-2K_Y|\subseteq\bigcap_{i=1}^8(E_{C_i}\cup E_{\iota_S^*C_i})=
\bigcup_{\footnotesize\{1,\dotsc,8\}=I\sqcup J}\Bigl(\bigcap_{i\in I}E_{C_i}\cap\bigcap_{j\in J}E_{\iota_S^*C_j}\Bigr).\end{equation}

Let us consider $X=X_h$ and the birational map $\xi\colon X\dasharrow Y$.
Recall from Lemma \ref{sequence} that  $\xi$ flips the curves $L_{ij}$ for $1\leq i<j\leq 8$ and $\Gamma_i$ for $i=1,\dotsc,8$ (notation as in \ref{notationP^4}); moreover $E_{C_i}\subset Y$ is the transform of the exceptional divisor $E_i\subset X$, by Th.~\ref{X}.

In $X$
the divisors $E_1,\dotsc,E_8$ are pairwise disjoint.
Fix a partition $\{1,\dotsc,8\}=I\sqcup J$ with the cardinality of $I$ at least $3$.
  Then the flipping curves in $X$ which intersect every divisor $E_i$ with $i\in I$ are $\Gamma_j$ for $j\in J$, thus
after the flips we have 
$$\bigcap_{i\in I}E_{C_i}=\bigcup_{j\in J}P_{e_{j}}\subset Y.$$
Similarly, by considering the   cubic $\iota_S^*h$ instead of $h$, we get 
$$\bigcap_{i\in I}E_{\iota_S^*C_i}=\bigcup_{j\in J}P_{\iota_S^*e_{j}}\subset Y.$$

Thus for every partition $\{1,\dotsc,8\}=I\sqcup J$ we have
$$\bigcap_{i\in I}E_{C_i}\cap\bigcap_{j\in J}E_{\iota_S^*C_j}\subseteq\bigcup_{i=1}^8 (P_{e_i}\cup P_{\iota_S^*e_i}).$$
which together with \eqref{bl} yields the statement.
\end{proof}
We are ready to show that $|-2K_Y|$ is base point free.
First of all recall that if $\ell,\ell'$ are $(-1)$-curves in $S$, we have $\iota_S^*\ell'\sim-2K_S-\ell'$ (see \eqref{w_0}) and hence $\ell\cdot \iota_S^*\ell'=2-\ell\cdot\ell'$.

It follows from Lemma \ref{partial} that $\Bs|-2K_Y|$ is contained in the union in $Y$ of the loci $P_{\ell}$, where $\ell$ is a $(-1)$-curve in $S$. We fix a $(-1)$-curve $\ell$, and we show that $P_{\ell}\cap  \Bs|-2K_Y|=\emptyset$; this gives the statement.

Let us choose a   cubic $h$ such that $\ell=2h-e_1-\cdots-e_5$. We have $$\ell\cdot e_i=\ell\cdot \iota_S^*e_i=1\ \text{ for }i=1,\dotsc,5,\ \text{ and } \ \ell\cdot e_j=0,\ \, \ell\cdot \iota_S^*e_j=2\ \text{ for }j=6,7,8.$$ Therefore, by Lemma \ref{excplanes}, we have 
$P_\ell\cap P_{e_i}=\emptyset$ for every $i=1,\dotsc, 8$, $P_\ell\cap P_{\iota_S^*e_i}=\emptyset$ for $i=1,\dotsc, 5$, and $P_\ell\cap P_{\iota_S^*e_i}$ is a point $y_i$ for $i=6,7,8$. By Lemma \ref{partial}, we get $P_{\ell}\cap  \Bs|-2K_Y|\subseteq\{y_6,y_7,y_8\}$. 

Notice also that $\iota_S^*e_6\cdot \iota_S^*e_7=e_6\cdot e_7=0$, hence $P_{\iota_S^*e_6}\cap P_{\iota_S^*e_7}=\emptyset$, in particular $y_6\neq y_7$. Similarly one sees that the points $y_6,y_7,y_8$ are distinct.

Now let us consider the $(-1)$-curves $\ell$ and $\iota_S^*e_6$. Since $\ell\cdot \iota_S^*e_6=2$, there exists a different   cubic $h'$ of $S$ such that
$\iota_S^*e_6=e_1'$ and $\ell\sim 3h'-2e_1'-e_2'-\cdots-e_7'$ (see Rem.~\ref{positions}$(b)$). 
We have
$$\ell\cdot e_i'=\ell\cdot \iota_S^*e_i'=1\ \text{ for }i=2,\dotsc,7,\quad \ell\cdot \iota_S^*e_1'=\ell\cdot e_8'=0,\ \text{ and }\
 \ell\cdot e_1'=\ell\cdot \iota_S^*e_8'=2.$$ Thus, by Lemma \ref{excplanes}, $P_{\ell}$ is disjoint from $P_{e_2'},\dotsc,P_{e_8'},P_{\iota_S^*e_1'},\dotsc,P_{\iota_S^*e_7'}$, and intersects 
$P_{e_1'}=P_{\iota_S^*e_6}$ in $y_6$ and $P_{\iota_S^*e_8'}$ in a point $y'$. Again using Lemma \ref{partial}, we conclude that $P_{\ell}\cap  \Bs|-2K_Y|\subseteq\{y_6,y'\}$.

Finally, let us notice that
$K_S+\ell\sim e_8'-e_1'$, hence $e_8'\sim K_S+\ell+\iota_S^*e_6=\ell-K_S-e_6$.
Therefore
$\iota_S^*e_8'\cdot \iota_S^*e_7=e_8'\cdot e_7=(\ell-K_S-e_6)\cdot e_7=1$, and by Lemma \ref{excplanes} we have  $P_{\iota_S^*e_8'}\cap P_{\iota_S^*e_7}=\emptyset$, in particular $y'\neq y_7$. Similarly we see that $y'\neq y_8$, and we conclude that $y_7,y_8\not\in \Bs|-2K_Y|$. Now  repeating the argument by replacing $\iota_S^*e_6$ with $\iota_S^*e_7$, we conclude that $y_6\not\in\Bs|-2K_Y|$ and hence that $P_{\ell}\cap  \Bs|-2K_Y|=\emptyset$. 
\subsection{Open question}
Describe the fixed locus of $\iota_Y$, the quotient $Y/\iota_Y$, and the action of $\iota_Y$ on $|-K_Y|$ and $|-2K_Y|$.  
\section{Geometry of the blow-up $X$ of $\pr^4$ in $8$ points}\label{sezX}
\noindent Let $X$ be the blow-up of $\pr^4$ at $8$ general points $p_1,\dotsc,p_8$. 
In this section we  apply our previous results to study the geometry of $X$. 
\subsection{Cones of divisors and fixed divisors}\label{secfixeddiv}
By Cor.~\ref{SfromX}, there are a degree one del Pezzo surface $S$ and a   cubic $h$ on $S$ such that $X\cong M_{S,\ma{B}_h}$. The determinant map $\rho\colon H^2(S,\R)\to H^2(X,\R)$ is, in this case, a completely explicit linear isomorphism (see Prop.~\ref{explicit}), which allows to describe the relevant cones of divisors in $H^2(X,\R)$, after Th.~\ref{iso}. In particular, it is possible to write explicitly equations for $\Mov(X)$  in terms of  the coefficients of a divisor $dH-\sum_im_iE_i$; one gets one equation for each $(-1)$-curve and conic on $S$, corresponding to the generators of $\Pi^{\vee}$ (see \ref{secPi}). The same can be done, in principle, for $\Eff(X)$; however the generators of $\ma{E}^{\vee}$ given by   cubics (see \eqref{dualE}) give a very large number of equations.

Concerning the generators of the effective cone, it follows from Th.~\ref{iso} and Lemma \ref{E_C=} that they are given by the fixed divisors  $E_C=\frac{1}{2}\rho(C)$, where $C$ is a conic in $S$. Using Prop.~\ref{explicit},
one computes that if $C\sim dh-\sum_{i}m_ie_i$, then the corresponding fixed divisor $E_C$ has class:
\begin{equation}\label{classE_C}
E_C\sim \frac{1}{2}\bigl(\sum_im_i-d\bigr)\bigl(H-\sum_iE_i\bigr)+\sum_im_iE_i.
\end{equation}
This proves Prop.~\ref{fixeddiv}; see also \cite{lespark} for related results.

Let us give the first examples of fixed divisors in $X$. If 
 $d=1$, one gets the $E_i$'s; otherwise, if  $d\geq 2$, $E_C\subset X$ is the transform of a hypersurface $D_C\subset\pr^4$ of degree $\frac{1}{2}(\sum_im_i-d)$.
For $d=2$, $D_C$ is a hyperplane through $4$ blown-up points, and for $d=3$ a quadric cone through $7$ blown-up points, with vertex a line through two of the points. 

When $d=4$, there are two types of  conics:
$C_1\sim 4h-e-2e_i$, with $i\in\{1,\dotsc,8\}$, and
$C_2\sim 4h-\sum_{i\in I}e_i-e+e_k$, with $I\subset\{1,\dotsc,8\}$, $|I|=3$, and $k\not\in I$. We have $E_{C_1}\sim 3H-2\sum_{j\neq i}E_j$, so that
$D_{C_1}$ is the  secant variety of the rational normal quartic $\gamma_i$. To describe $D_{C_2}$, let  $\pi_{p_k}\colon \pr^4\dasharrow \pr^3$ be the projection from $p_k$, and set $p_i':=\pi_{p_k}(p_i)$ for $i\neq k$. Then
 $D_{C_2}$ is the cone, with vertex $p_k$, over a Cayley nodal cubic surface in $\pr^3$, containing the $7$  points $p_i'$ for $i\neq k$, and having the $4$ nodes in $p'_j$ for $j\in \{1,\dotsc,8\}\smallsetminus(I\cup\{k\})$.

More generally, whenever the conic $C$ is such that $m_i=0$ for some $i\in\{1,\dotsc,8\}$, $D_C$ is a cone with vertex $p_i$, indeed it follows from \eqref{classE_C} that $D_C$ has in $p_i$ a singular point of multiplicity equal to its degree. 

\smallskip

Let $\ell$ be a $(-1)$-curve, and 
set $D_\ell:=\frac{1}{2}\rho(-K_S+\ell)\in H^2(X,\R)$; by Rem.~\ref{-KP^3}, $D_\ell$ is the class of an integral divisor in $X$, and $h^0(X,D_\ell)=3$.
\begin{lemma}\label{integralgenerators}
The semigroup $\Eff(X)_{\text{sg}}:=\{L\in H^2(X,\Z)\,|\,h^0(X,L)>0\}$ is generated by the $2401$ classes
 $-K_X$, $E_C$, and $D_\ell$, where $C$ is a conic and $\ell$ is  a $(-1)$-curve.
\end{lemma}
\begin{proof}
By \cite[Th.~2.7]{CoxCT} the semigroup  $\Eff(X)_{\text{sg}}$ is generated by its elements $L$ such that $-K_X\cdot L=3$ with respect to Dolgachev's pairing in $H^2(X,\Z)$, see \ref{secisometry}. Let $L$ be such an element; in particular $L\in\Eff(X)$.

Recall the isomorphism $\tilde{\rho}=\frac{1}{2}\rho\colon H^2(S,\R)\to H^2(X,\R)$ defined in Lemma \ref{isometry}, and consider $L':=\tilde{\rho}^{-1}(L)\in H^2(S,\R)$. We have
$L'\in H^2(S,\Z)$ by
Rem.~\ref{integral},  and $L'\in
\ma{E}$ by Th.~\ref{iso}$(b)$, so that $L'$ is nef (see \ref{secE}). 
Finally, using Lemma \ref{isometry}, it is not difficult to see that $-K_S\cdot L'=\frac{2}{3}(-K_X\cdot L)=2$. 

Therefore, by Rem.~\ref{easy}, 
$L'$ is one of the classes $-2K_S,C,-K_S+\ell$, where  $C$ is a conic and $\ell$ is  a $(-1)$-curve, and hence $L$ is one of the classes $-K_X=\tilde{\rho}(-2K_S)$, $E_C=\tilde{\rho}(C)$, and $D_\ell=\tilde{\rho}(-K_S+\ell)$. Conversely, all these classes are effective and $-K_X\cdot (-K_X)=-K_X\cdot E_C=-K_X\cdot D_\ell=3$, so they are all generators for  $\Eff(X)_{\text{sg}}$.
\end{proof}
\subsection{Special surfaces}\label{secspecialsurfaces}
Let $S$ be a del Pezzo surface of degree one, and $Y=M_{S,-K_S}$ the associated Fano $4$-fold. Consider a   cubic
$h\subset S$ and the map $\eta_h\colon Y\dasharrow \pr^4$ (see \ref{eta}); we have a factorization:
{\footnotesize$$
\xymatrix{Y\ar@{-->}[r]_{\xi_h^{-1}}
\ar@{-->}@/^1pc/[rr]^{\eta_h}&{X_h}\ar[r]&{\pr^4}
}
$$}and the indeterminacy locus of  $\eta_h$ is the
 union of the surfaces $P_\ell$ for the $(-1)$-curves $\ell\subset S$ such that $h\cdot\ell\leq 1$.
\begin{notation}
For every $(-1)$-curve $\ell$ with $h\cdot\ell\geq 2$, we set $V_{h,\ell}:=\overline{\eta_h(P_\ell)}\subset\pr^4$. 

We denote by
$\w{P}_\ell\subset X_h$ the transform of $P_\ell\subset Y$ under $\xi_h\colon X_h\dasharrow Y$, so that $V_{h,\ell}\subset\pr^4$ is the image of $\w{P}_\ell\subset X_h$ under $X_h\to\pr^4$.

We denote by $\w{\Gamma}_\ell\subset \w{P}_\ell\subset X_h$ the transform of a general line $\Gamma_\ell\subset P_\ell\subset Y$.
\end{notation}
Let $p_1,\dotsc,p_8\in\pr^4$ be the images of the exceptional divisors of $\eta_h\colon Y\dasharrow \pr^4$.
Together with the curves $\overline{p_ip_j}$ for $1\leq i<j\leq 8$ and $\gamma_i$ for $i=1,\dotsc,8$ (notation as in \ref{notationP^4}), and with the images in $\pr^4$ of the fixed divisors in $X_h$ described in \eqref{classE_C},
the surfaces $V_{h,\ell}\subset \pr^4$ appear naturally in the base loci of linear systems in $\pr^4$ with assigned multiplicities at $p_1,\dotsc,p_8$. We describe the degree and singularities of the special surfaces $V_{h,\ell}$ in Th.~\ref{specialsurfaces} below.

Let us first recall that an isolated surface singularity is of type $\frac{1}{3}(1,1)$ if it is analytically isomorphic to the vertex of the cone over a   cubic;
a normal surface singularity is  of type $\frac{1}{3}(1,1)$ if and only if its
minimal resolution has exceptional divisor a smooth rational curve $E$ with $E^2=-3$.
\begin{thm}\label{specialsurfaces}
Let $h$ be a   cubic and $\ell$ a $(-1)$-curve with $h\cdot\ell\geq 2$. 
\begin{enumerate}[$(a)$]
\item
If $h\cdot\ell=2$, we have $\ell\sim 2h-\sum_{j\not\in I}e_j$
with
$I\subset\{1,\dotsc,8\}$, $|I|=3$. Then
 $V_{h,\ell}$ is the plane through $p_i$ for $i\in I$.
\item
If $h\cdot\ell=3$, we have $\ell\sim 3h-e-e_i+e_j$ with
 $i,j\in\{1,\dotsc,8\}$, $i\neq j$. Then
 $V_{h,\ell}$ is the  cone over $\gamma_i$ with vertex $p_j$.
\item
If $h\cdot\ell=4$, we have $\ell\sim 4h-e-\sum_{i\in I}e_i$ with
$I\subset\{1,\dotsc,8\}$, $|I|=3$.
 Then
 $V_{h,\ell}$ is a normal surface of degree $6$ containing $p_1,\dotsc,p_8$, $\Sing(V_{h,\ell})= \{p_j\}_{j\not\in I}$, and  $V_{h,\ell}$ has a singularity
 of type $\frac{1}{3}(1,1)$ in $p_{j}$
for every $j\not\in I$.
\setcounter{long}{\value{enumi}}
\end{enumerate}
 Suppose that $S$ is general. 
\begin{enumerate}[$(a)$]
\setcounter{enumi}{\value{long}}
\item
If
  $h\cdot\ell=5$, we have $\ell\sim 5h-2e+e_i+e_j$ with
$i,j\in\{1,\dotsc,8\}$, $i<j$.
 Then
 $V_{h,\ell}$ is a surface of degree $10$ with $\Sing(V_{h,\ell})=\{p_k\}_{k\neq i,j}\cup\overline{p_ip_j}$, 
$V_{h,\ell}$ has a singularity
 of type $\frac{1}{3}(1,1)$ in  $p_{k}$
for every $k\neq i,j$,
and $V_{h,\ell}$ has multiplicity $3$ at the general point of the line $\overline{p_ip_j}$.
\item
If $h\cdot\ell=6$, we have $\ell\sim 6h-2e-e_i$ with
$i\in\{1,\dotsc,8\}$.
 Then
 $V_{h,\ell}$ is a surface of degree $15$ with 
$\Sing(V_{h,\ell})=\{p_i\}\cup\gamma_i$,
$V_{h,\ell}$ has a singularity
 of type $\frac{1}{3}(1,1)$ in 
 $p_{i}$, and $V_{h,\ell}$ has multiplicity $3$ at the general point of $\gamma_i$.
\end{enumerate}
\end{thm}
To prove Th.~\ref{specialsurfaces}, we first determine the numerical class of $\w{\Gamma}_\ell\subset X_h$ in Lemma \ref{Gamma}. We use this Lemma to show  Th.~\ref{specialsurfaces} $(a)$ and $(b)$, and this is used to prove Lemma \ref{excplanes} on the relative positions of the surfaces $P_\ell$ in $Y$. Finally we use Lemma \ref{excplanes} to prove the rest of Th.~\ref{specialsurfaces}.
\begin{lemma}\label{Gamma}
Let $h$ be a   cubic, and $\ell\sim dh-\sum_im_ie_i$ a $(-1)$-curve with $d\geq 2$. Then
$$\w{\Gamma}_\ell\sim \Bigl(6d-5-\sum_im_i\Bigr)h-\sum_i(d-m_i-1)e_i\quad
\text{in }\,\N(X_h).$$
\end{lemma}
\begin{proof}
Since $d=h\cdot\ell\geq 2$, $P_\ell\subset Y$ is not contained in  the indeterminacy locus of $\xi_h^{-1}\colon Y\dasharrow X_h$. Therefore 
 $P_\ell$ can intersect the indeterminacy locus of $\xi_h^{-1}$ at most in finitely many points (see for instance \cite[Rem.~2.9]{blowup}), and $\Gamma_\ell$ is contained in the open subset of $X_h$ where $\xi_h^{-1}$ is an isomorphism. By Rem.~\ref{cod1curves} and
 Cor.~\ref{facet_ell} we have
\begin{align*}
\zeta_{\ma{B}_h}(\w{\Gamma}_\ell)&=\zeta_{-K_S}(\Gamma_\ell)=2\ell+K_S\sim
(2d-3)h-\sum_i(2m_i-1)e_i\\
&=\zeta_{\ma{B}_h}\Bigl(\bigl(6d-5-\sum_im_i\bigr)h-\sum_i(d-m_i-1)e_i\Big),
\end{align*}
where the last equality follows from
 Prop.~\ref{explicit}. Since $\zeta_{\ma{B}_h}$ is an isomorphism by Th.~\ref{iso}$(a)$, we get the statement.
\end{proof}
\begin{proof}[Proof of Th.~\ref{specialsurfaces} $(a)$ and $(b)$]
If $\ell\sim 2h-\sum_{j\not\in I}e_j$, it follows from Lemma \ref{Gamma} that 
$\w{\Gamma}_{\ell}\subset X_h$ is the transform of a conic in $\pr^4$ passing through  $p_i$ for $i\in I$, which gives $(a)$.

For $(b)$, set for simplicity $i=1$ and $j=8$; then Lemma \ref{Gamma} yields
$\w{\Gamma}_{\ell}\sim 5h-e_2-\cdots-e_7-2e_8$.  Let $B\subset \pr^4$ be the image of $\w{\Gamma}_{\ell}\subset X_h$, and
let $\pi_{p_8}\colon\pr^4\dasharrow\pr^3$ be the projection from $p_8$. Then both $\gamma_1$ and $B$  have image, under $\pi_{p_8}$, the rational normal cubic through $\pi_{p_8}(p_2),\dotsc,\pi_{p_8}(p_7)$ in $\pr^3$, and the cone over $\gamma_1$ with vertex $p_8$ is (the closure of) the inverse image of this cubic.
Thus $B$ is contained in the cone, and $V_{h,\ell}$ coincides with the cone.
\end{proof}
\begin{proof}[Proof of Lemma \ref{excplanes}]
Suppose that $\ell\cdot\ell'=0$. Then there exists a   cubic $h$ such that $h\cdot \ell=h\cdot\ell'=0$ (see Rem.~\ref{positions}$(a)$). 
The line bundle $L_0=-3K_S+h$ is ample on $S$, lies on the boundary of the cone $\ma{N}$, and is contained in both the walls $(2\ell+K_S)^{\perp}$ and  $(2\ell'+K_S)^{\perp}$. By Lemma \ref{specialL}$(d)$, the surfaces $P_{\ell}$ and $P_{\ell'}$ are disjoint in $Y=M_{\ma{N}}$.

Suppose that $\ell\cdot\ell'=1$. Then $\ell+\ell'$ is linearly equivalent to a conic $C$. Similarly as before, $L_1=-2K_S+C$ is ample on $S$, lies on the boundary of  $\ma{N}$, and is contained in the walls $(2\ell+K_S)^{\perp}$ and  $(2\ell'+K_S)^{\perp}$, so $P_\ell\cap P_{\ell'}=\emptyset$ by  Lemma \ref{specialL}$(d)$.

Assume that $\ell\cdot\ell'=2$.
Then there exists  a   cubic $h'$ such that $\ell=\ell'_{12}\sim h'-e'_1-e'_2$ and $\ell'\sim 2h'-e'_4-\cdots-e'_8$ (see Rem.~\ref{positions}$(b)$); let us consider the birational map $\xi_{h'}^{-1}\colon Y\dasharrow X_{h'}$. Then $P_\ell=P_{\ell'_{12}}\subset Y$ is contained in the indeterminacy locus of $\xi^{-1}_{h'}$ (by Lemma \ref{sequence}), while
$P_{\ell'}\subset Y$ is the transform of the plane $\Lambda=V_{h',\ell'}$ through the points\footnote{Note that the points $p_i$ here depend on $h'$, as they are the images of the exceptional divisors of $\eta_{h'}\colon Y\dasharrow \pr^4$. For simplicity we still denote them by $p_1,\dotsc,p_8$, and similarly in the sequel of the proof.}
 $p_1,p_2,p_3\in \pr^4$ (by Th.~\ref{specialsurfaces}$(a)$). Moreover it follows from the explicit factorization of $\xi_{h'}$ given in Lemma \ref{sequence} that the induced birational map $\Lambda\dasharrow P_{\ell'}$ is a Cremona map centered in $p_1,p_2,p_3$. The corresponding $3$ points in $P_{\ell'}$ are the intersection points with $P_{\ell'_{12}}$,  $P_{\ell'_{13}}$, $P_{\ell'_{23}}$, and the intersection is transverse.

Finally suppose that $S$ is general and that $\ell\cdot\ell'=3$.
Recall from Rem.~\ref{positions}$(c)$  that $\ell'\sim-2K_S-\ell$. Let us choose a   cubic $h''$ such that $\ell\sim 3h''-e''-e''_1+e''_2$, and hence $\ell'\sim 3h''-e''-e''_2+e''_1$.

By 
Th.~\ref{specialsurfaces}$(b)$, the surfaces $V_{h'',\ell}$ and $V_{h'',\ell'}$ in $\pr^4$
 are, respectively: the cone over $\gamma_1$ with vertex $p_2$, and the cone over $\gamma_2$ with vertex $p_1$. 
For a general choice of $p_1,\dotsc,p_8$, $V_{h'',\ell}$ and $V_{h'',\ell'}$ are general cones over two general   cubics  contained in a hyperplane $H\subset\pr^4$, and they intersect transversally at $9$ points, including $p_3,\dotsc,p_8$.
Thus the transforms $\w{P}_\ell$ and $\w{P}_{\ell'}$ of $V_{h'',\ell}$ and $V_{h'',\ell'}$ respectively in $X_{h''}$  intersect transversally in $3$ points $x_1,x_2,x_3$. 

It is not difficult to check that  $\w{P}_\ell$ intersects the indeterminacy locus of  $\xi_{h''}\colon X_{h''}\dasharrow Y$ in $L_{23},\dotsc,L_{28},\Gamma_1$, and similarly 
$\w{P}_{\ell'}$ intersects the indeterminacy locus of  $\xi_{h''}$ in $L_{13},\dotsc,L_{18},\Gamma_2$. The curves $L_{ij}$ and $\Gamma_a$ are pairwise disjoint, so we conclude that $x_1,x_2,x_3\in X_{h''}$ are contained in the open subset where $\xi_{h''}\colon X_{h''}\dasharrow Y$ is an isomorphism. Hence $P_\ell$ and $P_{\ell'}$ intersect transversally in $3$ points $\xi_{h''}(x_1),\xi_{h''}(x_2),\xi_{h''}(x_3)$. 
\end{proof}
\begin{proof}[Proof of Th.~\ref{specialsurfaces} $(c)$, $(d)$, and $(e)$]
We show $(c)$;
set for simplicity $I=\{1,2,3\}$.
By Lemma \ref{excplanes}, $P_\ell$ meets the indeterminacy locus of $\eta_h\colon Y\dasharrow\pr^4$ in $13$ isolated points:
$$x_i:=P_{e_i}\cap P_\ell\ \text{ for }i=1,2,3\quad\text{and}\quad
y_{ab}:=P_{\ell_{ab}}\cap P_\ell\ \text{ for }4\leq a<b\leq 8$$
(recall that the components of  the indeterminacy locus of $\eta_h$ are pairwise disjoint in $Y$, see \ref{eta}).
By the description of the map $\xi_h$ as a sequence of smooth blow-ups in Lemma \ref{sequence}, we have a diagram:
{\footnotesize$$\xymatrix{&{\w{P}_\ell\subset X_h}\ar[dl]\ar[dr]&\\
{V_{h,\ell}\subset\pr^4}&&{P_\ell\subset Y}
}$$}where $\w{P}_\ell\to P_\ell$ is  the blow-up of  $\pr^2$ in the points $x_i$ and $y_{ab}$, with  exceptional curves $\Gamma_i$ and  $L_{ab}$,
for  $i=1,2,3$ and $4\leq a<b\leq 8$. The second morphism $\w{P}_{\ell}\to V_{h,\ell}$ is the restriction of $X_h\to\pr^4$, thus it is induced by $H_{|\w{P}_\ell}$ and contracts the curve $(E_i)_{|\w{P}_\ell}$ to $p_i$ for $i=1,\dotsc,8$. In particular, we see that $V_{h,\ell}\smallsetminus\{p_1,\dotsc,p_8\}$ is smooth.

Recall that  $\w{\Gamma}_\ell\subset \w{P}_\ell$ is the transform of a general line in $P_\ell\cong\pr^2$, and 
$H\cdot\w{\Gamma}_\ell=8$ by Lemma \ref{Gamma}. Since $\Gamma_i$ and $L_{ab}$ are the transforms respectively of a quartic and a line in $\pr^4$, in $X_h$ we have
$H\cdot\Gamma_i=4$ and $H\cdot L_{ab}=1$, and in $\w{P}_\ell$ we have
$$H_{|\w{P}_{\ell}}\sim 8\w{\Gamma}_\ell-4\sum_{i=1}^3\Gamma_i-\sum_{4\leq a<b\leq 8}L_{ab}.$$
Hence the degree of $V_{h,\ell}$  in $\pr^4$ is $(H_{|\w{P}_{\ell}})^2=6$.

Let $i\in\{1,\dotsc,8\}$. In $X_h$ the divisor $E_i$ intersects $\Gamma_j$ for $j\neq i$ and $L_{ib}$ for $b\neq i$, while it is disjoint from $\Gamma_i$ and from $L_{ab}$ for $a,b\neq i$. Thus in $Y$ its transform, that we still denote by $E_i$, contains $P_{e_j}$ for $j\neq i$ and $P_{\ell_{ib}}$ for $b\neq i$, while it is disjoint from $P_{e_i}$ and from $P_{\ell_{ab}}$ for $a,b\neq i$.
By \eqref{El} we also have, in $Y$:
$$E_i\cdot\Gamma_\ell=C_i\cdot\ell-1=\begin{cases}
1\quad\text{for }i=1,2,3\\
2\quad\text{for }i=4,\dotsc,8,
\end{cases}$$
so $(E_i)_{|P_\ell}$ is a line in $P_\ell\cong\pr^2$ for $i=1,2,3$, and a conic for 
$i=4,\dotsc,8$.

Since $E_1$ contains $P_{e_2}$ and $P_{e_3}$, it contains both $x_2$ and $x_3$; on the other hand these points are distinct, thus  $(E_1)_{|P_\ell}=\overline{x_2x_3}$, and $E_1$ does not contain other points of $P_\ell$ blown-up in $\w{P}_\ell$.
Similarly for $E_2$ and $E_3$. 

Since $E_4$ contains $P_{e_1},P_{e_2},P_{e_3}$, and $P_{\ell_{4b}}$ for $b=5,\dotsc,8$, $(E_4)_{|P_\ell}$ is a conic containing the $7$ points $x_1,x_2,x_3,y_{45},y_{46},y_{47},y_{48}$. Notice that 
the divisors $E_1,\dotsc,E_8$ are pairwise disjoint in $X_h$, so $(E_4)_{|P_\ell}$ and 
 $(E_i)_{|P_\ell}$ for $i=1,2,3$ can intersect only in the points $x_1,x_2,x_3$; this implies that  $(E_4)_{|P_\ell}$ is a smooth conic, and similarly for $(E_i)_{|P_\ell}$
when $i=5,\dotsc,8$.

Since for $i=1,2,3$ $(E_i)_{|P_\ell}$ is a line containing two points blown-up in $\w{P}_\ell\to P_\ell$, its transform  $(E_i)_{|\w{P}_\ell}$ is a $(-1)$-curve in the surface
$\w{P}_\ell$. This shows that around $p_i$, the map  $\w{P}_\ell\to  V_{h,\ell}$
factors as the contraction of $(E_i)_{|\w{P}_\ell}$ to a smooth point, followed by the normalization of $V_{h,\ell}$ at $p_i$. On the other hand,
the scheme-theoretical fiber of $p_i$ under the map $\w{P}_\ell\to  V_{h,\ell}$ is  $(E_i)_{|\w{P}_\ell}$, hence it is reduced; this shows that $p_i\in V_{h,\ell}$ is normal and hence smooth.

For $i=4,\dotsc,8$ $(E_i)_{|P_\ell}$ is a smooth conic containing $7$ points blown-up in $\w{P}_\ell\to P_\ell$. Thus its transform  $(E_i)_{|\w{P}_\ell}$ is a smooth rational curve  with self-intersection $4-7=-3$
in the surface
$\w{P}_\ell\cong\Bl_{13\text{pts}}\pr^2$. Similarly as before, this yields the statement on $p_i$ for $i=4,\dotsc,8$.

\smallskip

We prove $(d)$; set for simplicity $i=7$ and $j=8$.
By Lemmas \ref{sequence} and \ref{excplanes}, $P_\ell$ meets the indeterminacy locus of $\eta_h\colon Y\dasharrow\pr^4$ in $21$ isolated points:
$$x_i:=P_{e_i}\cap P_\ell\ \text{ for }i=1,\dotsc,6,\quad
y_{ab}:=P_{\ell_{ab}}\cap P_\ell\ \text{ for }a\leq 6, b\geq 7,\quad \{z^1,z^2,z^3\}:=P_{\ell_{78}}\cap P_\ell$$
(again, the components of  the indeterminacy locus of $\eta_h$ are pairwise disjoint in $Y$).
By the  description of the map $\xi_h$ in Lemma \ref{sequence}, we have a diagram:
{\footnotesize$$\xymatrix{{\w{P}_\ell\subset X_h}\ar[d]&{\widehat{P}_\ell\subset\widehat{X}_h}\ar[d]\ar[l]\\
{V_{h,\ell}\subset\pr^4}&{P_\ell\subset Y}
}$$}where $\widehat{P}_\ell\to P_\ell$ is  the blow-up of  $\pr^2$ in the $21$ points 
$x_i,y_{ab},z^j$, with  exceptional curves $\widehat{\Gamma}_i$,  $\widehat{L}_{ab}$,
and $\widehat{L}^j_{78}$ respectively, and   $\widehat{P}_\ell\to \w{P}_\ell$ 
is an isomorphism outside the curves $\widehat{L}^j_{78}$, while it glues the three curves $\widehat{L}^1_{78},\widehat{L}^2_{78},\widehat{L}^3_{78}$ onto $L_{78}\subset X_h$. Finally,
the morphism $\w{P}_{\ell}\to V_{h,\ell}$ is the restriction of $X_h\to\pr^4$ and contracts the curve $(E_i)_{|\w{P}_\ell}$ to $p_i$ for $i=1,\dotsc,8$. In particular, we see that $V_{h,\ell}\smallsetminus(\{p_1,\dotsc,p_6\}\cup\overline{p_7p_8})$ is smooth, and that  $V_{h,\ell}$ is singular along the line $\overline{p_7p_8}$, with a point of multiplicity $3$ at the general point of the line.

Let $\widehat{H}\in\Pic(\widehat{P}_\ell)$ be the pull-back of $H_{|\w{P}_\ell}$, and 
$\widehat{\Gamma}_\ell\subset \widehat{P}_\ell$ the transform of a general line in $P_\ell\cong\pr^2$. In $X_h$ we have $H\cdot L_{ab}=1$ for every $a<b$, $H\cdot\Gamma_i=4$, and 
$H\cdot\w{\Gamma}_\ell=11$ by Lemma \ref{Gamma}. 
Using the projection formula, in $\widehat{P}_\ell$ we get
$$\widehat{H}\sim 11\widehat{\Gamma}_\ell-4\sum_{i=1}^6\widehat{\Gamma}_i
-\sum_{a\leq6,b\geq 7}\widehat{L}_{ab}-\sum_{j=1}^3\widehat{L}^j_{78}$$
and hence the degree of $V_{h,\ell}$  in $\pr^4$ is $\widehat{H}^2=10$.

Let $i\in\{1,\dotsc,6\}$. Similarly to the proof of case $(c)$, we see that in $Y$ $(E_i)_{|P_\ell}$ is a smooth conic containing the $7$ points
$x_1,\dotsc,\check{x}_i\dotsc,x_6,y_{i7},y_{i8}$ and no other point blown-up in $\widehat{P}_\ell$.  Thus the transform of $(E_i)_{|P_\ell}$
in  $\widehat{P}_\ell$ is a smooth rational curve  with self-intersection $4-7=-3$, which yields 
the statement on $p_i$.

The proof of $(e)$ is very similar.
\end{proof}
\subsection{The Bertini involution in $X$  and in $\pr^4$}\label{secbertiniX}
\begin{proposition}\label{bertiniX}
Let $X$ be a blow-up of $\pr^4$ at $8$ general points. 
Then $X$ has a unique non-trivial pseudo-automorphism $\iota_X\colon X\dasharrow X$. 
It can be factored as follows:
{\footnotesize$$\xymatrix{X\ar@{-->}@/^1pc/[rrr]^{\iota_X}\ar@{-->}[r]_{\xi}&Y&{\widehat{X}_2}\ar[l]\ar[r]&X
},$$}
where
\begin{enumerate}[$\bullet$]
\item $\xi\colon X\dasharrow Y$ is described in Lemma~\ref{sequence};
\item $\widehat{X}_2\to Y$ is the blow-up of $36$ pairwise disjoint smooth rational surfaces in $Y$, given by the transforms in $Y$ of  $8$ surfaces 
of degree $10$ and $28$ surfaces of degree $15$ in $\pr^4$, all containing $p_1,\dotsc,p_8$;
 every exceptional divisor is isomorphic to $\pr^1\times\pr^2$ with normal bundle $\ol(-1,-1)$;
\item $\widehat{X}_2\to X$ contracts the exceptional divisors to pairwise disjoint smooth rational curves.
\end{enumerate}
\end{proposition}
\begin{proof}
Let $(S,h)$ be as in Cor.~\ref{SfromX}, so that $X\cong M_{\ma{B}_h}$. It follows from Th.~\ref{psautM} that $X$ has a unique non-trivial pseudo-automorphism $\iota_X$.
Under the isomorphism with $M_{\ma{B}_h}$, $\iota_X$ is the map induced by the Bertini involution $\iota_S$ of $S$ as follows:
$$
 M_{\ma{B}_h}\dasharrow M_{\ma{B}_h},\quad 
[F]\mapsto [\iota_S^*F].$$
This is nothing but the natural birational map $M_{\ma{B}_h}\dasharrow M_{\ma{B}_{\iota_S^*h}}$ (given by $[F]\mapsto[F]$, see Cor.~\ref{birational}) composed with the natural isomorphism $M_{\ma{B}_{\iota_S^*h}}\cong M_{\ma{B}_h}$ induced by $\iota_S$ (note that $\iota_S^*\ma{B}_h=\ma{B}_{\iota_S^*h}$). 
In particular, we can factor $\iota_X$ as a sequence of flips by varying the polarization from $\ma{B}_h$ to $\ma{B}_{\iota_S^*h}$ along the plane spanned by $h$ and $-K_S$ (see Fig.~\ref{figura}), and similarly for ${\pr^4}$:
{\footnotesize$$\xymatrix{
{X=X_h}\ar@{-->}[r]_(0.6){\xi_h}\ar@/^1pc/@{-->}[rr]^{\iota_X}&Y\ar@{-->}[r]_(0.35){\xi_{\iota_S^*h}^{-1}}&{X_{\iota_S^*h}\cong X}
}\qquad\qquad\qquad
\xymatrix{
{\pr^4}\ar@{-->}[r]_{\eta_h^{-1}}\ar@/^1pc/@{-->}[rr]^{\iota_{\pr^4}}
&Y
\ar@{-->}[r]_{\eta_{\iota_S^*h}}&{\pr^4}
}$$}

The factorization of $\xi_h$ is described in Lemma \ref{sequence}, so let us consider the second part $Y\dasharrow X$.
To go from the chamber $\ma{N}$ to the chamber $\ma{F}_{\iota_S^*h}$, we have to cross the $8$ walls $(2\ell_i+K_S)^{\perp}$, where $\ell_i\sim 6h-2e-e_i$, for $i=1,\dotsc,8$
(see Fig.~\ref{figura} and Lemma \ref{L_t}). Thus the map $Y\dasharrow M_{\ma{F}_{\iota_S^*h}}$ is the composition of $8$ flips, each replacing $P_{\ell_i}\cong\pr^2$ with a smooth rational curve. Moreover $P_{\ell_i}\subset Y$ is the transform of the surface $V_{h,\ell_i}\subset\pr^4$, described in  Th.~\ref{specialsurfaces}$(e)$.

Secondly, to go from the chamber $\ma{F}_{\iota_S^*h}$ to the chamber $\ma{B}_{\iota_S^*h}$, we have to cross the $28$ walls $(2\ell'_{ij}+K_S)^{\perp}$, where $\ell'_{ij}\sim 5h-2e+e_i+e_j$, for $1\leq i<j\leq 8$
(see Fig.~\ref{figura} and Lemma \ref{L_t}). Thus the map $M_{\ma{F}_{\iota_S^*h}}\dasharrow X$ is the composition of $28$ flips, each replacing $P_{\ell'_{ij}}\cong\pr^2$ with a smooth rational curve. Moreover $P_{\ell'_{ij}}\subset Y$ is the transform of the surface $V_{h,\ell'_{ij}}\subset\pr^4$, described in  Th.~\ref{specialsurfaces}$(d)$.
\end{proof}
\begin{corollary}\label{bertiniP^4}
Let $p_1,\dotsc,p_8\in\pr^4$ be general points, and let $V\subset|\ol_{\pr^4}(49)|$ be the linear system of hypersurfaces having multiplicity at least $30$ at $p_1,\dotsc,p_8$. Then $\dim V=4$, and $V$ defines a birational involution $\iota_{\pr^4}\colon\pr^4\dasharrow\pr^4$. The base locus of
$V$ has dimension $2$, and it is the union of $36$ irreducible rational surfaces, $28$ of degree $10$, and $8$ of degree $15$. The birational map $\iota_{\pr^4}$ contracts $8$ irreducible rational hypersurfaces of degree $10$.
\end{corollary}
\begin{proof}
Most of the statement is a direct consequence of Prop.~\ref{bertiniX}. 
The divisors contracted by $\iota_{\pr^4}$ are the transforms of $E_{\iota_S^*C_i}\subset X$, for $i=1,\dotsc,8$. We have $C_i\sim h-e_i$, $\iota_S^*C_i\sim-4K_S-C_i\sim 11h-4e+e_i$ (see \eqref{w_0}), and thus 
$E_{\iota_S^*C_i}\sim 10H-6\sum_jE_j+E_i$ after Prop.~\ref{fixeddiv}.
\end{proof}

\providecommand{\noop}[1]{}
\providecommand{\bysame}{\leavevmode\hbox to3em{\hrulefill}\thinspace}
\providecommand{\MR}{\relax\ifhmode\unskip\space\fi MR }
\providecommand{\MRhref}[2]{%
  \href{http://www.ams.org/mathscinet-getitem?mr=#1}{#2}
}
\providecommand{\href}[2]{#2}

\end{document}